\documentclass{amsart}

\usepackage{amssymb,latexsym,amscd}
\usepackage{cyrillic}
\usepackage{euscript}           
\usepackage{enumerate,calc,graphicx,array}
\usepackage[matrix,arrow,curve,frame]{xy}
\numberwithin{equation}{section}

\usepackage{caption,subcaption}
\usepackage{tikz}
\usetikzlibrary{matrix,calc,arrows,decorations.markings,patterns,calc}
\tikzset{->-/.style={decoration={
  markings,
  mark=at position .5 with {\arrow{>}}},postaction={decorate}}}

\newtheorem{theorem}{Theorem}[section]

\newtheorem{proposition}[theorem]{Proposition}
\newtheorem{conjecture}[theorem]{Conjecture}
\newtheorem{corollary}[theorem]{Corollary}

\newtheorem{lemma}[theorem]{Lemma}

\theoremstyle{definition}

\newtheorem{remark}[theorem]{Remark}
\newtheorem{example}[theorem]{Example}
\newtheorem{definition}[theorem]{Definition}

\def\endproof{\hfill$\square$\medskip}

\def\AA{\mathcal{A}}

\def\UU{\mathbb{U}}

\def\FF{\mathcal{F}}
\def\ZZ{\mathbb{Z}}
\def\QQ{\mathbb{Q}}
\def\CC{\mathbb{C}}

\def\cU{\mathcal{U}}

\def\ii{\mathbf{i}}
\def\jj{\mathbf{j}}

\def\TT{\mathbb{T}}
\def\YY{\mathbb{Y}}

\newcommand{\cyrrm}{\fontencoding{OT2}\selectfont\textcyrup}

\def\sgn{\operatorname{sgn}}

\def\FF{\mathcal{F}}

\begin{document}


\title
{Noncommutative marked surfaces}

\author{Arkady Berenstein}
\address{\noindent Department of Mathematics, University of Oregon,
Eugene, OR 97403, USA} \email{arkadiy@math.uoregon.edu}

\author{Vladimir Retakh}
\address{Department of Mathematics, Rutgers University, Piscataway, NJ 08854, USA}
\email{vretakh@math.rutgers.edu}

\begin{abstract} The aim of the paper is to attach a noncommutative cluster-like structure to each marked surface $\Sigma$. This is a noncommutative algebra $\AA_\Sigma$
generated by ``noncommutative geodesics" between marked points subject to certain triangle relations and noncommutative analogues of Ptolemy-Pl\"ucker relations.
It turns out that the algebra $\AA_\Sigma$ exhibits a noncommutative Laurent Phenomenon with respect to any triangulation of $\Sigma$, which confirms its ``cluster nature".
As a surprising byproduct, we obtain a new topological invariant of $\Sigma$, which is a free or a 1-relator group easily computable in terms of any triangulation of $\Sigma$.
Another application is the proof of Laurentness and positivity of certain discrete noncommutative integrable systems.
\end{abstract}

\dedicatory{To the memory of Andrei Zelevinsky
\\ \cyrrm{Svetlo\u{i} pamyati Andreya Vladlenovicha Zelevinskogo posvyashchae{t}sya}
}

\maketitle

\tableofcontents


\pagebreak[3]

\section{Introduction}

The goal of the paper is to introduce and study noncommutative algebras attached to surfaces (with marked boundary points and punctures) and their triangulations.
This provides an instance of the noncommutative cluster theory (which is the main theme of the forthcoming paper  \cite{BR3}).

Since each surface can be obtained by gluing edges of a polygon (actually, in many ways),
the most important object of study are {\it noncommutative polygons} and their {\it noncommutative triangulations}.

In the commutative case, cluster structure (of type $A_{n-3}$) on an $n$-gon is based on the  {\it Ptolemy relations}:
\begin{equation}
\label{eq:ptolemey}
x_{ik}x_{j\ell}=x_{ij}x_{k\ell}+x_{i\ell}x_{jk}
\end{equation}
for all  quadrilaterals $(i,j,k,\ell)$ inscribed in a circle, $1\le i,j,k,\ell\le n$, so that the chords $(i,k)$ and $(j,\ell)$ are diagonals of the quadrilateral, and
$x_{ij}=x_{ji}$, $i\ne j$ is the Euclidean length of the chord $(ij)$. The Ptolemy relations \eqref{eq:ptolemey} can also be interpreted as  Pl\"ucker identities for $2\times n$ matrices.

In the  noncommutative version we do not assume that $x_{ij}=x_{ji}$ and we think of $x_{ij}$ as a measurement of a directed chord from $i$
to $j$. We suggest
the following noncommutative generalization of the Ptolemy identity based on the theory of noncommutative quasi-Pl\"ucker coordinates
developed in \cite{gr3}:
\begin{equation}
\label{eq:noncomptolemey}
x_{ik}x_{jk}^{-1}x_{j\ell}=x_{i\ell}+x_{ij}x_{kj}^{-1}x_{k\ell}.
\end{equation}
for every quadrilateral $(i,j,k,\ell)$, in which $(i,k)$ and $(j,\ell)$ are the diagonals.

Note that since elements $x_{ij}$ correspond to directed arrows,
the products of the form $x_{ij}x_{k\ell}^{-1}$, $x_{\ell k}^{-1}x_{ji}$ make sense only when  $\ell=j$.

It turns out that in order to establish the noncommutative Laurent Phenomenon and thus obtain a noncommutative cluster structure on the $n$-gon,
it is crucial to impose  additional
{\it triangle relations} (also suggested by properties of quasi-Pl\"ucker coordinates):
\begin{equation}
\label{eq:noncomtriangular}
x_{ij}x_{kj}^{-1}x_{ki}=x_{ik}x_{jk}^{-1}x_{ji}
\end{equation}
for all distinct $i,j,k$ (of course, \eqref{eq:noncomtriangular}
is redundant in the commutative case).

The triangle relations \eqref{eq:noncomtriangular} are of fundamental importance because they allow to introduce {\it noncommutative angles} $T_i^{j,k}:=x_{ji}^{-1}x_{jk}x_{ik}^{-1}$ in each triangle $(i,j,k)$ so that $T_i^{j,k}=T_i^{k,j}$ due to \eqref{eq:noncomtriangular}. That is,  the noncommutative angle at a vertex of a triangle does not depend on the order of the remaining two vertices. The "commutative" angles were introduced by Penner in \cite[Section 3]{Penner} (where they were called ``$h$-lengths") and each $x_{ij}=x_{ji}$ was viewed as the $\lambda$-length of the side $(i,j)$ of an ideal  triangle $(i,j,k)$ (see also \cite[Lemma 7.9]{FTh}, \cite[Section 12]{FG}, and \cite[Section 1.2]{GL}, in the latter paper the term ``angle" was used, apparently, for the first time) and thus noncommutative angles together with the "noncommutative $\lambda$-lengths" $x_{ij}$ can be thought of as  a totally noncommutative  metric on the Lobachevsky plane. The term ``angle" is justified by the following observation. The noncommutative Ptolemy relations \eqref{eq:noncomptolemey} together with the triangle relations \eqref{eq:noncomtriangular}  are equivalent to:
$$T_j^{ik}=T_j^{i\ell}+ T_j^{k\ell}$$
for every quadrilateral $(i,j,k,\ell)$, in which $(i,k)$ and $(j,\ell)$ are the diagonals. In other words, the (both commutative and noncommutative)  angles are additive, which justifies the name. Using  additivity of noncommutative angles, we establish the first instance of the noncommutative Laurent Phenomenon for the $n$-gon with vertices $1,\ldots,n$:
$$x_{ij}=\sum_{k=i}^{j-1} x_{i,1} T_1^{k,k+1}x_{1,j}$$
for all $2\le i<j\le n-1$, e.g., each $x_{ij}$ is a noncommutative Laurent polynomial in $x_{1,k}$, $x_{k,1}$, $k=2,\ldots,n-1$ and all $x_{i,i\pm 1}$. In fact, the latter elements correspond to a triangulation of the $n$-gon where each triangle has a vertex at $1$. We  generalize this to any triangulation of the $n$-gon  in Theorem \ref{th:noncomlaurent n-gon}, and, as expected, the commutative ``limit" of this result (with all $x_{ij}=x_{ji}$) specializes to the Schiffler formula (\cite[Theorem 1.2]{schiffler}).

These arguments extend verbatim if we replace a polygon with a surface $\Sigma$ with marked points.
That is, for each such $\Sigma$ one defines a $\ZZ$-algebra $\AA_\Sigma$ generated by $x_\gamma^{\pm 1}$, where $\gamma$ runs over homotopy classes of curves on $\Sigma$ between marked points subject to the triangle and noncommutative Ptolemy  relations. The Noncommutative Laurent Phenomenon (Theorem  \ref{th:noncomlaurent surface}) asserts that for a given triangulation $\Delta$ of $\Sigma$ each $x_\gamma$ belongs to the subalgebra generated by all $x_{\gamma'}^{\pm 1}$, $\gamma'\in \Delta$.
In any case, the assignments $\Sigma\mapsto \AA_\Sigma$ and $\Sigma\mapsto \TT_\Sigma$ define functors from the category
of surfaces with marked points to respectively the category of algebras and the category of groups (Theorem \ref{th:functorial triangular}).

A surprising byproduct of our approach is that the corresponding {\it triangle group} $\TT_\Delta$ (generated by all $t_\gamma$, $\gamma\in \Delta$ subject to the triangle relations) does not depend on the triangulation of $\Sigma$, therefore, is a topological invariant of
$\Sigma$ (Theorem \ref{th:TDeltaDelta' surface}). Moreover, each $\TT_\Delta$ is either free or a one-relator group  which looks like the fundamental group of $\Sigma$, however it is different from $\pi_1(\Sigma)$. For instance, if $\Sigma_n$ is the sphere with $n$ punctures, then $\TT_\Delta$ is a free group in
$5$ generators if $n=3$ and it is a  $1$-relator torsion-free  group in $4n-7$ generators if $n\ge 4$.\footnote{Misha Kapovich explained to us that $\TT_\Delta$ is related to the fundamental group of a ramified two-fold cover of $\Sigma$.}
 It turns out that each group $\TT_\Delta$ has a ``universal cover"
$\TT_\Sigma$  which is a group generated by  $t_\gamma$, as $\gamma$ runs over all isotopy classes of directed curves on $\Sigma$ between marked point, subject to the triangle relations (see Sections \ref{subsec:Big noncommutative polygons} and
 \ref{subsec:triangular groups} for details).
This group, which we refer to as {\it big triangle group} is of interest as well: if $\Sigma$ is the $n$-gon, we prove (Proposition \ref{pr:big triangular presentation}) that $\TT_\Sigma$  has a presentation with $\frac{(n-1)(n+2)}{2}$
generators and $(n-3)^2$ relations and expect
that the multiplicative group of $\AA_\Sigma$ is  isomorphic to $\TT_\Sigma$.

For each marked point $i$ on $\Sigma$ and each triangulation $\Delta$ we also introduce a {\it total (noncommutative) angle} $T_i^\Delta\in \AA_\Sigma$ in Section \ref{subsec:regular surface} to be the sum of noncommutative angles of all adjacent triangles.  Similarly to the commutative case, we establish (Theorem \ref{th:total anglle surface}) that the total angles do not depend on the choice of a triangulation $\Delta$. Thus the collection of the total angles $\{T_i\}$ can be thought of as a noncommutative version of a (hyperbolic) Riemann structure on $\Sigma$.
Using them we define in Section \ref{subsec:regular surface} the algebra $\cU_\Sigma$ to be the subalgebra  of $\AA_\Sigma$ generated by all noncommutative edges $x_\gamma$, the inverses of the boundary edges and all noncommutative angles $T_i$ and argue that $\cU_\Sigma$ is a totally noncommutative analogue of the {\it upper cluster algebra} corresponding to $\Sigma$ (see e.g., \cite{BFZ-cluster}).

As an application of our noncommutative Laurent phenomenon, taking $\Sigma$ to be a cylinder with no punctures,
one marked point on the inner boundary and $k$ marked points on the outer boundary, we  prove Laurentness of the
following noncommutative recursion for each $k\in 1+2\ZZ_{>0}$:

\begin{equation}
\label{eq:exchange relations sigma1k U}
\begin{cases} U_{n-k}DU_n=C_n+ U_{n-1}\overline D U_{n+1-k} & \text{if $n$ is even}\\
U_n\overline D U_{n-k}=C_n+ U_{n+1-k} D U_{n-1} &\text {if $n$ is odd} \\
\end{cases}
\end{equation}
for all $n\ge k+1$, where $D,\overline D$, and $C_i$, $i\in \ZZ_{> 0}$ belong to a noncommutative ground ring so that $C_{n+k-1}=C_{k-1}$ for $n\in \ZZ_{>0}$.

We prove (Theorem \ref{th:discrete Kontsevich}) that for odd $k>0$ this recursion has a unique solution in the group algebra $\QQ F_{2k+1}$ of the free group $F_{2k+1}$ freely generated by
$D,\overline D,C_1,\ldots,C_{k-1}$, $U_1,\ldots, U_k$, more precisely, each $U_n$  is a sum of elements of $F_{2k+1}$.
We also prove (Theorem \ref{th:discrete Kontsevich}) that the element $H_n$ in the skew field of fractions of $\QQ F_{2k+1}$, $n\ge k$,
given by
\begin{equation}
\label{eq:Hn}
H_n:=  \begin{cases}
\overline D U_{n+1-k}U_n^{-1}+D U_{n+k-1}U_n^{-1}  & \text{if $n$ is even}\\
U_n^{-1} U_{n+1-k}D+ U_n^{-1}U_{n+k-1}\overline D & \text{if $n$ is odd}\\
\end{cases}
\end{equation}
belongs to $\ZZ F_{2k+1}$ and does not depend on $n$ hence is a ``noncommutative conserved quantity."




Setting $D=\overline D=C_i=1$ for all $i>0$, we recover the Laurentness of the noncommutative discrete dynamical system established by  Di Francesco and Kedem in \cite[Theorem 6.2]{KDF}
(conjectured by M. Kontsevich in \cite[Section 3]{kontsevich}).

We finish the introduction with establishing Laurentness of the following noncommutative recursion (which specializes to the discrete integrable system recently studied by P. Di Francesco in \cite{DF}, see Section \ref{sect:kontsevich} for details) in the skew field $\FF$ freely generated by
 $A,\overline A_i$, $B_i,\overline B_i$, $U_{i,i}$, $V_{i,i}$,  $U_{i,i+1}$, $i\in \ZZ$:
\begin{equation}
\label{eq:difrancesco1}
U_{i+1,j}A_j V_{j+1,i}=B^{-1}_{i+1}+U_{i+1,j+1}\overline A_j V_{ji},~V_{i+1,j}B_j U_{j+1,i}=A^{-1}_{i+1}+V_{i+1,j+1}\overline B_j U_{ji}\ ,
\end{equation}
\begin{equation}
\label{eq:difrancesco2}
U_{ij}A_j V_{j+1,i}= U_{i,j+1}\overline A_j V_{ij},~V_{ij}B_j U_{j+1,i}=V_{i,j+1}\overline B_j U_{ij} \ .
\end{equation}
We prove (Theorem \ref{th:infinite strip})
that this recursion has a (unique) solution in the group algebra $\QQ T_\infty$ of the free group  $\TT_\infty$ freely generated by $A_i,\overline A_i$, $B_i,\overline B_i$, $U_{i,i}$, $V_{i,i}$,  $U_{i,i+1}$, $i\in \ZZ$, more precisely, each $U_{ij}$ and $V_{ij}$  is a sum of elements of the group. We also prove (Theorem \ref{th:infinite strip}) that  the elements
$H_{ij}^\pm\in Frac(\ZZ \TT_\infty)$, $i\in \ZZ$, given by
\begin{equation}
\label{eq:Hnprimeinfty}
H_{ij}^+:=U_{ji}^{-1}(U_{j,i-1}A_{i-1} + U_{j,i+1}\overline A_i),~
H_{ij}^-:=V_{ji}^{-1}(V_{j,i-1}B_{i-1}+V_{j,i+1}\overline B_i^{-1})
\end{equation}
belong to $\ZZ \TT_\infty$ and do not depend on $j$.


These examples and their treatment in Section \ref{sect:kontsevich} suggest the following general approach to constructing noncommutative discrete integrable systems. That is, such a system consists of a marked surface $\Sigma$, its automorphism $\tau:\Sigma:\to \Sigma$ permuting marked points, and a triangulation $\Delta$  so that the collection ${\mathcal T}=\{x_\gamma\in \AA_\Sigma,\gamma\in \cup_{k\in \ZZ} \tau^k(\Delta)\}$ evolves in ``discrete time" $k\in \ZZ$ and for each marked point $p$ of $\Sigma$, the total noncommutative angle $T_p$ is a (noncommutative) conserved quantity. The noncommutative Laurent Phenomenon (Theorems \ref{th:noncomlaurent surface} and \ref{th:A_Delta sigma}) then guarantees that each ${\mathcal T}$ belongs to the algebra isomorphic to the group algebra of $\TT_\Delta$.

In Appendix we collect relevant results on noncommutative localizations.

\bigskip

{\bf Acknowledgments}. This work was partly done during our visits to Mathematisches Forschungsinstitut Oberwolfach, Max-Planck-Institut f\"ur Mathematik,
Institut des Hautes \'Etudes Scientifiques, and Centre de Recerca Matem\`{a}tica, Barcelona. We gratefully acknowledge the support of these institutions.
We are very grateful to Alexander Goncharov
and especially to Maxim Kontsevich for their encouragement and support. Thanks are due to George Bergman, Dolors Herbera, and Alexander Lichtman for stimulating discussions of noncommutative localizations and unique factorizations, and to Misha Kapovich, Feng Luo, and Misha Shapiro for explaining important aspects of low dimensional topology and hyperbolic geometry.

\section{Noncommutative polygons}

\label{sect:Noncommutative polygons}

\subsection{Definition and main results}
\label{subsect:Definition and basic results}
For each $n\ge 3$ consider a cyclic order $i\mapsto i^+$ on $[n]=\{1,2,\ldots,n\}$ by
$$i^+=\begin{cases}
i+1 & \text{if $i<n$}\\
1   & \text{if $i=n$}
\end{cases}$$
(and $i\mapsto i^-$ to be the inverse of $i\mapsto i^+$).  We will view $[n]$ with this cyclic order
as $n$ points on a circle (or vertices of a convex $n$-gon) and each pair $(i,j)$ as a chord from $i$ to $j$ (or as an edge or diagonal of the $n$-gon).

We also say that a sequence $\ii=(i_1,\ldots,i_\ell)$ of distinct elements in $[n]$ is {\it cyclic} if a cyclic permutation $\ii\mapsto (i_k,\ldots,i_\ell,i_1,\ldots,i_{k-1})$ is strictly increasing. In particular, the sequence $(k,k+1,\ldots,n,1,\ldots,k-1)$ is cyclic for each $k$.

\begin{definition}
\label{def:An}

Denote by ${\mathcal A}_n$ the $\QQ$-algebra generated by $x_{ij}$ and $x_{ij}^{-1}$, $i,j\in [n]$, $i\ne j$ subject to the relations:

(i) (Triangle relations) For any distinct indices $i,j,k\in [n]$:
\begin{equation}
\label{eq:triangle relations}
x_{ij}x_{kj}^{-1}x_{ki}=x_{ik}x_{jk}^{-1}x_{ji} \ .
\end{equation}

(ii) (Exchange relations) For any cyclic $(i,j,k,\ell)$ in $[n]$:
\begin{equation}
\label{eq:exchange relations}
x_{j\ell}=x_{jk}x_{ik}^{-1}x_{i\ell}+ x_{ji}x_{ki}^{-1}x_{k\ell} \ .
\end{equation}

\end{definition}

\bigskip

\begin{picture}(110,110)
\put(0,10){\includegraphics[scale=0.4]{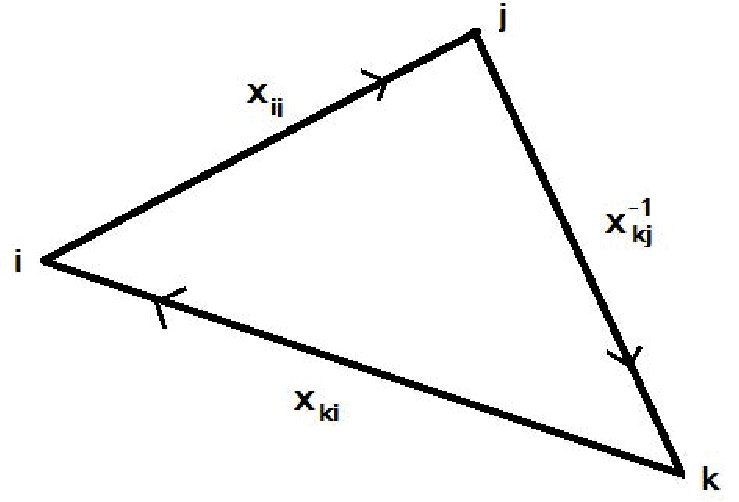}}
\put(135,70){{\bf=}}
\put(150,10){\includegraphics[scale=0.4]{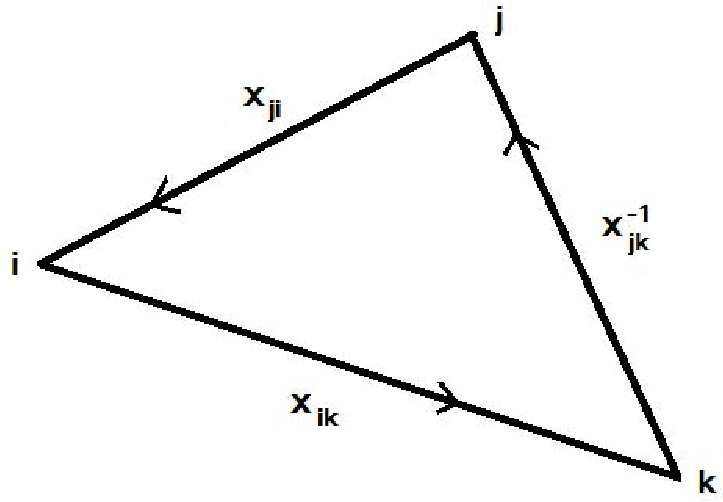}}
\put(300,0){\includegraphics[scale=0.4]{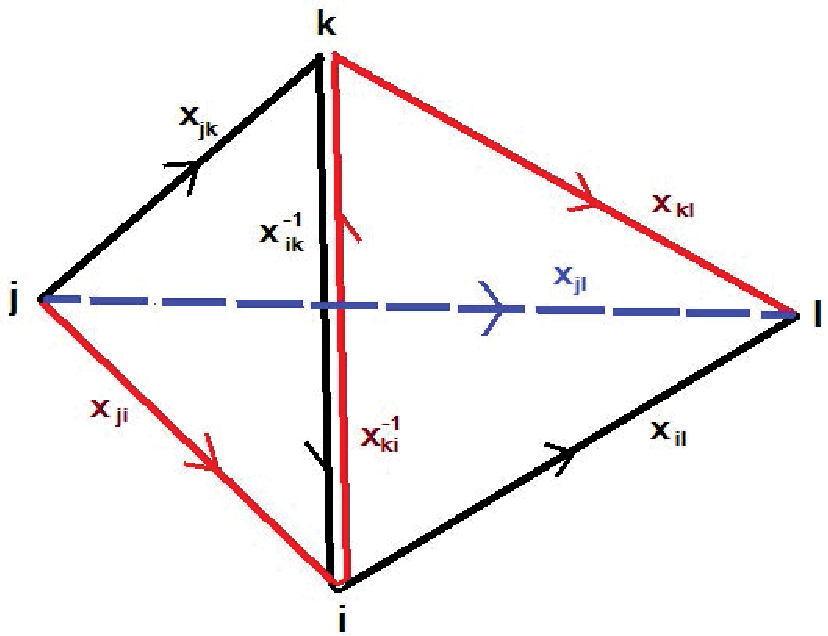}}
\put(150,0){Triangle and exchange relations}

\end{picture}

\begin{remark} One can show that the exchange relations \eqref{eq:exchange relations} are equivalent to noncommutative
Ptolemy relations \eqref{eq:noncomptolemey} provided the triangle relations \eqref{eq:triangle relations} hold. Namely, multiplying \eqref{eq:exchange relations}
by $x_{ik}x_{jk}^{-1}$ on the left and using the triangle relation \eqref{eq:triangle relations}, we obtain \eqref{eq:noncomptolemey}:
$$
x_{ik}x_{jk}^{-1}x_{j\ell}=x_{i\ell}+x_{ik}x_{jk}^{-1}x_{ji}x_{ki}^{-1}x_{k\ell}=x_{i\ell}+x_{ij}x_{kj}^{-1}x_{ki}x_{ki}^{-1}x_{k\ell}=x_{i\ell}+x_{ij}x_{kj}^{-1}x_{k\ell}
$$
Conversely, multiplying \eqref{eq:noncomptolemey} on the left by $x_{jk}x_{ik}^{-1}$ and using \eqref{eq:triangle relations}, we recover \eqref{eq:exchange relations}.
\end{remark}

At the first glance the number of relations of ${\mathcal A}_n$ greatly exceeds the number of generators, moreover, we expect that the subalgebra  of ${\mathcal A}_n$ generated by all $x_{ij}$ is a free algebra in $n^2-n$ generators.

However, we will demonstrate below that the algebra ${\mathcal A}_n$ is ``rationally" generated only by  $3n-4$ free generators.



Denote by $F_m$ the free group on $m$ generators, so that its group algebra $\QQ F_m$  is the free Laurent polynomial algebra $\QQ\langle c_1^{\pm 1},\ldots,c_{m}^{\pm 1}\rangle$. Following Amitsur and Cohn (see e.g., \cite{Cohn} or Section \ref{sect:appendix} below) denote by $\FF_m$ the free skew field on $m$ generators, which is the ``largest" skew field of fractions of $\QQ F_m$. The following is our first main result, in which we freely use notation of Section \ref{sect:appendix}.

\begin{theorem}
\label{th:A_n divisible} For each $n\ge 2$ the algebra $\AA_n$ contains a subalgebra $\AA'_n$ isomorphic to the free group algebra $\QQ F_{3n-4}$ so that
 $\AA_n$ is a universal localization of $\AA'_n$ by a certain multiplicative submonoid of $A_n'\setminus\{0\}$.

\end{theorem}

We prove the theorem in Section \ref{subsect:A_Delta}. In fact, it will follow from a more precise assertion (Theorem \ref{th:A_Delta}).

In view of universality of the localization (Lemma \ref{le:universal property of localization}), Theorem \ref{th:A_n divisible} implies the following immediate corollary.

\begin{corollary} The canonical monomorphism of algebras $\varphi':\AA_n'\hookrightarrow \FF_{3n-4}$ uniquely extends to a homomorphism of algebras
\begin{equation}
\label{eq:phi from A'}
\varphi:\AA_n\to \FF_{3n-4}
\end{equation}
\end{corollary}


In fact, we expect that \eqref{eq:phi from A'} is injective, so far we can deduce this from another, ``innocent looking" conjectural property of the group algebras $\QQ F_m$ (Conjecture \ref{conj:rigid linear}, see also  Section \ref{subsec:rigidity=>injectivity}).

\begin{remark}
Injectivity of \eqref{eq:phi from A'} would imply, in particular, that $\AA_n$ has no zero divisors, which is a rather non-trivial assertion because of the following ``counter-example" which was communicated to us by George Bergman. The universal localization $\QQ\langle x,y\rangle[(xy)^{-1}]$ of the free algebra $\QQ\langle x,y\rangle$ has a zero-divisor $y(xy)^{-1} x-1$.

\end{remark}

\begin{remark}
\label{re:injective A to A'} Given $n'\ge n$ and an injective map ${\bf j}:[n]\hookrightarrow [n']$ for some $n'>n$, clearly, the assignments $x_{ij}\mapsto x_{{\bf j}(i),{\bf j}(j)}$ define a homomorphism of algebras ${\bf j}_\star:\AA_n\to \AA_{n'}$. One can conjecture that each ${\bf j}_\star$ is injective. In fact, this would directly follow from the injectivity of each \eqref{eq:phi from A'}.
\end{remark}

Now we explore the ``cluster" structure of ${\mathcal A}_n$.
We say that a pair $(i,k)$ {\it crosses} $(j,\ell)$  if $(i,j,k,\ell)$ is cyclic.

A {\it triangulation} $\Delta$ of $[n]$  is a maximal crossing-free subset of $[n]\times [n]\setminus \{(i,i)|i\in [n]\}$. Clearly, each triangulation of $[n]$ has cardinality $4n-6$.

For each triangulation $\Delta$ of $[n]$ define:

$\bullet$   The subalgebra ${\mathcal A}_\Delta$  of ${\mathcal A}_n$ generated by $x_{ij}$, $i,j\in [n]$, $i\ne j$ and $x_{ij}^{-1}$, $(i,j)\in \Delta$.

$\bullet$ The {\it triangle} group $\TT_\Delta$ generated by all $t_{ij}$, $(i,j)\in \Delta$ subject to the relations:
$$t_{ij}t_{kj}^{-1}t_{ki}=t_{ik}t_{jk}^{-1}t_{ji}$$
for all $i,j,k\in [n]$ such that $(i,j),(j,k),(k,i)\in \Delta$.

The term ``triangle group" is normally used for a group generated by reflections about sides of a triangle. In this paper we are using it in a somewhat similar way: a group generated by ``side lengths'' of noncommutative triangles.

\begin{theorem}
\label{th:T_Delta polygon}
Each  $\TT_\Delta$ is a free group in $3n-4$ generators.

\end{theorem}

We prove Theorem \ref{th:T_Delta polygon} in Section \ref{subsec:proof TDelta is free}. We generalize it in Theorem \ref{th:TDeltaDelta' surface} to all surfaces.


Clearly,
the assignments $t_{ij}\mapsto x_{ij}$,  $(i,j)\in \Delta$ define a homomorphism of algebras:
\begin{equation}
\label{eq:i_Delta}
\ii_\Delta:\QQ\TT_\Delta \to  {\mathcal A}_\Delta \ ,
\end{equation}
where $\QQ\TT_\Delta$ is the group algebra of $\TT_\Delta$.



	Recall (see, e.g., \eqref{eq:canonical embedding localization}) that for a given algebra $\AA$ with no zero divisors and a submonoid $S\subset \AA\setminus \{0\}$ one has a universal localization $\AA[S^{-1}]$ of $\AA$ by $S$.

\begin{theorem}
\label{th:A_Delta}
For each triangulation $\Delta$ of $[n]$ one has:

(a) The homomorphism $\ii_\Delta$ given by \eqref{eq:i_Delta} is an isomorphism of algebras.

(b)  $\AA_n=\AA_\Delta[{\bf S}^{-1}]$, where ${\bf S}$ is the multiplicative submonoid of $\AA_\Delta$ generated by all $x_{ij}$.

\end{theorem}

We will prove Theorem \ref{th:A_Delta} in Section \ref{subsect:A_Delta}.  In fact, Theorem \ref{th:A_Delta}(a)  establishes a {\it noncommutative cluster structure} on ${\mathcal A}_n$ and Theorem \ref{th:A_Delta}(b) -- a {\it noncommutative Laurent Phenomenon} (see also Section \ref{subsect:Noncommutative Laurent Phenomenon}).


\subsection{Noncommutative Laurent Phenomenon}
\label{subsect:Noncommutative Laurent Phenomenon}
For each even sequence $\ii=(i_1,\ldots,i_{2m})\in [n]^{2m}$ such that adjacent indices are distinct define the monomial $x_\ii\in {\mathcal A}_n$ by:
$$x_\ii:=x_{i_1,i_2}x_{i_3,i_2}^{-1}x_{i_3,i_4}\cdots x_{i_{2m-1},i_{2m-2}}^{-1}x_{i_{2m-1},i_{2m}} \ .$$

\begin{definition}
\label{def:admissible}
For a directed chord $(i,j)$, $i,j\in [n]$, $i\ne j$ and a triangulation $\Delta$ of $[n]$, we say that a
sequence $\ii=(i_1,\ldots,i_{2m})\in [n]^{2m}$ is $(i,j,\Delta)$-{\it admissible} if:

\noindent (i) $i_1=i$, $i_{2m}=j$ and $(i_s,i_{s+1})\in \Delta$ for $s=1,\ldots,2m-1$;

\noindent (ii) each chord $(i_{2s},i_{2s+1})$, $s=1,\ldots,m-1$ intersects $(i,j)$;

\noindent (iii) If ${\bf p}:=(i_k,i_{k+1})\cap (i,j)\ne \emptyset$ and ${\bf q}:=(i_{\ell},i_{\ell+1})\cap (i,j)\ne \emptyset$ for some $k<\ell$, then the point ${\bf p}$ of $(i,j)$ is closer in the path to $i$ than the point ${\bf q}$.

We denote by
$Adm_\Delta(i,j)$ the set of all $(i,j,\Delta)$-admissible sequences $\ii$.

\end{definition}

\begin{theorem} (Noncommutative Laurent Phenomenon)
\label{th:noncomlaurent n-gon}
Let $\Delta$ be a triangulation of  $[n]$. Then
\begin{equation}
\label{eq:noncomlaurent n-gon}
x_{ij}=\sum_{\ii\in Adm_\Delta(i,j)} x_\ii \ .
\end{equation}
for all $i,j\in [n]$, $i\ne j$.

\end{theorem}

We prove Theorem \ref{th:noncomlaurent n-gon} in Section \ref{subsect:proof of Theorems th:noncomlaurent n-gon and th:noncompol n-gon}.

\begin{remark} This is a noncommutative generalization of Schiffler's formula (\cite{schiffler}).

\end{remark}

Now we illustrate Theorem \ref{th:noncomlaurent n-gon} for each {\it starlike} triangulation
\begin{equation}
\label{eq:starlike}
\Delta_i=\{(i,j),(j,i)|j\in [n]\setminus \{i\}\}\cup\{(k,k^\pm),k\in [n]\},\quad i\in [n] \ .
\end{equation}

\begin{example} Fix $i\in [n]$. Then for each $k,\ell\in [n]\setminus \{i\}$ such that $(i,k,\ell)$ is cyclic, the following relation holds in ${\mathcal A}_n$:
$$x_{k\ell}=\sum_s x_{ki}x_{si}^{-1}x_{s,s^+}x_{i,s^+}^{-1}x_{i\ell} \, $$
where summation is over all $s=k,k^+,\ldots,\ell^-$ in cyclic order. Hence
$x_{k\ell}= \ii_{\Delta_i}(\sum\limits_s t_{ki}t_{si}^{-1}t_{s,s^+}t_{i,s^+}^{-1}t_{i\ell})$.


\end{example}

\begin{example}
\label{ex:Schiffler x n=5,6}
 (a) If $n=5$ and $\Delta=\{(1,3), (3,1),(1,4), (4,1);(i,i^\pm)|i\in [5]\}$, then
$$x_{25}=x_{21}x_{41}^{-1}x_{45}+x_{23}x_{13}^{-1}x_{15}+x_{21}x_{31}^{-1}x_{34}x_{14}^{-1}x_{15}.$$

(b) If $n=6$ and $\Delta=\{(1,3),(3,1),(3,6),(6,3),(4,6),(6,4);(i,i^\pm)|i\in [6]\}$, then
$$
x_{25}=x_{23}x_{63}^{-1}x_{65}+x_{21}x_{31}^{-1}x_{36}x_{46}^{-1}x_{45}+
x_{21}x_{31}^{-1}x_{34}x_{64}^{-1}x_{65}+x_{23}x_{13}^{-1}x_{16}x_{46}^{-1}x_{45}
+x_{23}x_{13}^{-1}x_{16}x_{36}^{-1}x_{34}x_{64}^{-1}x_{65}.
$$

\end{example}

\begin{picture}(200,110)
\put(60,20){\includegraphics[scale=0.4]{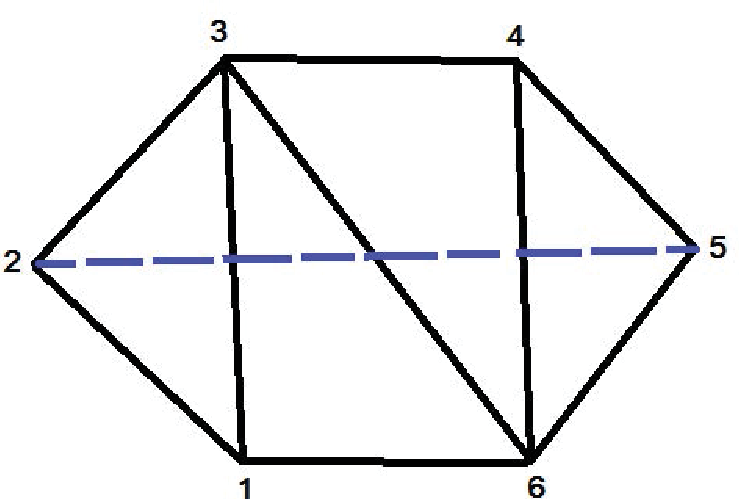}}
\put(220,10){\includegraphics[scale=0.3]{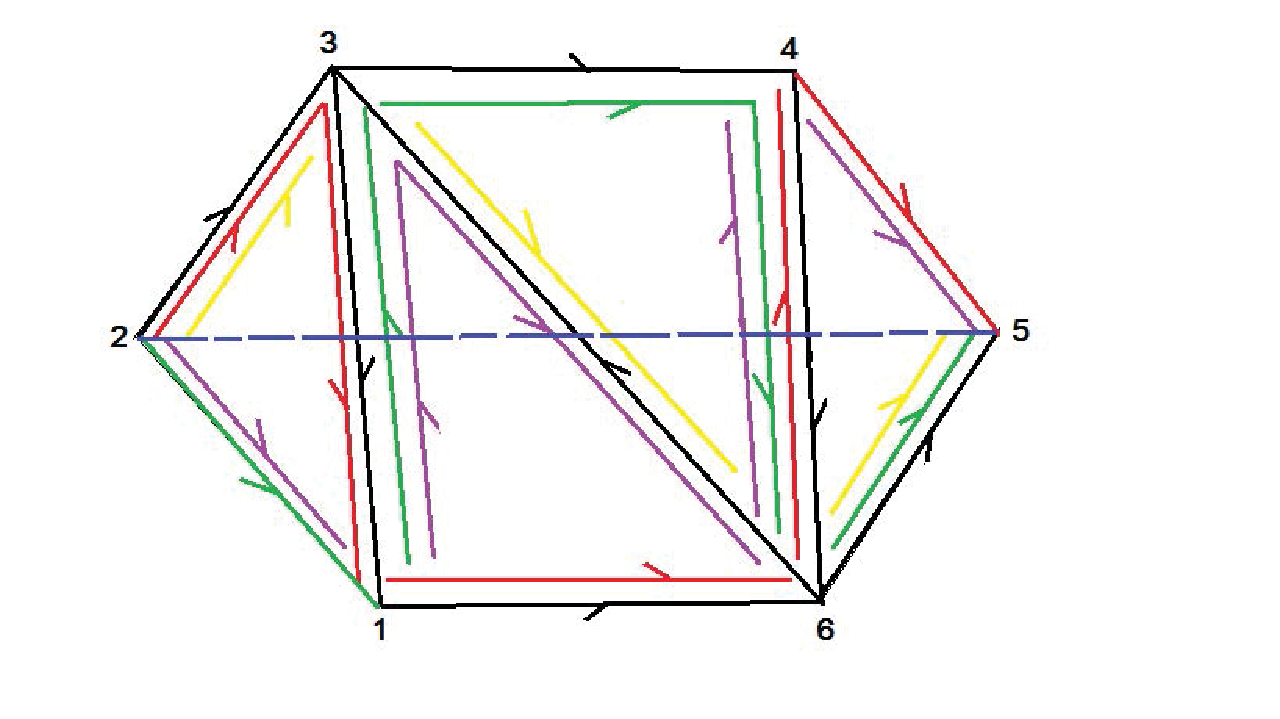}}
\put(100,0){A triangulation of a hexagon and all $(2,5)$-admissible sequences}
\end{picture}

\medskip

In fact, we will streamline the formula for $x_{ij}$ by introducing new coordinates $y_{ij}^k\in {\mathcal A}_n$ for distinct $i,j,k\in [n]$ by:
$$y_{ij}^k:=x_{ki}^{-1}x_{kj} \ .$$
We refer to $y_{ij}^k$ as  {\it noncommutative sectors} and denote by ${\mathcal Q}_n$ the subalgebra of ${\mathcal A}_n$ generated by all $y_{ij}^k$ (with the convention $y_{ii}^k=1$).

\begin{theorem}
\label{th:presentation y}
The algebra ${\mathcal Q}_n$ is generated by all $y_{ij}^k$ subject to the relations:

(i) (triangle relations):
\begin{equation}
\label{eq:triangle relations y ijk}
y_{ij}^ky_{ji}^k=1,~y_{ij}^ky_{jk}^iy_{ki}^j=1
\end{equation}
for distinct $i,j,k\in [n]$ and
\begin{equation}
\label{eq:triangle relations y ijkl}
y_{ij}^\ell y_{jk}^\ell y_{ki}^\ell=1
\end{equation}
for distinct $i,j,k,\ell\in [n]$.

(ii) (exchange relations) For all cyclic $(i,j,k,\ell)$ in $[n]$:
\begin{equation}
\label{eq:exchange relations y}
y_{i\ell}^j=y_{ij}^ky_{j\ell}^i+y_{i\ell}^k\ .
\end{equation}

\end{theorem}

We prove Theorem  \ref{th:presentation y} in Section \ref{subsec:Free factorizations   and proof of Theorem}.

For any sequence $\jj=(j_0,j_1,\ldots,j_{2m})\in [n]^{2m}$ such that adjacent indices are distinct define a monomial $y_\jj\in {\mathcal Q}_n$ by:
$$y_\jj:=y_{j_0 j_2}^{j_1}y_{j_2 j_4}^{j_3}\cdots y_{j_{2m-2} j_{2m}}^{j_{2m-1}} \ .$$

The following  is a  ``polynomial equivalent" in ${\mathcal Q}_n$ of Theorem \ref{th:noncomlaurent n-gon}.

\begin{theorem} (Noncommutative polynomial phenomenon)
\label{th:noncompol n-gon} Let $\Delta$ be a triangulation of $[n]$. Then for any triple $(i,j,k)$ of distinct indices such that $(i,k)\in \Delta$ one has:
\begin{equation}
\label{eq:noncompol n-gon}
y_{kj}^i=\sum_{\ii\in Adm_\Delta(i,j)} y_{(k,\ii)}
\end{equation}
where $(k,\ii)$ stands for the sequence $\ii$ preceded by $k$.

\end{theorem}

We prove Theorem  \ref{th:noncompol n-gon} in Section \ref{subsect:proof of Theorems th:noncomlaurent n-gon and th:noncompol n-gon}.

\begin{example}
\label{ex:Schiffler y n=5,6} Following Example \ref{ex:Schiffler x n=5,6},

 (a) If $n=5$ and $\Delta=\{(1,3), (3,1),(1,4),(4,1);(i,i^\pm)|i\in [5]\}$, then
$$y_{15}^2=y_{15}^4+ y_{13}^2y_{35}^1+y_{14}^3y_{45}^1\ .$$

(b) If $n=6$ and $\Delta=\{(1,3),(3,1),(3,6),(6,3),(4,6),(6,4);(i,i^\pm)|i\in [5]\}$, then
$$
y_{15}^2=y_{16}^3 y_{65}^4+y_{13}^2y_{35}^6+y_{14}^3 y_{46}^5+y_{13}^2y_{36}^1 y_{65}^4+y_{13}^2y_{36}^1y_{64}^3y_{45}^6\ .
$$

\end{example}


Similarly to Section \ref{subsect:Definition and basic results}, for each triangulation $\Delta$ of $[n]$ define:

$\bullet$   The subalgebra ${\mathcal Q}_\Delta$  of ${\mathcal Q}_n$ generated by all $y_{ij}^k$, $i,j,k\in [n]$ such that  $(i,k),(k,j)\in \Delta$.

$\bullet$ the subgroup $\UU_\Delta$ of $\TT_\Delta$ generated by
\begin{equation}
\label{eq:uijk}
u_{ij}^k:=t_{ki}^{-1}t_{kj}\
\end{equation}
for $i,j,k\in [n]$ such that  $(i,k),(kj)\in \Delta$.

Clearly, the restriction of the homomorphism $\ii_\Delta$ given by \eqref{eq:i_Delta} to $\QQ\UU_\Delta\subset \QQ\TT_\Delta$ is a surjective homomorphism of algebras:
\begin{equation}
\label{eq:i_Delta U'}
\ii'_\Delta:\QQ\UU_\Delta \twoheadrightarrow  {\mathcal Q}_\Delta \ .
\end{equation}

Since ${\mathcal Q}_\Delta$ is a subalgebra of $\AA_\Delta$, the following is an immediate corollary of Theorem
\ref{th:A_Delta}.

\begin{corollary}
\label{cor:A_Delta Q} For each triangulation $\Delta$ one has:

(a) The homomorphism $\ii'_\Delta$ given by \eqref{eq:i_Delta U'} is an isomorphism.

(b)  ${\mathcal Q}_n={\mathcal Q}_\Delta[{{\bf S}'}^{-1}]$ for some multiplicative submonoid ${\bf S}'\subset {\mathcal Q}_\Delta\setminus \{0\}$.

(c) $\ii'_\Delta$ extends to a monomorphism of algebras $\QQ {\mathcal Q}_n\hookrightarrow Frac({\mathcal Q}_\Delta)=\FF_{2n-4}$.
\end{corollary}


%
%


%
%

\subsection{Regular elements in noncommutative polygons}

\label{subsec:regular polygon}
We start with a more economical presentation of ${\mathcal A}_n$. Denote by $\cU_n$ the subalgebra of $\AA_n$ generated by  all $x_{ij}$, $i\ne j$ and $x_{i,i^\pm}^{-1}$.
The following result is obvious.

\begin{lemma}
\label{le:upper bound type A} The algebra  $\cU_n$ satisfies the following relations

(a) (reduced triangle relations) for all $i,j\in [n]$, $i\notin \{j^-,j\}$:
\begin{equation}
\label{eq:triangle relations upper}
x_{i,j^-}x_{j,j^-}^{-1}x_{ji}=x_{ij}x_{j^-,j}^{-1}x_{j^-,i} \ .
\end{equation}

(b) (reduced exchange relations) for all cyclic $(i,j,k)$ in $[n]$ such that $i^-\notin \{j,k\}$:
\begin{equation}
\label{eq:exchange relations upper}
x_{ij}x_{j^-,j}^{-1}x_{j^-,k}= x_{ik}+x_{i,j^-}x_{j,j^-}^{-1}x_{jk},~x_{k,j^-}x_{j,j^-}^{-1}x_{ji}= x_{ki}+x_{kj}x_{j^-,j}^{-1}x_{j^-,i} \ .
\end{equation}

\end{lemma}

\begin{remark}
We expect that these relations are defining for the algebra $\cU_n$.
\end{remark}

Noncommutative Laurent phenomenon \eqref{th:noncomlaurent n-gon} guarantees that  $\cU_n$ belongs to each subalgebra $\AA_\Delta$. The following conjecture implies, in particular, that $\cU_n$ is a totally noncommutative analogue of the upper cluster algebra of type $A_{n-3}$.

\begin{conjecture} For each $n\ge 2$ one has:
\label{conj:upper bound type A}
\begin{equation}
\label{eq:ij_delta cup sharp } \cU_n= \bigcap_{\Delta}{\mathcal A}_\Delta \ ,
\end{equation}
where the intersection is over all triangulations $\Delta$ of $[n]$.

\end{conjecture}

We say that an element $x\in \AA_n$ is {\it regular} if it belongs to each subalgebra $\AA_\Delta$ as $\Delta$ runs over all triangulations $\Delta$ of $[n]$.
Thus, Conjecture \ref{conj:upper bound type A}  asserts that each regular element of ${\mathcal A}_n$ belongs to $\cU_n$, i.e., is a noncommutative polynomial in $x_{ij}$ and $x_{i,i^\pm}^{-1}$.




\subsection{Noncommutative angles}
\label{subsect:Noncommutative angles}
Now we take advantage of the ``invariant" algebra ${\mathcal Q}_n$ and will view the ambient algebra ${\mathcal A}_n$ as some ``Galois extension'' of ${\mathcal Q}_n$
(in fact,  Proposition \ref{pr:free action factorization} below guarantees that ${\mathcal A}_n$ is freely generated by $x_{i,i^-}$, $i\in [n]$ and ${\mathcal Q}_n$).

However, we want a more symmetric and ``geometric" presentation of ${\mathcal A}_n$ over ${\mathcal Q}_n$.
The following result provides such a presentation of $\AA_n$, $n\ge 3$.

\begin{proposition}
\label{pr:sector plus angle}
The algebra ${\mathcal A}_n$ is generated by
${\mathcal Q}_n$ and  $(T_i^{jk})^{\pm 1}$ for all distinct $i,j,k\in [n]$
subject to:

(i) (triangle relations) $T_i^{jk}=T_i^{kj}$ for all distinct $(i,j,k)$ in $[n]$.

(ii) (modified exchange relations) $T_i^{j\ell}=T_i^{jk}+T_i^{k\ell}$ for any cyclic $(i,j,k,\ell)$ in $[n]$.

(iii) (consistency relations) $y_{ji}^k T_i^{jk}=y_{ji}^\ell T_i^{j\ell}$ for all distinct $i,j,k,\ell\in [n]$.

\end{proposition}

\begin{proof} Denote by $\AA_n'$ the algebra whose presentation is given in the proposition.
It is easy to see that the assignments
$y_{ij}^k\mapsto x_{ki}^{-1}x_{kj}, T_i^{jk}\mapsto x_{ji}^{-1}x_{jk}x_{ik}^{-1}$
for distinct $i,j,k\in [n]$
define a homomorphism of algebras $\AA'_n\to \AA_n$.

On the other hand, the consistency relations imply that the element $(T_i^{jk})^{-1}y_{ij}^k$ does not depend on $k$. 

The following is immediate.

\begin{lemma} The assignments
$x_{ij}\mapsto (T_i^{jk})^{-1}y_{ij}^k$ for distinct $i,j\in [n]$ define a homomorphism of algebras $f:\AA_n\to \AA'_n$.
\end{lemma}

In particular, $f(x_{ji}^{-1}x_{jk}x_{ik}^{-1})=y_{ij}^k T_j^{ik}(T_j^{ki})^{-1}y_{jk}^iy_{ki}^jT_i^{jk}=
y_{ij}^k y_{jk}^iy_{ki}^jT_i^{jk}=T_i^{jk}$ by \eqref{eq:triangle relations y ijk}.

These homomorphisms are, clearly, inverse to each other and hence are isomorphisms.

The proposition is proved.
\end{proof}

We refer to $T_i^{jk}:=x_{ji}^{-1}x_{jk}x_{ik}^{-1}$ for all distinct $i,j,k\in [n]$ as {\it noncommutative angles} by a number of reasons. First,
because of the triangle relations in Proposition \ref{pr:sector plus angle} (so that we can attach $T_i^{jk}$ to the angle in the triangle $(i,j,k)$ at the vertex $i$) and, second, because of the modified exchange relations (ii) of Proposition \ref{pr:sector plus angle}
can be viewed as an ``addition law" of angles in a quadrilateral. In fact, such an addition law holds in more general situation.

\begin{corollary}
\label{cor:angle is additive}
For any cyclic $(i_0,i_1,i_2,\ldots,i_\ell)$  one has:
$T_{i_0}^{i_1,i_k}= T_{i_0}^{i_1,i_2}+T_{i_0}^{i_2,i_3}+\cdots +T_{i_0}^{i_{\ell-1},i_{\ell}}$.
In particular, $T_1^{2,n}=T_1^{23}+T_1^{34}+\cdots +T_1^{n-1,n}$.
\end{corollary}

Moreover, this view is supported by the following observation.
For each triangulation $\Delta$ of $n$ and each $i\in [n]$ define the
{\it total angle} $T_i^\Delta$ around the vertex $i$ to be the sum of all noncommutative angles in $\Delta$ at the vertex $i$. For instance, we have in Example \ref{ex:Schiffler y n=5,6}:
$$T_1^\Delta=T_1^{23}+T_1^{34}+T_1^{45},~T_2^\Delta=T_2^{13},~T_3^\Delta=T_3^{12}+T_3^{14},~T_4^\Delta=T_4^{13}+T_4^{15},~T_5^\Delta=T_5^{14} \ .$$

\begin{corollary} \label{cor:total angle}
$T_i^\Delta=T_i^{i^-,i^+}$ for any triangulation $\Delta$ of $[n]$ and any $i\in [n]$.
In particular, $T_i^\Delta$ does not depend on a choice of $\Delta$.

\end{corollary}

\begin{remark} Based on Corollary \ref{cor:total angle}, we can view $T_i:=T_i^{i^-,i^+}$ as the {\it total angle} of the noncommutative $n$-gon at the vertex $i$. The sum of all total angles $T:=T_1+T_2+\cdots +T_n$ also does not depend on a choice of triangulations and, in particular, can be specialized to any constant value (e.g., to $\pi\cdot(n-2)$).
\end{remark}

\begin{remark} The independence of $T_i$ of a choice of $\Delta$ means that $T_i$ is invariant under noncommutative mutations.
We will encounter the noncommutative angles again in Section \ref{sect:Noncommutative triangulated surfaces}.

\end{remark}

\subsection{Big triangle group of noncommutative polygons}
\label{subsec:Big noncommutative polygons}

For each $n\ge 2$ let $\TT_n$ be a group generated by $t_{ij}$, $i,j\in [n]$, $i\ne j$ subject to the triangle relations
$$t_{ij} t_{kj}^{-1}t_{ki}=t_{ik} t_{jk}^{-1}t_{ji}$$
for all distinct $i,j,k\in [n]$; and refer to this group as the {\it big triangle group} of the $n$-gon.

The following is obvious.

\begin{lemma} For any $n\ge 3$ one has:

(a) the assignments
$t_{ij}\mapsto
\begin{cases}
t_{1j} & \text{if $i=n$}\\
t_{i1} & \text{if $j=n$}\\
t_{ij} & \text{otherwise}\\
\end{cases}$
for $i,j\in [n]$, $i\ne j$  (with   the convention $t_{11}=1$)
define an epimorphism of groups $\pi_n^+:\TT_n\twoheadrightarrow \TT_{n-1}$.

(b) The assignments $t_{ij}\mapsto t_{ij}$ for $i,j\in [n-1]$, $i\ne j$ define an injective  homomorphism of groups
$\TT_{n-1} \hookrightarrow \TT_n$
which splits $\pi_n^+$.

\end{lemma}

The following result gives a presentation of $\TT_n$.

\begin{proposition}
\label{pr:big triangular presentation} For each $n\ge 3$ the group  $\TT_n$ is generated by  $t_{ij}$, $1\le i<j\le n$ and $t_{i1}$, $i=2,\ldots,n$, subject to:
$$t_{i1}t_{j1}^{-1}t_{jk}t_{1k}^{-1}t_{1j}t_{ij}^{-1}t_{ik}=t_{ik}t_{1k}^{-1}t_{1j}t_{ij}^{-1}t_{i1}t_{j1}^{-1}t_{jk}$$
for all $2\le i<j<k\le n$.
\end{proposition}

\begin{proof} By a slight abuse of notation, set $T_i^{jk}=t_{ji}^{-1}t_{jk}t_{ik}^{-1}$.
Clearly, if $n=3$, then $\TT_3$ is free in $t_{12},t_{13},t_{23},t_{21},t_{31}$. Furthermore, let $n\ge 4$. Then we can group the defining relations for $\TT_n$ into the following quadruples for $2\le i<j<k\le n$:
\begin{equation}
\label{eq:T_4 1ijk}
T_1^{ij}=T_1^{ji},~T_1^{ik}=T_1^{ki},~T_1^{jk}=T_1^{kj},~T_j^{ik}=T_j^{ki} \ .
\end{equation}

It is easy to see that each such quadruple \eqref{eq:T_4 1ijk} is equivalent to the following quadruple of relations (here $(i',j')\in \{(i,j),(i,k),(j,k)\}$):
$t_{j',i'}= t_{j',1}t_{i',1}^{-1}t_{i',j'}t_{1,j'}^{-1}t_{1,i'},~ t_{i1}t_{j1}^{-1}t_{jk}t_{1k}^{-1}t_{1j}t_{ij}^{-1}t_{ik}=t_{ik}t_{1k}^{-1}t_{1j}t_{ij}^{-1}t_{i1}t_{j1}^{-1}t_{jk}$.
Thus, eliminating the redundant generators $t_{j',i'}$, we finish the proof of the proposition.
\end{proof}

The following is obvious.

\begin{lemma} For each $n$ one has:

(a) The assignments $t_{ij}\mapsto x_{ij}$ define a ring epimorphism $\pi_n:\ZZ \TT_n\twoheadrightarrow \AA_n$.

(b) For each triangulation $\Delta$ of $[n]$ the assignments $t_{ij}\mapsto t_{ij}$ define a group homomorphism  $\hat {\bf j}_\Delta:\TT_\Delta\to \TT_n$.

(c) The symmetric group $S_n$ acts on $\TT_n$ by automorphisms: $\sigma(t_{ij}):=t_{\sigma(i),\sigma(j)}$ for  $\sigma\in S_n$, $i,j\in [n]$, $i\ne j$.

\end{lemma}

\begin{conjecture}
The restriction of $\pi_n$ to $\TT_n$ is an isomorphism of monoids $\TT_n\widetilde \to \AA_n^\times$.
\end{conjecture}

%
%
%
%

\begin{theorem}
\label{th:piDelta}
For any triangulation $\Delta$ of $[n]$ there exists an epimorphism
$\pi_\Delta:\TT_n\twoheadrightarrow \TT_\Delta$ such that
$$\hat {\bf j}_\Delta\circ \pi_\Delta=Id_{\TT_\Delta}\ .$$
In particular, $\hat {\bf j}_\Delta$ is an injective homomorphism $\TT_\Delta\hookrightarrow \TT_n$.

\end{theorem}

We prove Theorem \ref{th:piDelta} in Section \ref{subsec:Retraction of TT_n onto TT_Delta and proof of Theorem th:piDelta}.

The following is obvious.

\begin{corollary}
For any triangulations $\Delta,\Delta'$ of $n$ the composition $\tau_{\Delta,\Delta'}:=\pi_{\Delta'} \circ \hat {\bf j}_\Delta$ is an isomorphism $\TT_\Delta\to \TT_{\Delta'}$ such that $\tau_{\Delta,\Delta}=Id_{\TT_\Delta}$ and
$\tau_{\Delta,\Delta''}=\tau_{\Delta',\Delta''}\circ \tau_{\Delta,\Delta'}$
for any triangulation $\Delta''$ of $[n]$.

\end{corollary}


\subsection{Representation of ${\mathcal A}_n$ and ${\mathcal Q}_n$ in noncommutative $2\times n$ matrices}
\label{sect:2xn matrices}

In what follows, we identify the free skew field generated by  by all  $a_{1i}$, $a_{2i}$, $i\in [n]$ with $\FF_{2n}$
and view it as totally noncommutative rational functions on the space $Mat_{2\times n}$ of $2\times n$ matrices.

Following \cite{gr3} and \cite{BR} (see also \cite{ggrw}) define  $2\times 2$-{\it positive quasiminors} by
$$\left |\begin{matrix} a_{1i}&\boxed {a_{1j}}\\a_{2i}&a_{2j}\end{matrix}
\right |_+=
\sgn(i-j)(a_{1j}-a_{1i}a_{2i}^{-1}a_{2j}),~ \left |\begin{matrix} a_{1i}&a_{1j}\\a_{2i}&\boxed {a_{2j}}\end{matrix}
\right |_+=  \sgn(j-i)(a_{2j}-a_{2i}a_{1i}^{-1}a_{1j})$$
for  $i,j\in [n]$ and {\it positive} quasi-Pl\"ucker coordinates $Q_{ij}^k$ for distinct $i,j,k\in [n]$  by:
\begin{equation}
\label{eq:quasiplucker definition}
Q_{ij}^k=\left |\begin{matrix} a_{1k}&\boxed {a_{1i}}\\a_{2k}&a_{2i}\end{matrix}
\right |_+^{-1}\cdot \left |\begin{matrix} a_{1k}&\boxed {a_{1j}}\\a_{2k}&a_{2j}\end{matrix}
\right |_+=\left |\begin{matrix} a_{1k}& a_{1i}\\a_{2k}&\boxed {a_{2i}}\end{matrix}
\right |_+^{-1}\cdot \left |\begin{matrix} a_{1k}&a_{1j}\\a_{2k}&\boxed {a_{2j}}\end{matrix}
\right |_+
\end{equation}
(the latter identity is proved in \cite[Section 4.3]{gr3} and in
\cite[Proposition 4.2.1]{ggrw}).

\begin{theorem}
\label{th:phi tilde}
For each $n\ge 2$ one has:

(a)
The assignments
$x_{ij}\mapsto  \left |\begin{matrix} a_{1i}&\boxed {a_{1j}}\\a_{2i}&a_{2j}\end{matrix}
\right |_+, ~x_{ij}\mapsto \left |\begin{matrix} a_{1i}&a_{1j}\\a_{2i}&\boxed {a_{2j}}\end{matrix}
\right |_+$ for all distinct $i,j\in [n]$ define respectively the homomorphisms of algebras
\begin{equation}
\label{eq:phi plus minus}
\varphi_n^+:{\mathcal A}_n\to \FF_{2n}, ~\varphi_n^-:{\mathcal A}_n\to \FF_{2n}\ .
\end{equation}

(b) The restrictions of $\varphi_n^+$ and $\varphi_n^-$ to ${\mathcal Q}_n$ are equal to an algebra homomorphsim $\varphi_n:{\mathcal Q}_n \to \FF_{2n}$ such that $\varphi_n(y_{ij}^k)= Q_{ij}^k$ for distinct $i,j,k\in [n]$.
%

\end{theorem}

\begin{proof} Our proof is based on Proposition \ref{pr:free action factorization} below. It follows from \cite[Section 4.4]{gr3} that positive quasi-Pl\"ucker coordinates satisfy \eqref{eq:triangle relations y ijk}, \eqref{eq:triangle relations y ijkl}, and \eqref{eq:exchange relations y}
This implies that the assignments
\begin{equation}
\label{eq:phin}
y_{ij}^k\mapsto Q_{ij}^k
\end{equation}
for distinct $i,j\in [n]$ define a homomorphism of algebras $\varphi_n:{\mathcal Q}_n\to \FF_{2n}$.

Furthermore, for any $\QQ$-algebras $\AA$ and ${\mathcal B}$ denote by $\AA*{\mathcal B}$ their free product, i.e., the universal algebra generated by $\AA$ and ${\mathcal B}$ as subalgebras (with no relations between them).
The most fundamental property of the free product is that  any algebra homomorphisms $f_1:\AA\to {\mathcal C}$, $f_2:{\mathcal B}\to {\mathcal C}$ canonically lift
to an algebra homomorphism $f_1*f_2:\AA*{\mathcal B}\to {\mathcal C}$.

Denote by $F_m$ the free group generated by $c_i^{\pm 1}$, $i=1,\ldots,m$.

By definition, the group algebra $\QQ F_m$, is free Laurent polynomial algebra $\QQ<c_1^{\pm 1},\ldots,c_{m}^{\pm 1}>$.

\begin{proposition}
\label{pr:free action factorization}
For each $n\ge 2$ the assignments $x_{ij}\mapsto c_i*y_{i^-,j}^i$, $i,j\in [n]$, $i\ne j$ define an isomorphism of algebras
\begin{equation}
\label{eq:quasi-coaction}
f:\AA_n\widetilde \to  (\QQ F_n)*{\mathcal Q}_n\ .
\end{equation}

\end{proposition}

\begin{proof} Let us prove that the homomorphism \eqref{eq:quasi-coaction} is well-defined.
We need the following obvious fact.

\begin{lemma} Let ${\mathcal B}$ be  a $\QQ$-algebra and let $c_1,\ldots,c_n$ be invertible elements of ${\mathcal B}$. Then the assignments
\begin{equation}
\label{eq:xij ci xij}
x_{ij}\mapsto c_i*x_{ij}
\end{equation}
for $i,j\in [n]$, $i\ne j$ define a homomorphism of algebras $\AA_n\to {\mathcal B}*\AA_n$.

\end{lemma}

By the above Lemma ${\mathcal B}:=\QQ F_n$ generated by $c_i^{\pm 1}$, $i\in [n]$, the assignments
\eqref{eq:xij ci xij} define a homomorphism of algebras
\begin{equation}
\label{eq:AnGnAn}\AA_n\to (\QQ F_n)*\AA_n \ .
\end{equation}

Furthermore, the assignments $c_i\mapsto c_i*x_{i,i^-}^{-1}$, $i\in [n]$ define an algebra homomorphism
$f_1:\QQ F_n\to (\QQ F_n)*\AA_n$ and the identity map $\AA_n\to \AA_n$ defines a homomorphism of algebras $f_2:\AA_n\to (\QQ F_n)*\AA_n$.
This gives an algebra homomorphism $f_1*f_2:(\QQ F_n)*\AA_n\to (\QQ F_n)*\AA_n$ determined by $c_i\mapsto c_i*x_{i,i^-}^{-1}$, $x_{ij}\mapsto x_{ij}$.
Then the composition of the homomorphism \eqref{eq:AnGnAn} with $f_1*f_2$ is a homomorphism of algebras
$$\AA_n \to  (\QQ F_n)*\AA_n$$
given by
$x_{ij}\mapsto c_i*x_{ij}\mapsto c_i*x_{i,i^-}^{-1}x_{ij}=c_i*y_{i^-,j}^i$
for all $i,j\in [n]$, $i\ne j$.  Since the image of the latter homomorphism belongs to $(\QQ F_n)*{\mathcal Q}_n$,
we see that the algebra homomorphism $f:\AA_n \to  (\QQ F_n)*{\mathcal Q}_n$ given by \eqref{eq:quasi-coaction} is well-defined.

It remains to show that $f$ is invertible. Indeed, denote by $f'_1:\QQ F_n\to \AA_n$ the homomorphism of algebras given by
$c_i\mapsto x_{i,i^-}$, $i\in [n]$ and denote by $f'_2$ the natural inclusion ${\mathcal Q}_n \hookrightarrow \AA_n$.
This defines a homomorphism of algebras  $g=f'_1*f'_2:(\QQ F_n)*{\mathcal Q}_n\to \AA_n$
which is determined by $c_i\mapsto x_{i,i^-}$, $y_{ij}\mapsto y_{ij}$. This immediately implies that
$(g\circ f)(x_{ij})=g(c_i*y_{i^-,j}^i)=x_{i,i^-}y_{i^-,j}^i=x_{ij}$
for all $i\ne j$. Therefore, $g\circ f=Id$. Similarly,
$$(f\circ g)(c_i)=f(x_{i,i^-})=c_i*y_{i^-,i^-}^i=c_i*1=c_i,~(f\circ g)(y_{i^-,j}^i)=f(y_{i^-,j}^i)$$
$$=f(x_{i,i^-}^{-1}x_{ij})=f(x_{i,i^-})^{-1}f(x_{ij})=(c_i*x_{i,i^-})^{-1}c_i*x_{ij}=x_{i,i^-}^{-1}x_{ij}=y_{ij}\ .$$
Therefore, $f\circ g=Id$ as well.

Proposition \ref{pr:free action factorization} is proved.
\end{proof}

Now we can finish the proof of Theorem \ref{th:phi tilde}. Define algebra homomorphisms $\psi_n^+,\psi_n^-:\QQ F_n\to \FF_{2n}$ by
$$\psi_n^+(c_i)=  \left |\begin{matrix} a_{1i}&\boxed {a_{1,i^-}}\\a_{2i}&a_{2,i^-}\end{matrix}
\right |_+,\psi_n^-(c_i)= \left |\begin{matrix} a_{1i}&a_{1,i^-}\\a_{2i}&\boxed {a_{2,i^-}}\end{matrix}
\right |_+$$
for $i\in [n]$.
Universality of free products gives natural homomorphsims of algebras
$$\psi_n^+*\varphi_n, \psi_n^-*\varphi_n:(\QQ F_n)*{\mathcal Q}_n\to \FF_{2n}\ ,$$
where $\varphi_n$ is given by \eqref{eq:phin}.
Composing these homomorphisms with the isomorphism  \eqref{eq:quasi-coaction}, we obtain respectively algebra homomorphsims $\varphi_n^+,\varphi_n^-:\AA_n\to \FF_{2n}$, whose restriction to ${\mathcal Q}_n$ equals $\varphi_n$.

Finally, note that
$$\varphi_n^+(x_{ij})=\psi_n^+*\varphi_n(c_i*y_{i^-,j}^i)=\left |\begin{matrix} a_{1i}&\boxed {a_{1,i^-}}\\a_{2i}&a_{2,i^-}\end{matrix}
\right |_+\cdot Q_{i^-,j}^i=\left |\begin{matrix} a_{1i}&\boxed {a_{1j}}\\a_{2i}&a_{2j}\end{matrix}
\right |_+\ ,$$
$$\varphi_n^-(x_{ij})=\psi_n^-*\varphi_n(c_i*y_{i^-,j}^i)=\left |\begin{matrix} a_{1i}&a_{1,i^-}\\a_{2i}&\boxed {a_{2,i^-}}\end{matrix}
\right |_+\cdot Q_{i^-,j}^i=\left |\begin{matrix} a_{1i}&a_{1j}\\a_{2i}&\boxed {a_{2j}}\end{matrix}
\right |_+\ $$
for all distinct $i,j\in [n]$.

Theorem \ref{th:phi tilde} is proved. \end{proof}

\begin{remark} Proposition \ref{pr:free action factorization} is a noncommutative algebraic analogue of the following assertion: if a group $G$ acts freely on a set $X$,
then there a bijection $X\widetilde \to G\times (X/G)$.

\end{remark}

\begin{remark} Unlike $\varphi_n$ (see Theorem \ref{th:ii' phi}), the homomorphisms $\varphi_n^+$ and $\varphi_n^-$ are not injective. In particular, one can show that $x_{ij}x_{kj}^{-1}+1-x_{ik}x_{jk}^{-1}\in Ker~\varphi_n^-$ for all distinct $i,j,k\in [n]$.

\end{remark}

For any groups $G$ and $H$ denote by $G*H$ their free product. It is well-known (see, e.g.,  \cite{Cohn}) that $\QQ (G*H)=(\QQ G)*(\QQ H)$.

\begin{proposition} \label{pr:T_Delta polygon U}
For each triangulation $\Delta$ of $[n]$ the assignments
$t_{ij}\to c_i*u_{i^-,j}^i$ for all $(i,j)\in \Delta$ (in the notation of \eqref{eq:uijk}) define an isomorphism of (free) groups
\begin{equation}
\label{eq:free factorization T=GU}
\TT_\Delta\widetilde \to F_n*\UU_\Delta\ .
\end{equation}

\end{proposition}

\begin{proof} We essentially copy the proof of Proposition \ref{pr:free action factorization}. Indeed, the following fact is obvious.

\begin{lemma} Let $G$ be any group and  let $c_1,\ldots,c_n\in G$. Then for any triangulation $\Delta$ of $[n]$ the assignments
\begin{equation}
\label{eq:tij ci tij}
t_{ij}\mapsto c_i*t_{ij}
\end{equation}
for $(i,j)\in \Delta$,  define a homomorphism of groups
$\TT_\Delta\to G*\TT_\Delta$.

\end{lemma}

Clearly, the assignments
$c_i\mapsto c_i*t_{i,i^-}^{-1}$
for $i\in [n]$ define a group homomorphism $F_n\to F_n*\TT_\Delta$. Composing this  with \eqref{eq:tij ci tij} taken with $G=F_n=<c_1,\ldots,c_n>$ and applying the multiplication homomorphism $\TT_\Delta*\TT_\Delta\to \TT_\Delta$, we obtain a group homomorphism:
$\TT_\Delta\to  F_n*\TT_\Delta*\TT_\Delta\to F_n*\TT_\Delta$
given by $t_{ij}\mapsto c_i*u_{i^-,j}^i$
for all $i,j\in \Delta$. Clearly, the image of this homomorphism contains all $c_i$ and $u_{ij}^k$, $(i,j),(jk)\in \Delta$ and is contained in $F_n*\UU_\Delta$, hence
this gives a group homomorphism \eqref{eq:free factorization T=GU}.
It is also clear that the homomorphism $F_n*\UU_\Delta\to \TT_\Delta$ given by
$c_i\mapsto t_{i,i^-},~u_{ij}^k\mapsto u_{ij}^k$
is inverse of \eqref{eq:free factorization T=GU}.

The proposition is proved.
\end{proof}

Taking into account that $F_n*F_m\cong F_{m+n}$, we obtain an obvious corollary from Theorem \ref{th:T_Delta polygon}.

\begin{corollary} For each triangulation $\Delta$ of $[n]$ the group $\UU_\Delta$ is isomorphic to $F_{2n-4}$, the free group in $2n-4$ generators.

\end{corollary}

Furthermore, denote by  $\FF'_{2n-4}$ the  skew sub-field of $\FF_{2n}$ generated by   $\varphi_n({\mathcal Q}_n)$, i.e., by all $Q_{ij}^k$.

\begin{proposition}
\label{pr:F' free}
$\FF'_{2n-4}$ is isomorphic to $\FF_{2n-4}$.
\end{proposition}

\begin{proof} Denote:
\begin{equation}
\label{eq:ABC}
A=\begin{pmatrix} a_{11} &\cdots & a_{1n}\\
a_{21} &\cdots & a_{2n}\\
\end{pmatrix}, B=\begin{pmatrix} a_{13} &\cdots & a_{1n}\\
a_{23} &\cdots & a_{2n}\\
\end{pmatrix},
C=\begin{pmatrix} a_{11} &a_{12}\\
a_{21} & a_{22}\\
\end{pmatrix}
\end{equation}
so that $A=[C\,|\,B]$.

The following lemma easily follows from Theorem 4.4.4 and Proposition 4.5.2 in \cite{ggrw}.

\begin{lemma} 
$C^{-1}B =
\begin{pmatrix}  q_{13}^2 &\cdots & q_{1n}^2\\
q_{23}^1 &\cdots & q_{2n}^1\\
\end{pmatrix}$,
where $q_{ij}^k=q_{ij}^k(A)=\sgn(k-j)\sgn(k-i)Q_{ij}^k$ for distinct $i,j,k\in [n]$  are quasi-Pl\"ucker coordinates (in the notation \eqref{eq:quasiplucker definition}).

\end{lemma}

It was proved in \cite[Section 4]{gr3} that $q_{ij}^k(A)=q_{ij}^k(DA)$ for all distinct $i,j,k\in [n]$ and any invertible $2\times 2$ matrix $D$ over $\FF_{2n}$. In particular, taking $D=C^{-1}$, we see that $q_{ij}^k=q_{ij}^k([C\,|\,B])=q_{ij}^k([I_2\,|\,C^{-1}B])$, therefore, each $q_{ij}^k$ belongs to the skew subfield of $\FF_{2n}$ generated by the matrix coefficients of $C$ (here $I_2$ is the $2\times 2$ identity matrix). This proves that $\FF'_{2n-4}$ is a sub-field of $\FF_{2n}$ generated by the entries of $C^{-1}B$, i.e., by all $q_{1j}^2,q_{2j}^1$, $j=3,\ldots,n$.

It remains to show that matrix coefficients of $C^{-1}B$ (freely) generate a free subfield of $\FF_{2n}$. We need the following obvious fact.

\begin{lemma} \label{le:C-1B}
Let $\FF$ be a skew field, $C\in GL_m(\FF)$ and $B\in Mat_{m,n-m}(\FF)$ such that matrix coefficients of the partitioned matrix $A=[C\,|\,B]$ generate $\FF$. Then the matrix coefficients of $[C\,|\,C^{-1}B]$ also generate $\FF$.

\end{lemma}

Now we take
$m=2$ and $B,C$ as in \eqref{eq:ABC},  $\FF=\FF_{2n}$, the free skew field freely generated by matrix coefficients of $A=[C\,|\,B]$. Then $C\in GL_2(\FF_{2n})$ and $B\in Mat_{2,n-2}(\FF_{2n})$. Then, by Lemma \ref{le:C-1B},  the matrix coefficients $A'=[C\,|\,C^{-1}B]$ also generate $\FF_{2n}$.
Since $A'$ is $2\times n$, then Proposition \ref{pr:Schofield} implies that the matrix coefficients of $A'$ are free generators of $\FF_{2n}$. In particular, the matrix coefficients of the $2\times (n-2)$ matrix $C^{-1}B$ are free generators of the free skew sub-field of $\FF_{2n}$. That is, $\FF'_{2n-4}$ is freely generated by the matrix coefficients $q_{1j}^2,q_{2j}^1$, $j=3,\ldots,n$ of $C^{-1}B$.

The proposition is proved.
\end{proof}

\begin{remark} Proposition \ref{pr:F' free} and its proof generalize verbatim to $m\times n$ matrices.

\end{remark}


\begin{theorem}
\label{th:ii' phi}
For each triangulation $\Delta$ of $[n]$ the homomorphism
\begin{equation}
\label{eq:idelta phi_+}
\varphi_n\circ \ii'_\Delta:\QQ \UU_\Delta\to \FF'_{2n-4}
\end{equation}
is injective.

\end{theorem}

\begin{proof}
%
%
Taking any free generating set  $u_1,\ldots,u_{2n-4}$ of the free group $\UU_\Delta\cong F_{2n-4}$,
we see that $t_i:=\varphi_n(\ii'_\Delta(u_i))$, $i=1,\ldots,2n-4$ generate $\FF'_{2n-4}$ due to the following fact.

\begin{lemma}
\label{le:quasiplucker embedding delta}
For each triangulation $\Delta$ of $[n]$ the image  $\varphi_n({\mathcal Q}_\Delta)$ generates the skew field $\FF'_{2n-4}$.

\end{lemma}

\begin{proof} Denote by $\FF_{2n-4}''$ the skew subfield of $\FF_{2n}$ generated by image  $\varphi_n({\mathcal Q}_\Delta)$. Since image  ${\mathcal Q}_\Delta\subset \QQ_n$, we have an obvious inclusion $\FF_{2n-4}''\subseteq \FF_{2n-4}'$.
\end{proof}

Therefore using Proposition \ref{pr:Schofield} with $\ell=2n-4$, we see that  $t_1,\ldots,t_{2n-4}$ are free generators of $\FF'_{2n-4}$ and hence
the  homomorphism \eqref{eq:idelta phi_+}  is injective.

Theorem \ref{th:ii' phi} is proved.
\end{proof}

\subsection{Some symmetries of noncommutative polygons}
First, we establish a new presentation of $\AA_n$ in generators $\tilde x_{ij}^{\pm 1}:=\sgn(j-i)x_{ij}^{\pm 1}$   and $\tilde T_i^{jk}=\tilde x_{ji}^{-1}\tilde x_{jk}\tilde x_{ik}^{-1}=\sgn(i-j)\sgn(k-j)\sgn(k-i)  x_{ji}^{-1} x_{jk}  x_{ik}^{-1}$ for distinct $i,j,k\in [n]$ (see also Section \ref{subsect:Noncommutative angles}). Also define $\tilde y_{ij}^k=\tilde x_{ki}^{-1}\tilde x_{kj}=\sgn(i-k)\sgn(j-k)y_{ij}^k$
for distinct $i,j,k\in [n]$.

We need the following useful fact.

\begin{lemma}
\label{le:presentation An tilde} For each $n\ge 2$ one has:

(a) The algebra $\AA_n$ is generated by $\tilde x_{ij}$ for distinct $i,j\in [n]$ subject to the relations:
\begin{equation}
\label{eq:triangle relations tilde}
\tilde T_i^{jk}=-\tilde T_i^{kj}
\end{equation}
for any distinct $i,j,k\in [n]$:

\begin{equation}
\label{eq:exchange relations tilde}
\tilde T_i^{jk}+\tilde  T_i^{k\ell}+\tilde  T_i^{\ell j}=0
\end{equation}
for any distinct $i,j,k,\ell\in [n]$.

(b) The algebra ${\mathcal Q}_n$ is generated by all $\tilde y_{ij}^k$ subject to the relations:

\begin{equation}
\label{eq:triangle relations y ijk tilde}
\tilde y_{ij}^k\tilde y_{ji}^k=1,~\tilde y_{ij}^k\tilde y_{jk}^i\tilde y_{ki}^j=-1
\end{equation}
for distinct $i,j,k\in [n]$,
\begin{equation}
\label{eq:triangle relations y ijkl tilde}
\tilde y_{ij}^\ell \tilde y_{jk}^\ell \tilde y_{ki}^\ell=1,~\tilde y_{ik}^j\tilde y_{ki}^\ell +\tilde y_{i\ell}^j\tilde y_{\ell i}^k=1
\end{equation}
for distinct $i,j,k,\ell\in [n]$.

\end{lemma}

\begin{proof} Prove (a).
Denote by $\AA_n''$ the algebra freely generated by all $\tilde x_{ij}^{\pm 1}$, $i\ne j$. That is, $\AA_n''$ is the group algebra of a free group in $n^2-n$ generators. Define $\tilde r_{ijk}=\tilde T_i^{kj}(\tilde T_i^{jk})^{-1}$.
for all distinct  $i,j,k\in [n]$.
Clearly,
$$\tilde r_{ijk}=\tilde x_{ki}^{-1}\tilde x_{kj}\tilde x_{ij}^{-1}\tilde x_{ik}\tilde x_{jk}^{-1}\tilde x_{ji}= -x_{ki}^{-1}x_{kj}x_{ij}^{-1}x_{ik}x_{jk}^{-1}x_{ji}=\tilde y_{ij}^k\tilde y_{jk}^i \tilde y_{ki}^j=-y_{ij}^k y_{jk}^i y_{ki}^j$$
for all distinct  $i,j,k\in [n]$. Denote by ${\mathcal I}'$ the ideal in $\AA''_n$ generated by all $\tilde r_{ijk}+1$.
Then the quotient $\AA_n':=\AA_n''/{\mathcal I}'$  is an algebra generated by $x_{ij}$, $i,j\in [n]$, $i\ne j$  subject to the triangle relations
\eqref{eq:triangle relations}.

Furthermore, for any distinct $i,j,k,\ell\in [n]$ define $\tilde r_{i;j,k,\ell}\in \AA_n'$ by
$\tilde r_{i;j,k,\ell}=\tilde T_i^{jk}+\tilde  T_i^{k\ell}+\tilde  T_i^{\ell j}$.

Clearly,
$\tilde r_{i;j,k,\ell}=-\tilde r_{i;k,j,\ell}=-\tilde r_{i;j,\ell,k}$
for all $i,j,k,\ell$, i.e., $r_{i;j,k,\ell}$ is skew-symmetric in $j,k,\ell$ because of \eqref{eq:triangle relations tilde}. Note also that
$$\tilde r_{i;j,k,\ell}=\tilde x_{ji}^{-1}(\tilde x_{jk}\tilde x_{ik}^{-1}\tilde x_{i\ell} + \tilde x_{ji}\tilde x_{ki}^{-1}\tilde x_{k\ell}-\tilde x_{j\ell})\tilde x_{i\ell}^{-1}
=(\tilde y_{ik}^j\tilde y_{k\ell}^i + \tilde y_{i\ell}^k-\tilde y_{i\ell}^j)\tilde x_{i\ell}^{-1}=(-\tilde y_{ik}^j\tilde y_{ki}^\ell + 1-\tilde y_{i\ell}^j\tilde y_{\ell i}^k)\tilde x_{i\ell}^{-1}$$
for all distinct $i,j,k,\ell\in [n]$. Moreover, if $(i,j,k,\ell)$ is cyclic, i.e., $(i,k)$ crosses $(j,\ell)$,  this gives:
$$\tilde r_{i;j,k,\ell}=\pm x_{ji}^{-1}(x_{jk}x_{ik}^{-1}x_{i\ell} + x_{ji}x_{ki}^{-1}x_{k\ell}-x_{j\ell})x_{i\ell}^{-1} \ .$$
Therefore, if we denote by ${\mathcal I}$ the ideal in $\AA'_n$ generated by all $\tilde r_{i;j,k,\ell}$, then, obviously, $\AA'_n/{\mathcal I}\cong \AA_n$.

This proves (a).

Part (b) also follows because the relations \eqref{eq:triangle relations y ijk tilde} and  \eqref{eq:triangle relations y ijkl tilde} are equivalent to
\eqref{eq:triangle relations y ijk},
\eqref{eq:triangle relations y ijkl}, and
\eqref{eq:exchange relations y}.
The lemma is proved.
\end{proof}

In the notation of Lemma \ref{le:presentation An tilde}  define the action of the symmetric group $S_n$ on the set $\tilde X=\{\tilde x_{ij}|i,j\in [n], i\ne j\}$ by the formula
$$w(\tilde x_{ij})=\tilde x_{w(i),w(j)}$$
for all $w\in S_n$, $i,j\in [n]$, $i\ne j$.

\begin{proposition} For each $n\ge 2$ one has:

(a) The above action uniquely extends to an action of $S_n$ on $\AA_n$ by algebra automorphisms.

(b) The action commutes with homomorphisms $\varphi_n^+,\varphi_n^-:\AA_n\to \FF_{2n}$ given by \eqref{eq:phi plus minus}, where the action of $S_n$ on $\FF_{2n}$ is given by
$w(a_{s,i})=a_{s,w(i)}$
for $s=1,2$, $i\in [n]$, $w\in S_n$.

(c) The subalgebra ${\mathcal Q}_n$ is invariant under the $S_n$-action, i.e.,
$w(\tilde y_{ij}^k)= \tilde y_{w(i),w(j)}^{w(k)}$
for all $i,j,k\in [n]$, $w\in S_n$.

\end{proposition}

\begin{proof} Prove (a). The following fact is obvious.

\begin{lemma} The $S_n$ action on $\tilde X$ uniquely extends to that on $\AA_n''=\QQ\langle\tilde X\rangle$ by algebra automorphisms.
\end{lemma}

Thus, it suffices to prove that the $S_n$-action on $\AA_n''$ preserves the ideal of triangle relations \eqref{eq:triangle relations tilde} and exchange relations \eqref{eq:exchange relations tilde}.

Let us prove that the ideal ${\mathcal I}'$ of $\AA''$ generated by all $\tilde r_{ijk}$ is invariant under the $S_n$-action. Indeed,
for distinct $i,j,k\in [n]$ and  $w\in S_n$ one has
$$w(\tilde r_{ijk})=w(\tilde x_{ij})w(\tilde x_{kj})^{-1}w(\tilde x_{ki})w(\tilde x_{ji})^{-1}w(\tilde x_{jk})w(\tilde x_{ik})^{-1}=\tilde r_{w(i),w(j),w(k)} \ .$$
This proves that $S_n({\mathcal I}')={\mathcal I}'$ hence  $S_n$ acts on $\AA_n'$ by algebra automorphisms.

It remains to prove that the ideal of exchange relations \eqref{eq:exchange relations tilde} in $\AA_n'$ is invariant under the $S_n$-action.
Now we show that the ideal ${\mathcal I}$ of $\AA'_n=\AA''_n/{\mathcal I}'_n$ generated by all $\tilde r_{i;j,k,\ell}$ is invariant under the $S_n$-action.
Indeed,
$$w(\tilde r_{i;j,k,\ell})=w(\tilde T_i^{jk})+w(\tilde T_i^{k\ell})+w(\tilde T_i^{\ell j})=\tilde T_{w(i)}^{w(j),w(k)} + \tilde  T_{w(i)}^{w(k),w(\ell)}+\tilde T_{w(i)}^{w(\ell),w(j)}=\tilde  r_{w(i);w(j),w(k),w(\ell)}$$
for all distinct $i,j,k,\ell\in [n]$.
This proves that $S_n({\mathcal I})={\mathcal I}$.

Part (a) is proved.

Part (b) follows from the fact that the homomorphisms $\varphi_n^+,\varphi_n^-:\AA_n\to \FF_{2n}$  from Theorem \ref{th:phi tilde} commute with the $S_n$-action.

Part (c) is obvious.

The proposition is proved.
\end{proof}

The Lie algebra $gl_n(\QQ)$ (viewed as $Mat_{n\times n}$) naturally acts on the space $Mat_{2\times n}$ by right multiplications, i.e.,
$E_{ij}(a_{s,t})=\delta_{t,j}a_{s,i}$ for $s\in \{1,2\}$, $i,j,t\in [n]$), where $E_{ij}\in gl_n(\QQ)$ are the matrix units.

This action uniquely extends to $\FF_{2n}$ by the Leibniz rule:
$E(fg)=E(f)g+fE(g),~E(h^{\,-1})=-h^{-1}E(h)h^{-1}$
for any $E\in gl_n(\QQ)$, $f,g\in \FF_{2n}$, $h\in \FF_{2n}\setminus \{0\}$.

\begin{proposition}
\label{pr:gln action}
For each $n\ge 2$
there exists a unique action of  $gl_n(\QQ)$  on ${\mathcal Q}_n$ by derivations
such that the homomorphism $\varphi_n:{\mathcal Q}_n\to \FF_{2n}$ from Theorem \ref{th:phi tilde}(b)
is $gl_n(\QQ)$-equivariant. The action is given by:
%
\begin{equation}
\label{eq:gln action}
E_{i',j'}(\tilde y_{i,j}^k)=\begin{cases}
0 & \text{if $j'\notin\{i,j,k\}$}\\
\tilde y_{i,i'}^k&\text{if $j'=j$}\\
-\tilde y_{i,i'}^k \tilde y_{ij}^k &\text{if $j'=i$}\\
-\tilde y_{i,i'}^k \tilde y_{kj}^i &\text{if $j'=k$}\\
\end{cases}
\end{equation}
for any distinct indices $i,j,k\in [n]$.
\end{proposition}

\begin{proof} Indeed, in view of Theorem \ref{th:ii' phi}, it suffices to prove \eqref{eq:gln action} for $q_{ij}^k=\varphi_n(\tilde y_{ij}^k)$.
Indeed, if we abbreviate $\underline x_{ij}=\left |\begin{matrix} a_{1i}&\boxed {a_{1j}}\\a_{2i}&a_{2j}\end{matrix}
\right |$ for distinct $i,j\in [n]$, then
$$E_{i'j'}(\underline x_{ij})=E_{i',j'}(a_{1j}-a_{1i}a_{2i}^{-1}a_{2j})=
\begin{cases}
0 & \text{if $j'\notin\{i,j\}$}\\
a_{1,i'}-a_{1i}a_{2i}^{-1}a_{2,i'} & \text{if $j'=j$}\\
-E_{i',i}(a_{1i}a_{2i}^{-1})a_{2j}& \text{if $j'=i$}\\
\end{cases}=
\begin{cases}
0 & \text{if $j'\notin\{i,j\}$}\\
\underline x_{i,i'} & \text{if $j'=j$}\\
\underline x_{i,i'}\underline x_{ji}^{-1}\underline x_{ij}& \text{if $j'=i$}\\
\end{cases}$$
because $-E_{i',i}(a_{1i}a_{2i}^{-1})=-a_{1,i'}a_{2i}^{-1}+a_{1i}a_{2i}^{-1}a_{2,i'}a_{2i}^{-1}=-\underline x_{i,i'}a_{2i}^{-1}$
and $a_{2i}^{-1}a_{2j}=-\underline x_{ji}^{-1}\underline x_{ij}$
for $i\ne j$.
Therefore,

$$E_{i'j'}(q_{ij}^k)=E_{i'j'}(\underline x_{ki}^{-1}\underline x_{kj})=E_{i'j'}(\underline x_{ki}^{-1})\underline x_{kj}+\underline x_{ki}^{-1} E_{i'j'}(\underline x_{kj})
=\begin{cases}
\underline x_{ki}^{-1}E_{i'j}(\underline x_{kj}) & \text{if $j'=j$}\\
E_{i'i}(\underline x_{ki}^{-1})\underline x_{kj}& \text{if $j'=i$}\\
E_{i',k}(\underline x_{ki}^{-1})\underline x_{kj}+\underline x_{ki}^{-1}\underline E_{i'k}(x_{kj})& \text{if $j'=k$}\\
0 & \text{otherwise}\\
\end{cases}.
$$

Note that
$$E_{i',k}(\underline x_{ki}^{-1})\underline x_{kj}+\underline x_{ki}^{-1}E_{i'k}(\underline x_{kj})=
-\underline x_{ki}^{-1}(\underline x_{k,i'}\underline x_{ik}^{-1}\underline x_{ki})\underline x_{ki}^{-1}\underline x_{kj}
+\underline x_{ki}^{-1}(\underline x_{k,i'}\underline x_{jk}^{-1}\underline x_{kj})$$
$$=\underline x_{ki}^{-1}\underline x_{k,i'}(-\underline x_{ik}^{-1}+\underline x_{jk}^{-1})\underline x_{kj}
=-\underline x_{ki}^{-1}\underline x_{k,i'}\underline x_{ik}^{-1}(\underline x_{ik}-\underline x_{jk})\underline x_{jk}^{-1}\underline x_{kj}=
\underline x_{ki}^{-1}\underline x_{k,i'}\underline x_{ik}^{-1}\underline x_{ij}$$
because
$$(\underline x_{ik}-\underline x_{jk})\underline x_{jk}^{-1}\underline x_{kj}=((a_{1k}-a_{1i}a_{2i}^{-1}a_{2k})-(a_{1k}-a_{1j}a_{2j}^{-1}a_{2k}))(-a_{2k}^{-1}a_{2j})=-(-a_{1i}a_{2i}^{-1}a_{2j}+a_{1j})=-\underline x_{ij}$$
Therefore,
$$E_{i'j'}(q_{ij}^k)=\begin{cases}
0 & \text{if $j'\notin\{i,j,k\}$}\\
\underline x_{ki}^{-1}\underline x_{k,i'} & \text{if $j'=j$}\\
-\underline x_{ki}^{-1}\underline x_{k,i'}\underline x_{ki}^{-1}\underline x_{kj}& \text{if $j'=i$}\\
-\underline x_{ki}^{-1}\underline x_{k,i'}\underline x_{ik}^{-1}\underline x_{ij}& \text{if $j'=k$}\\
\end{cases}
=\begin{cases}
0 & \text{if $j'\notin\{i,j,k\}$}\\
q_{i,i'}^k& \text{if $j'=j$}\\
-q_{i,i'}^kq_{ij}^k& \text{if $j'=i$}\\
-q_{i,i'}^kq_{kj}^i& \text{if $j'=k$}\\
\end{cases}.
$$
The proposition is proved.
\end{proof}

For $i,j\in [n]$ define the elements $\tilde y_{ij}\in {\mathcal F}_n$ by:
$$\tilde y_{ij}=\tilde y_{i^-,j}^i=\tilde x_{i,i^-}^{-1}\tilde x_{ij}$$
(with the convention that $y_{ii}=0$). Clearly, $\tilde y_{i,i^-}=1$ and $\tilde y_{i,i^+}=\tilde x_{i,i^-}^{-1}\tilde x_{i,i^+}$.

Denote by $\overline \AA'_n$ the subalgebra of ${\mathcal Q}_n$ generated by all $\tilde y_{ij}$ and $\tilde y_{i,i^+}^{-1}$.
The following is an immediate corollary of Proposition \ref{pr:gln action}.

\begin{corollary} For each $i,j,i',j'\in [n]$ one has:
$E_{i',j'}(\tilde y_{ij})=\begin{cases}
0 & \text{if $j'\notin\{i^-,i,k\}$}\\
\tilde y_{i,i'} &\text{if $j'=j$}\\
-\tilde y_{i,i'}\tilde y_{ij} &\text{if $j'=i^-$}\\
\tilde y_{i,i'}(\tilde y_{i^-,i})^{-1}\, \tilde y_{i^-,j} &\text{if $j'=i$}\\
\end{cases}$.

In particular, $\overline \AA'_n$ is invariant under the $gl_n(\QQ)$-action.

\end{corollary}

\begin{remark} Note, however, that the subalgebra $\cU_n$ of ${\mathcal A}_n$ defined in Section \ref{subsec:regular polygon}
is not $gl_n(\QQ)$-invariant.

\end{remark}
\subsection{Extended noncommutative $n$-gons}
\label{subsec:Extended noncommutative n-gons}
In this section we define a larger algebra $\widetilde {\mathcal A}_n$ which is an extension of ${\mathcal Q}_n$ and can be viewed as a carrier of {\it double} noncommutative triangulations of the $n$-gon.

\begin{definition}
Let  ${\mathcal A}_n^\pm$ be the algebra generated by $x_{ij}^\varepsilon$ and $(x_{ij}^\varepsilon)^{-1}$, $i,j\in [n]$, $i\ne j$, $\varepsilon\in \{-,+\}$ subject to the relations:

(i) (triangle relations) For any triple $(i,j,k)$ of distinct indices in $[n]$:
\begin{equation}
\label{eq:triangle relations pm}
x_{ij}^+(x_{kj}^-)^{-1}x_{ki}^+=x_{ik}^-(x_{jk}^+)^{-1}x_{ji}^- \ .
\end{equation}

(ii) (exchange relations) For all cyclic $(i,j,k,\ell)$ in $[n]$:
\begin{equation}
\label{eq:exchange relations pm}
x_{j\ell}^-=x_{jk}^+(x_{ik}^+)^{-1}x_{i\ell}^- + x_{ji}^-(x_{ki}^+)^{-1}x_{k\ell}^+,~
x_{j\ell}^+=x_{jk}^+(x_{ik}^-)^{-1}x_{i\ell}^- + x_{ji}^-(x_{ki}^-)^{-1}x_{k\ell}^+ \ .
\end{equation}

\end{definition}

The following result is obvious.

\begin{lemma} The assignments $x_{ij}^{\,\pm}\mapsto x_{ij}$ define an epimorphism of algebras $\widetilde {\mathcal A}_n\twoheadrightarrow {\mathcal A}_n$.

\end{lemma}

\medskip

In what follows, we adopt a convention for all distinct $i,j,k\in [n]$:

\noindent $x_{ij}^k:=
\begin{cases}
x_{ij}^+ & \text{if the triangle $(i,j,k)$ is to the {\bf right} of the chord $(i,j)$ when one goes from $i$ to $j$}\\
x_{ij}^- & \text{if the triangle $(i,j,k)$ is to the {\bf left} of the chord $(i,j)$ when one goes from $i$ to $j$}
\end{cases};$

Equivalently,
$x_{ij}^k=x_{ij}^\ell$
whenever $(i,k)$ crosses $(j,\ell)$.


The following result is a generalization of Proposition \ref{pr:sector plus angle}.

\begin{proposition}
\label{pr: sector plus angle pm}

The algebra $\widetilde  {\mathcal A}_n$ is generated by
${\mathcal Q}_n$ and $( T_i^{jk})^{\pm 1}$ for all distinct triples $(i,j,k)$
subject to:

(i) (triangle relations)
$ T_i^{jk}= T_i^{kj}$ for all distinct $i,j,k\in [n]$.

(ii) (modified exchange relations)
$  T_i^{j\ell}= T_i^{jk}+ T_i^{k\ell}$ for all cyclic $(i,j,k,\ell)$ in $[n]$.

(iii) (consistency relations)
$y_{ji}^k T_i^{jk}=y_{ji}^\ell T_i^{j\ell}$
for all cyclic $(i,j,k,\ell)$ in $[n]$.

\end{proposition}

\begin{proof} We proceed similarly to the proof of Proposition \ref{pr:sector plus angle}. Denote by $\tilde \AA_n'$ the algebra whose presentation is given in the proposition.
It is easy to see that the assignments
$y_{ij}^k\mapsto (x_{ki}^j)^{-1}x_{jk}^i, T_i^{jk}\mapsto (x_{ji}^k)^{-1}x_{jk}^i(x_{ik}^j)^{-1}$
for distinct $i,j,k\in [n]$
define a homomorphism of algebras $\tilde \AA'_n\to \tilde \AA_n$.

On the other hand, the consistency relations imply that the  $( T_i^{jk})^{-1}y_{ij}^k=( T_i^{j\ell})^{-1}y_{ij}^\ell$ if $(ik)$ crosses $(j\ell)$. 

The following is immediate.

\begin{lemma} The assignments
$x_{ij}^k\mapsto ( T_i^{jk})^{-1}y_{ij}^k$ for distinct $i,j,k\in [n]$ define a homomorphism of algebras $\tilde f:\tilde \AA_n\to \tilde \AA'_n$.
\end{lemma}

In particular, $\tilde f((x_{ji}^k)^{-1}x_{jk}^i(x_{ik}^i)^{-1})=y_{ij}^k  T_j^{ik}( T_j^{ki})^{-1}y_{jk}^iy_{ki}^j T_i^{jk}=
y_{ij}^k y_{jk}^iy_{ki}^j T_i^{jk}= T_i^{jk}$ by \eqref{eq:triangle relations y ijk}.

These homomorphisms are, clearly, inverse to each other and hence are isomorphisms.

The proposition is proved.
\end{proof}

Similarly to Section \ref{subsect:Definition and basic results}, for each triangulation $\Delta$ of $[n]$ define:

$\bullet$ The subalgebra $\tilde \AA_\Delta$  of $\tilde \AA_n$ generated by $x_{k \ell}^\pm$ for all distinct $k,\ell\in [n]$ and by $(x_{ij}^\pm)^{-1}$  for all $(i,j)\in \Delta$.

$\bullet$ The group $\tilde \TT_\Delta$ generated by all $t_{ij}^\pm$ subject to the triangle relations \eqref{eq:triangle relations pm}.

Clearly, $\tilde \TT_\Delta$ is a free group on $5(n-2)$ generators. It is also clear that the assignments  $t_{ij}^\pm \mapsto x_{ij}^\pm$, $(i,j)\in \Delta$ define a homomorphism of algebras
$\tilde \ii_\Delta:\ZZ \tilde \TT_\Delta\to \tilde \AA_\Delta$.
	
The following is immediate.
	
\begin{proposition}
\label{pr:laurent phenomenon pm}
(Laurent Phenomenon for extended noncommutative polygons) For each triangulation $\Delta$ of $[n]$ the homomorpism $\tilde \ii_\Delta$ is an epimorphism.

\end{proposition}

%

Note, however, that $\tilde \ii_\Delta$ is not an isomorphsim, unlike its counterpart $\ii_\Delta$ given by \eqref{eq:i_Delta}.

\begin{proposition} For each triangulation $\Delta$ of $[n]$ the kernel of $\tilde \ii_\Delta$ contains the elements

\begin{equation}
\label{eq:plus minus relations}
(\partial_{i_4,i_1}-\partial_{i_3,i_1})T_{i_1}^{i_4,i_5} +
T_{i_2}^{i_1,i_3}(\partial_{i_1,i_4}^{-1}-\partial_{i_1,i_3}^{-1})+(t_{i_3,i_1}^-)^{-1}t_{i_3,i_4}^+(t_{i_1,i_4}^+)^{-1}-(t_{i_3,i_1}^+)^{-1}t_{i_3,i_4}^+
(t_{i_1,i_4}^-)^{-1}
\end{equation}
for each $5$-tuple $(i_1,i_2,i_3,i_4,i_5)$ in the cyclic order such that  $(i_k,i_\ell)\in \Delta$ for all distinct $(k,\ell)\in [5]\times [5]$ except for $(k,\ell)=(2,4),(4,2),(2,5),(5,2)$, where we abbreviated  $\partial_{ij}=(t_{ij}^+)^{-1}t_{ji}^-$.

\end{proposition}

\begin{proof} Without loss of generality, we assume that $i_k=k$ for $k=1,2,3,4,5$. Then
$$
x_{25}^-=x_{21}^-(x_{41}^+)^{-1}x_{45}^+ + x_{24}^+(x_{14}^+)^{-1}x_{15}^-
$$
hence
$
x_{25}^-=x_{21}^-(x_{41}^+)^{-1}x_{45}^+ +
x_{23}^+(x_{13}^-)^{-1}x_{14}^-(x_{14}^+)^{-1}x_{15}^-
+x_{21}^-(x_{31}^-)^{-1}x_{34}^+(x_{14}^+)^{-1}x_{15}^-$.
On the other hand,
$$
x_{25}^-=x_{21}^-(x_{31}^+)^{-1}x_{35}^+ + x_{23}^+(x_{13}^+)^{-1}x_{15}^-
$$
and
$$
x_{35}^+=x_{31}^-(x_{41}^-)^{-1}x_{45}^+ + x_{34}^+(x_{14}^-)^{-1}x_{15}^-
$$
hence
$
x_{25}^-=x_{21}^-(x_{31}^+)^{-1}x_{31}^-(x_{41}^-)^{-1}x_{45}^+ +
x_{23}^+(x_{13}^+)^{-1}x_{15}^-
+x_{21}^-(x_{31}^+)^{-1}x_{34}^+(x_{14}^-)^{-1}x_{15}^-$.

Comparing the expressions for $x_{25}^-$, we obtain a relation in $\tilde \AA_\Delta$ which gives the appropriate element in the kernel of $\tilde \ii_\Delta$. The proposition is proved.
\end{proof}

\begin{remark} It is natural to conjecture that the kernel of $\tilde \ii_\Delta$ is generated (as a two-sided ideal in $\ZZ \tilde \TT_\Delta$) by the elements \eqref{eq:plus minus relations}.
\end{remark}

\subsection{Further generalizations and specializations}
\begin{definition}
\label{def:hat An}
Let $\widehat {\mathcal A}_n$ be the algebra generated by all $x_{ij}^k, (x_{ij}^k)^{-1}$, where  $i,j,k$ are distinct indices in $[1,n]$ subject to the relations:

(i) (triangle relations) $ T_i^{jk}= T_i^{kj}$
for all distinct $i,j,k$, where $ T_i^{jk}=(x_{ji}^k)^{-1}x_{jk}^i(x_{ik}^j)^{-1}$.

(ii) (exchange relations) $ T_i^{j\ell}= T_i^{jk}+ T_i^{k\ell}$ whenever $(i,k)$ crosses $(j,\ell)$.

\end{definition}

The following result is obvious.

\begin{lemma} (a) The assignments $x_{ij}^k\mapsto x_{ij}$ define an epimorphism of algebras $\widehat {\mathcal A}_n\to {\mathcal A}_n$.

\noindent (b) The assignments $x_{ij}^k\mapsto x_{ij}^k$ define an epimorphism of algebras $\widehat {\mathcal A}_n\to \widetilde {\mathcal A}_n$ (as in Section \ref{subsec:Extended noncommutative n-gons}).
\end{lemma}

We refer to each $ T_{i}^{jk}$ as the generalized  noncommutative angle and view it as a certain measure of the angle at the vertex $i$ in the triangle $(ijk)$. For any triangulation  $\Delta$ of the $n$-gon and $i\in [n]$, define the {\it total angle} $ T_i^\Delta$ to be the sum of all noncommutative angles of all triangles of $\Delta$ at the vertex $i$.

\begin{theorem} For any triangulations $\Delta$ and $\Delta'$ of the $n$-gon, we have $ T_{\Delta}= T_{\Delta'}$.

\end{theorem}


Furthermore, let  ${\mathcal A}'_n$ be  the algebra generated by $x_{ij}$, $c_i^{jk}=c_i^{kj}$,
$d_i^{jk}=d_i^{kj}$ and their inverses subject to the relations:

(i) (triangle relations) $T_i^{jk}=T_i^{kj}$
for all distinct $i,j,k\in [n]$, where
$$T_i^{jk}=x_{ji}^{-1}x_{jk}x_{ik}^{-1}\ ;$$

(ii) (exchange relations) $(d_i^{j\ell})^{-1}T_i^{j\ell}(c_i^{j\ell})^{-1}
=(d_i^{jk})^{-1}T_i^{jk}(c_i^{jk})^{-1}+(d_i^{k\ell})^{-1}T_i^{k\ell}(c_i^{k\ell})^{-1}$ for any cyclic $(i,j,k,\ell)$ in $[n]$.

\begin{proposition} The assignments $x_{ij}^k\mapsto c_i^{jk}x_{ij}d_j^{ik}$ define a homomorphism  of algebras:
\begin{equation}
\label{eq:hat phi}
\hat \varphi: \hat {\mathcal A}_n\hookrightarrow {\mathcal A}'_n\ .
\end{equation}

\end{proposition}

\begin{proof} Denote by $\hat {\mathcal A}'_n$ the algebra freely generated by all $x_{ij}^k$. Then, clearly, the assignments $x_{ij}^k\mapsto c_i^{jk}x_{ij}d_j^{ik}$ define an algebra homomorphism
$$\hat {\mathcal A}'_n\to {\mathcal A}'_n\ .$$
Denote $ T_i^{'jk}:=(x_{ji}^k)^{-1}x_{jk}^i(x_{ik}^j)^{-1}$.
We need the following fact.

\begin{lemma}
$\hat \varphi'( T_i^{'jk})=(d_i^{jk})^{-1}T_i^{jk}(c_i^{jk})^{-1}$.
\end{lemma}

\begin{proof} Indeed,
$$\hat \varphi'( T_i^{'jk})=\hat \varphi'((x_{ji}^k)^{-1}x_{jk}^i(x_{ik}^j)^{-1})=(c_j^{ik}x_{ji}d_i^{jk})^{-1}c_j^{ik}x_{jk}d_k^{ij}(c_i^{jk}x_{ik}d_k^{ij})^{-1}$$
$$=(d_i^{jk})^{-1}x_{ji}x_{jk}x_{ik}(c_i^{jk})^{-1}=(d_i^{jk})^{-1}T_i^{jk}(c_i^{jk})^{-1}\ .$$
The lemma is proved.
\end{proof}

We can finish now the proof of the proposition. The lemma implies that $\hat \varphi'( T_i^{'jk})=\hat \varphi'( T_i^{'kj})$ and:
$$\hat \varphi'( T_i^{'j\ell}- T_i^{'jk}- T_i^{'k\ell})=
(d_i^{j\ell})^{-1}T_i^{j\ell}(c_i^{j\ell})^{-1}-(d_i^{jk})^{-1}T_i^{jk}(c_i^{jk})^{-1}-(d_i^{k\ell})^{-1}T_i^{k\ell}(c_i^{k\ell})^{-1}=0\ .$$
This proves the proposition.
\end{proof}

\begin{proposition} For each collection of integers $a=\{{a_i^{jk}}={a_i^{kj}}|i,j,k\in [n]~are~distinct\}$, the assignments
$x_{ij}^k\mapsto (T_i^{jk})^{a_i^{jk}} x_{ij}(T_j^{ik})^{-{a_i^{jk}}}$ define an algebra homomorphism
$\varphi_a: \hat {\mathcal A}_n\to {\mathcal A}_n$.
\end{proposition}

\begin{proof} Clearly, $\varphi_a=\psi\circ \hat \varphi$, where $\hat \varphi$ is given by \eqref{eq:hat phi} and $\psi:{\mathcal A}'_n\to {\mathcal A}_n$ is an epimorphism  given by
$$x_{ij}\mapsto x_{ij},~c_i^{jk}\mapsto (T_i^{jk})^a,~d_i^{jk}\mapsto (T_i^{jk})^{-a}\ .$$
\end{proof}

\begin{remark} Note that if $a_i^{jk}=1$, then
$\varphi_a(x_{ij}^k)=x_{ki}^{-1}x_{kj}x_{jk}x_{ik}^{-1}x_{ij}$.
\end{remark}

\subsection{Free factorizations of $\AA_n$ and proof of Theorem \ref{th:presentation y}}
\label{subsec:Free factorizations   and proof of Theorem}

 First, we verify that the relations \eqref{eq:triangle relations y ijk}, \eqref{eq:triangle relations y ijkl}, and \eqref{eq:exchange relations y} hold.
The left hand side of the first relation \eqref{eq:triangle relations y ijk} is:
$$y_{ij}^ky_{ji}^k=(x_{ki}^{-1}x_{kj})(x_{kj}^{-1}x_{ki})=1\ .$$

Furthermore, the left hand side of the second relation \eqref{eq:triangle relations y ijk} is:
$$y_{ij}^ky_{jk}^iy_{ki}^j=(x_{ki}^{-1}x_{kj})(x_{ij}^{-1}x_{ik}) (x_{jk}^{-1}x_{ji})=(x_{ki}^{-1}x_{kj}x_{ij}^{-1})(x_{ik} x_{jk}^{-1}x_{ji})=1$$
for all distinct $i,j,k\in [n]$ by the triangle relations \eqref{eq:triangle relations}. Similarly, the left hand side of \eqref{eq:triangle relations y ijkl} is:
$$y_{ij}^\ell y_{jk}^\ell y_{ki}^\ell=(x_{\ell i}^{-1}x_{\ell j}) (x_{\ell j}^{-1}x_{\ell k}) (x_{\ell k}^{-1}x_{\ell i})= 1$$
for all distinct quadruples $(i,j,k,\ell)$.

Finally, the difference between the right and left hand sides of \eqref{eq:exchange relations y} is:
$$y_{ij}^ky_{j\ell}^i+y_{i\ell}^k-y_{i\ell}^j= (x_{ki}^{-1}x_{kj}) (x_{ij}^{-1}x_{i \ell})+x_{ki}^{-1}x_{k\ell}-x_{ji}^{-1}x_{j\ell}=(x_{ji}^{-1}x_{jk}x_{kj}^{-1})x_{i \ell}+x_{ki}^{-1}x_{k\ell}-x_{ji}^{-1}x_{j\ell}$$
$$=x_{ji}^{-1}(x_{jk}x_{kj}^{-1}x_{i \ell}+x_{ji}x_{ki}^{-1}x_{k\ell}-x_{j\ell})=0$$
for all cyclic $(i,j,k,\ell)$ by the exchange relations \eqref{eq:exchange relations}.

Now let us show that the relations \eqref{eq:triangle relations y ijk}, \eqref{eq:triangle relations y ijkl}, \eqref{eq:exchange relations y} are defining. Indeed, Proposition \ref{pr:free action factorization} implies
an epimorphism of algebras $\AA_n\twoheadrightarrow {\mathcal Q}_n$ given by
$x_{ij}\mapsto y_{i^-,j}^i$. In other words, ${\mathcal Q}_n$ is isomorphic to the quotient of $\AA_n$ by the ideal generated by elements $x_{i,i^-}-1$, $i\in [n]$.

Therefore, we obtain the following obvious result.

\begin{lemma} The algebra ${\mathcal Q}_n$ is generated by all $y_{ij}:=y_{i^-,j}^i$ and $y_{ij}^{-1}$, $i,j\in [n]$, $i\ne j$,  subject to $y_{i,i^-}=1$, $i\in [i]$ and the relations \eqref{eq:triangle relations}, \eqref{eq:exchange relations}, i.e.,
\begin{equation}
\label{eq:triangle relations yij}
y_{ij}y_{kj}^{-1}y_{ki}=y_{ik}y_{jk}^{-1}y_{ji}
\end{equation}
for any distinct indices $i,j,k\in [n]$;
\begin{equation}
\label{eq:exchange relations yij}
y_{j\ell}=y_{jk}y_{ik}^{-1}y_{i\ell}+ y_{ji}y_{ki}^{-1}y_{k\ell}
\end{equation}
for all cyclic $(l,k,j,i)$ in $[n]$.
\end{lemma}

Since $y_{ij}^k=y_{ki}^{-1}y_{kj}$, the relations \eqref{eq:triangle relations y ijk} directly follow from \eqref{eq:triangle relations yij} and the relations
\eqref{eq:exchange relations y} directly follow from
\eqref{eq:exchange relations yij} (this is obvious if we ``reverse engineer" the fist part of the proof and replace all $x_{ij}$ by $y_{ij}$ there).

Therefore, Theorem \ref{th:presentation y} is proved.
\endproof

The following obvious corollary from the proof of Theorem \ref{th:presentation y} will be instrumental in Section \ref{sect:Noncommutative triangulated surfaces}.

\begin{corollary} For each triangulation $\Delta$ of $[n]$ the $\UU_\Delta$ is generated by $u_{ij}^k$,  $(i,k),(jk)\in \Delta$
subject to the relations \eqref{eq:triangle relations y ijk} and \eqref{eq:triangle relations y ijkl}, i.e.,  for all distinct $i,j,k,\ell\in [n]$ such that $(i,j),(jk)\in \Delta$ one has:
$$u_{ii}^k=1,~u_{ij}^k u_{ji}^k=u_{ij}^k u_{jk}^i u_{ki}^j,~u_{ij}^\ell u_{jk}^\ell u_{ki}^\ell=1\ .$$

\end{corollary}

\subsection{Freeness of $\TT_\Delta$  and proof of Theorem \ref{th:T_Delta polygon}}
\label{subsec:proof TDelta is free}
Let $\Delta$ be  a triangulation of $[n]$. Fix a {\it directed} triangulation $\underline \Delta\subset \Delta$ so for each $(i,j)\in \Delta$ with $j\notin \{i^+,i^-\}$ exactly one out of $(i,j)$ and $(j,i)$ belongs to  $\underline \Delta$ and $\underline \Delta$ contains all $(i,i^\pm)$, $i\in [n]$. By definition,  any such $\underline \Delta$ has cardinality $3n-3$.

\begin{proposition} Given $i_0\in [n]$. Then for any triangulation $\Delta$ and any $\underline \Delta$ as above, the group $\TT_\Delta$ is freely generated by $t_{ij}$, $(i,j)\in \underline \Delta\setminus\{(i_0,i_0^+)\}$.

\end{proposition}

\begin{proof} We proceed by induction on $n$. The assertion is obvious for $n\le 3$. Suppose that $n\ge 4$. Then it is easy to see that there exist distinct $j_0,j_0'\in [n]$ such that $(j_0^-,j_0^+),({j'}_0^-,{j'}_0^+)\in \Delta$.

Without loss of generality we may assume that $j'_0=n$ and $j_0\ne i$ (hence $j_0\notin \{i,n-1,n,1\}$). Then $\hat \Delta=\Delta\setminus\{(1,n),(n,1),(n-1,n),(n,n-1)\}$ is a triangulation of $[n-1]$ and
$\underline {\hat \Delta}=\underline \Delta\setminus\{(1,n),(n,1),(n-1,n),(n,n-1)\}$ is the corresponding directed triangulation.

The following result is obvious.

\begin{lemma}
\label{le:iota injective}
(a) The assignments $t_{ij}\mapsto t_{ij}$, $(i,j)\in \hat \Delta$, $t_{n-1,n}\mapsto 1$, $t_{n,n-1}\mapsto 1$, $t_{1,n}\mapsto 1$,
$t_{n,1}\mapsto t_{n-1,1}^{-1}t_{1,n-1}$ define an epimorphism of groups $\varphi:\TT_\Delta\twoheadrightarrow \TT_{\hat \Delta}$.

(b)  $\iota\circ \varphi=Id_{\TT_{\hat \Delta}}$, where $\iota:\TT_{\hat \Delta}\to \TT_\Delta$ is a homomorphism given by $\iota(t_{ij})=t_{ij}$ for $(i,j)\in \hat \Delta$.

(c) The homomorphism $\iota:\TT_{\hat \Delta}\to \TT_\Delta$ is injective.

\end{lemma}

Denote by $\Delta_0$ the triangulation of the triangle with the vertices $1,n-1,n$. Clearly, $\TT_\Delta$ is generated by $\TT_{\hat \Delta}$ (via the embedding $\iota$) and $\TT_{\Delta_0}$, more precisely,
$$\TT_\Delta=\TT_{\hat \Delta}*\TT_{\Delta_0}/\langle (t_{n-1,1}*1)(1*t_{n-1,1})^{-1},(t_{1,n-1}*1)(1*t_{1,n-1})^{-1}\rangle \ .$$
This and the inductive hypothesis (asserting that  $\TT_{\hat \Delta}$ is  freely generated by $t_{ij}$, $(i,j)\in \underline {\hat \Delta}\setminus\{(j_0,j_0^+)\}$)  imply (by eliminating $t_{n-1,1}$ and $t_{1,n-1}$ and setting $t_{k\ell}:=1*t_{k,\ell}$ for $(k,\ell)=(1,n), (n,1),(1,n-1),(n-1,1)$) that $\TT_\Delta$ is freely generated by all $t_{ij}$, $(i,j)\in\underline  \Delta\setminus\{(j_0,j_0^+)\}$.
\end{proof}

The theorem is proved.
\endproof

\subsection{Retraction of $\TT_n$ onto $\TT_\Delta$ and proof of Theorem \ref{th:piDelta}}
\label{subsec:Retraction of TT_n onto TT_Delta and proof of Theorem th:piDelta}
It suffices to construct an element $\tau_{ij}\in \TT_\Delta$ for each pair $(i,j)\in [n]\times [n]$, $i\ne j$ such that $\tau_{ij}=t_{ij}$ whenever  $(i,j)\in \Delta$ and for any distinct $i,j,k\in n$ one has the triangle relation:
\begin{equation}
\label{eq:triangle relation general}
\hat T_i^{jk}=\hat T_i^{kj}
\end{equation}
where $\hat T_i^{j,k}:=\tau_{ji}^{-1}\tau_{jk}\tau_{ik}^{-1}$.

We  construct such $\tau_{ij}$  by induction on $n$. Retain notation from the proof of Theorem \ref{th:T_Delta polygon} and assume, without loss of generality, that $(n-1,n+1)\in \Delta$. If $n\notin \{i,j\}$,
then, by deleting the vertex $n$ and using the natural inclusion $\TT_{\hat \Delta}\subset \TT_\Delta$ given by Lemma \ref{le:iota injective}(c),
we set $\tau_{ij}$ to be that one which belongs to $\TT_{\hat \Delta}$. Finally, we set $\tau_{1,n}:=t_{1,n}$, $\tau_{n,1}:=t_{n,1}$ and:
$$\tau_{i,n}:=\tau_{i,n-1} \tau_{1,n-1}^{-1} \tau_{1,n},~\tau_{n,i}:=\tau_{n,1} \tau_{n-1,1}^{-1} \tau_{n-1,i}$$
for $1<i<n$.

Now verify that so constructed elements satisfy
\eqref{eq:triangle relation general}. Indeed, if $i,j,k\in [n-1]$, we have nothing to prove because \eqref{eq:triangle relation general} holds  by the inductive hypothesis. Otherwise, it suffices to consider the case when $k=n$ and verify:
\begin{equation}
\label{eq:hat Tinj}
T_i^{n,j}=T_i^{j,n}
\end{equation}
for all $i,j\in [n-1]$, $i\ne j$.
Indeed,
$$\hat T_i^{n,j}=\tau_{ni}^{-1}\tau_{nj}\tau_{ij}^{-1}=\tau_{n-1,i}^{-1}\tau_{n-1,j}\tau_{ij}^{-1}=\hat T_i^{n-1,j},~\hat T_i^{j,n}=\tau_{ji}^{-1}\tau_{jn}\tau_{in}^{-1}=\tau_{ji}^{-1}\tau_{j,n-1}\tau_{i,n-1}^{-1}=\hat T_i^{j,n-1}$$
which, together with the inductive hypothesis, proves  \eqref{eq:hat Tinj}.

Therefore, the assignments $t_{ij}\mapsto \tau_{ij}$ for all $i\ne j$ define a group epimorphisms $\TT_n\to \TT_\Delta$.

Theorem \ref{th:piDelta} is proved.
\endproof

\subsection{Noncommutative Laurent Phenomenon and proof of Theorems \ref{th:noncomlaurent n-gon} and \ref{th:noncompol n-gon}}
\label{subsect:proof of Theorems th:noncomlaurent n-gon and th:noncompol n-gon}
Clearly, Theorem \ref{th:noncomlaurent n-gon} is a direct corollary of Theorem \ref{th:noncompol n-gon}, so we will only prove the latter one. We  proceed by induction on $n$.  In fact, due to the relations \eqref{eq:triangle relations y ijkl}
in the form $y_{kj}^i=y_{k,i^+}^iy_{i^+,j}^i$ (hence $y_{(k,\ii)}=y_{k,i^+}^iy_{(i^+,\ii)}$), it suffices to prove  \eqref{eq:noncompol n-gon} only with  $k=i^+$ (however, we will use the inductive hypothesis without this restriction).

Indeed, if $n\le 3$, the assertion is immediate. Now suppose that $n\ge 4$. In what follows we  retain some notation of Section \ref{subsec:proof TDelta is free}, that is, we fix a triangulation $\Delta$ and suppose that $(n-1,1)\in \Delta$ and $(j_0,j_0^+)\in \Delta$ for some $j_0\notin \{i,1,n-1,n\}$.  If $1 \notin\{i,j\}$, then the assertion \eqref{eq:noncompol n-gon} for $\Delta$ coincides with that for $\hat \Delta=\Delta\setminus\{(1,n),(n,1),(n-1,n),(n,n-1)\}$ and we have nothing to prove. Now suppose that $n\in \{i,j\}$. Without loss of generality we may assume that $i=n$ (the case $j=n$ is obtained by reversing all chords in $[n]$). Then, we will use the inductive hypothesis \eqref{eq:noncompol n-gon} for $\hat \Delta$ in the form:
$$y_{1,j}^{n-1}=\sum_{\ii'} y_{(1,\ii')},~y_{n-1,j}^1=\sum_{\ii''} y_{(n-1,\ii'')} \ ,$$
where the first (resp. the second) summation is over all $(n-1,j,\hat \Delta)$ (resp. $(1,j,\hat \Delta)$)-admissible sequences.

Using these and  the relation \eqref{eq:exchange relations y} in the form
$y_{1,j}^n=y_{1,j}^{n-1}+y_{1,n-1}^ny_{n-1,j}^1$, we obtain:
$$y_{1,j}^n=\sum_{\ii'} y_{(1,\ii')}+\sum_{\ii''} y_{1,n-1}^n y_{(n-1,\ii'')}=\sum_{\ii'} y_{(1,n,1,\ii')}+\sum_{\ii''} y_{(1,n,n-1,\ii'')}\ .$$
Clearly, this gives \eqref{eq:noncompol n-gon} because each $(n,j,\Delta)$-admissible sequence is either of the form $(n,1,\ii')$, where $\ii'$ is  $(n,j,\hat \Delta)$-admissible or is of the
form $(n,n-1,\ii'')$, where $\ii''$ is  $(1,j,\hat \Delta)$-admissible    (and vice versa).

Theorem \ref{th:noncompol n-gon} is proved.
\endproof

Therefore, Theorem \ref{th:noncomlaurent n-gon} is proved.
\endproof

\subsection{Noncommutative cluster variables and proof of Theorems \ref{th:A_n divisible} and \ref{th:A_Delta}}
\label{subsect:A_Delta}

For each triangulation $\Delta$ of $[n]$ and $(p,q)\in [n]\times [n]$, $p\ne q$ define an element
$t_{pq}^\Delta\in \QQ \TT_\Delta$ (in the notation of Theorem \ref{th:noncomlaurent n-gon}) by
\begin{equation}
\label{eq:noncomm schiffler TDelta n-gon}
t_{pq}^\Delta=\sum_{\ii\in Adm_\Delta(p,q)} t_\ii \ ,
\end{equation}
where $t_\ii\in \TT_\Delta$ is given by:
$t_\ii:=t_{i_1,i_2}t_{i_3,i_2}^{-1}t_{i_3,i_4}\cdots t_{i_{2m-1},i_{2m-2}}^{-1}t_{i_{2m-1},i_{2m}}$
for any $\ii\in [n]^{2m}$ (with the convention $t_{ii}=1$ for $i\in [n]$).

We need the following result.

\begin{theorem}
\label{th:universal localization n-gon} For any triangulations $\Delta$ and $\Delta'$ of $[n]$
the assignments $t_{ij}^{\Delta'}\mapsto t_{ij}^\Delta$ for $(i,j)\in [n]\times [n]$, $i\ne j$ define an isomorphism of algebras
\begin{equation}
\label{eq:universal localization homomorphism Delta Delta' n-gon}
\psi_{\Delta,\Delta'}:\QQ\TT_{\Delta'}[S_{\Delta'}^{-1}]\widetilde \to \QQ \TT_\Delta[S_\Delta^{-1}]\ ,
\end{equation}
where $S_\Delta$ (resp. $S_\Delta'$) is a submonoid in $\QQ \TT_\Delta$ generated by all $t_{ij}^\Delta$. These isomorphisms satisfy:
\begin{equation}
\label{eq:universal localization n-gon transitivity}
\psi_{\Delta,\Delta'}=\psi_{\Delta,\Delta''}\circ \psi_{\Delta'',\Delta'}
\end{equation}
for any triangulations $\Delta,\Delta',\Delta''$ of $[n]$.
\end{theorem}

\begin{proof}
First,  prove the assertion for {\it adjacent} triangulations $\Delta,\Delta'$ of $\Sigma$, i.e., such that $\Delta\setminus \Delta'=\{(i,k),(k,i)\}$, $\Delta'\setminus \Delta=\{(j,\ell),(\ell,j)\}$, where $(i,j,k,\ell)$ is a cyclic quadruple.

By definition,
\begin{equation}
\label{eq:tDelta jl}
t^\Delta_{j\ell}=t_{jk}t_{i k}^{-1}t_{i\ell}+
t_{ji} t_{k i}^{-1} t_{k \ell},~t^\Delta_{\ell j}=t_{\ell i}t_{ki}^{-1}t_{kj}+
t_{\ell k} t_{ik}^{-1} t_{ij} \ .
\end{equation}

We need the following result.

\begin{lemma}
\label{le:neighboring homomorphism}
For any adjacent triangulations $\Delta,\Delta'$ of $[n]$ with $\Delta\setminus \Delta'=\{(i,k),(k,i)\}$, $\Delta\setminus \Delta'=\{(j,\ell),(\ell,j)\}$
there is a unique homomorphism of algebras
$\varphi_{\Delta',\Delta}:\QQ \TT_{\Delta'}\to \QQ \TT_\Delta[(t^\Delta_{j\ell})^{-1}]$
such that
$$\varphi_{\Delta,\Delta'}(t_{i',j'})= \begin{cases}
t_{i',j'} & \text{if $\{i',j'\} \ne \{j,\ell\}$}\\
t_{j\ell}^\Delta & \text{if $(i',j')=(j,\ell)$}\\
t_{\ell j}^\Delta & \text{if $(i',j')=(\ell,j)$}\\
\end{cases}$$ for all $(i',j')\in \Delta'$.

\end{lemma}

\begin{proof}
Indeed, 
it suffices only to prove that $\varphi_{\Delta,\Delta'}$ respects the triangle relations
$$T_{i'}^{j',k'}=T_{i'}^{k',j'}$$
for all triangles $(i',j',k')$ in $\Delta'$. Clearly, if $(i',j',k')$ belongs to $\Delta\cap \Delta'$, then we have nothing to prove.
It suffices only to consider the case when $(j',k')=(j,\ell)$, i.e., we have to prove that
$$\varphi_{\Delta,\Delta'}(T_{i'}^{j\ell})=\varphi_{\Delta,\Delta'}(T_{i'}^{\ell j})$$
for $i'\in \{i,k\}$. Taking into account that both $(i'j)$ and $(i'\ell)$ belong to $\Delta\cap \Delta'$, we have only to prove that in $\QQ \TT_\Delta$ one has:
$$t_{j i'}^{-1}t_{j\ell}^\Delta t_{i'\ell}^{-1}=t_{\ell i'}^{-1}t_{\ell j}^\Delta t_{i'j}^{-1} \ .$$
In view of \eqref{eq:tDelta jl}, this is equivalent to:
\begin{equation}
\label{eq:TDelta equal}
t_{j i'}^{-1}(t_{jk}t_{i k}^{-1}t_{i\ell}+t_{ji} t_{k i}^{-1} t_{k \ell})t_{i'\ell}^{-1}=t_{\ell i'}^{-1}(t_{\ell i}t_{ki}^{-1}t_{kj}+t_{\ell k} t_{ik}^{-1} t_{ij}) t_{i'j}^{-1}\ .
\end{equation}

If $i'=i$, then both sides of \eqref{eq:TDelta equal} are, clearly, equal to $T_i^{jk}+T_i^{k\ell}$, and if $i'=k$, then both sides of
\eqref{eq:TDelta equal} are equal to $T_k^{ij}+T_k^{i\ell}$.

This proves that $\varphi_{\Delta,\Delta'}$ is well-defined homomorphism of algebras.
\end{proof}

Furthermore, we prove that in the assumptions of Lemma \ref{le:neighboring homomorphism} one has
\begin{equation}
\label{eq:phi delta delta' neighbors n-gon}
\varphi_{\Delta,\Delta'}(t_{pq}^{\Delta'})=t_{pq}^{\Delta}
\end{equation}
for all $(p,q)\in [n]\times [n]$, $p\ne q$.

Define a partial order $\prec$ on $[n]^\bullet$ by the covering insertion relations $\ii\prec \ii'$ if
\begin{equation}
\label{eq:covering relation}
\ii=(\ldots,i_k,i_{k+1},i_{k+2},\ldots), \ii'=(\ldots,i_k,i_{k+1},a,i_{k+1},i_{k+2},\ldots)
\end{equation}
for any $a\in [n]$.

We need the following obvious fact.

\begin{lemma} For each $\ii\in [n]^\bullet$ there is a unique element $[\ii]$ such that:

$\bullet$  $[\ii]\preceq \ii$.

$\bullet$ $[\ii]$ is minimal in the partial order $\prec$.

\end{lemma}

Clearly, if $\ii,\ii'\in [n]^{2\bullet}$ and $\ii\prec \ii'$, then $t_\ii=t_{\ii'}$.

Furthermore, fix a distinct quadruple $P:=(i,j,k,\ell)$ in $[n]$
and denote by ${\underline P}$  the underlying set $\{i,j,k,\ell\}$.

For any  $\ii=(i_1,\ldots,i_r)\in [n]^r$, $r\ge 2$  define the {\it index set}  $Ind_\ii(P)\subset [r-1]$ by:
$$Ind_\ii(P)=\{s\in [r-1]: \{i_s,i_{s+1}\}\in\{\{i,k\},\{j,\ell\}\}$$
(with the convention that $i_k=0$ if $k\le 0$ and $i_k=\infty$ if $k>r$)
and the {\it index}  $ind_\ii(P)\in \ZZ_{\ge 0}$ by
$$ind_\ii(P)=\min Ind_\ii(P)$$
with the convention that
$\min \emptyset:=0$.

Denote by $I_P$ the set of all sequences $\ii$ such that $|Ind_\ii(P)|=1$

Clearly, $I_{P'}=I_P$ for any permutation $P'=(i',j',k',\ell')$ of $P=(i,j,k,\ell)$ such that $\{i',k'\}\in \{\{i,k\},\{j,\ell\}\}$.

\begin{proposition}
\label{pr:minimal in I_P}
For each $\ii\in I_P$ one has $[\ii]\in I_P$ and $ind_{[\ii]}(P)\equiv ind_\ii(P) \mod 2$.

\end{proposition}

\begin{proof}
We need the following fact.

\begin{lemma} Let $\ii,\ii'\in [n]^\bullet$ be such that $\ii\prec \ii'$ and  $\ii'\in I_P$. Then $\ii\in I_P$.
\end{lemma}

\begin{proof}
It suffices to prove the assertion only for $\ii$ and $\ii'={\bf j}_{ab}^t(\ii)$ as in \eqref{eq:covering relation}. Let  $s'=ind_{\ii'}(P)$. Since $|Ind_{\ii'}(P)|=1$, and $i'_{s'-1}\neq i'_{s'+1}$, $i'_{s'}\neq i_{s'+2}$, but $i'_{t+1}=i'_{t+3}$, then $s'\notin \{t+1,t+2\}$. In particular, $\{i_t,a\}\notin \{\{i,k\},\{j,\ell\}\}$. This immediately implies that  $|Ind_\ii(P)|=1$ and
\begin{equation}
\label{eq:s versus sprime}
Ind_\ii(P)=\begin{cases}
\{s'\} & \text{if $s'\le t$}\\
\{s'-2\} & \text{if $s'\ge t+3$}\\
\end{cases}\ .
\end{equation}
The lemma is proved.
\end{proof}

Thus, for any $\ii\in I_P$ we see that $\{\ii''\in [n]^\bullet: \ii''\prec \ii\}\subset I_P$, in particular, $[\ii]\in I_P$.

The proposition is proved. \end{proof}

For $a,b\in [n]$ and $1\le s<r$ define the map
${\bf j}_{ab}^s:[n]^r\to [n]^{r+2}$ by
$(\ldots,i_s,i_{s+1},\ldots)\mapsto (\ldots,i_s,a,b,i_{s+1},\ldots)$.
Define a map $J_P:I_P\times \{-1,1\}\to [n]^\bullet \times \{-1,1\}$ by
\begin{equation}
\label{eq:map J_P}
J_P(\ii,\varepsilon)=({\bf j}_{i'k'}^s(\ii), (-1)^{(s-1)\chi_{\{i,j\}} (i_s)})
\end{equation}
where $s:=ind_\ii(P)$ and $\chi_{\{b,c\}}(a)$ is the characteristic function, i.e., it is $1$ if $a\in \{b,c\}$ and $0$ otherwise, and the pair $(i',k')$
is determined by $\{i',k'\}={\underline P}\setminus \{i_s,i_{s+1}\}$ and:
\medskip

$\bullet$ If $s$ is odd then $\{i'\}={\underline P}_\varepsilon\setminus \{i_s,i_{s+1}\}$, where we abbreviated ${\underline P}_\varepsilon:=
\begin{cases} \{i,j\}  & \text{if $\varepsilon=-1$}\\
\{k,\ell\}  & \text{if $\varepsilon=1$}
\end{cases}$.

$\bullet$  If $s$ is even then
$\{i'\}=\begin{cases} \{i_{s-1}\} & \text{if $i_{s-1}\in {\underline P} \setminus \{i_s,i_{s+1}\}$}\\
{\underline P}\setminus \{i_s,i_{s+1}, i_{s+2}\} & \text{if $i_{s+2}\in {\underline P}\setminus \{i_{s-1},i_s,i_{s+1}\}$}\\
\{i,j\}\setminus \{i_s,i_{s+1}\} &\text{ otherwise}
\end{cases}$.

Let $\hat I_P$ be the set of all $(\ii,\varepsilon)\in I_P\times \{-1,1\}$ such that

$\bullet$ if $s=ind_\ii(P)$ is even, then $\varepsilon=1$;

$\bullet$ if $s=ind_\ii(P)$ is odd then:

(i) If $\{i_{s-1}\}={\underline P}_\varepsilon\setminus\{i_s,i_{s+1}\}$, $\{i_{s+2}\}={\underline P}_{-\varepsilon}\setminus\{i_s,i_{s+1}\}$,
$i_{s-2}\ne i_s$, $i_{s+3}\ne i_{s+1}$, then $i_s\in \{i,j\}$.

(ii) If $\{i_{s-1}\}=  {\underline P}_\varepsilon\setminus\{i_s,i_{s+1}\}$, $\{i_{s+2}\}\ne {\underline P}_{-\varepsilon}\setminus\{i_s,i_{s+1}\}$, then  $i_{s-2}\ne i_{s+1}$.


\begin{proposition}
\label{pr:monotonic J_P}
$J_P(\hat I_P)\subset \hat I_P$, that is,  $J_P$ is a map $J_P: \hat I_P\rightarrow \hat I_P$.


\end{proposition}

\begin{proof} We need the following fact.

\begin{lemma}
\label{le:ind+1} Let $\ii\in I_P$ and let $s=ind_\ii(P)$. Then
\begin{equation}
\label{eq:index shifted by one}
Ind_{{\bf j}_{i'k'}^s(\ii)}(P)=\{ind_\ii(P)+1\}
\end{equation}
for any $i',k'\in [n]$ such that  $\{i',k'\}=\underline P\setminus \{i_s,i_{s+1}\}$.
\end{lemma}

\begin{proof} Let $s=ind_\ii(P)$ and $\ii':={\bf j}_{i'k'}^s(\ii)$.
 Note that $s+1\in Ind_{\ii'}(P)$ because  $\{i'_{s+1}, i'_{s+2}\}\in \{\{i,k\},\{j,\ell \}\}$.
This and the fact that $\{i'_s, i'_{s+1}, i'_{s+2}, i'_{s+3}\}={\underline P}$ imply that  $s\notin Ind_{\ii'}(P)$ and $s+2\notin Ind_{\ii'}(P)$. Finally, if $s''\le s-1$ (res. $s''\ge s+3$), then $s'' \notin Ind_{\ii'}(P)$ because $s''\notin Ind_\ii(P)$ (resp. because $s''-2\notin Ind_\ii(P)$).

This proves \eqref{eq:index shifted by one}.
\end{proof}

Furthermore, let $(\ii,\varepsilon)\in \hat I_P$, $(\ii',\varepsilon'):= J_P(\ii,\varepsilon)$, $s:=Ind_\ii(P)$, $s'=Ind_{\ii'}(P)$. By Lemma \ref{le:ind+1},
$s'=s+1$. This, in particular, implies that  $\varepsilon'=(-1)^{(s-1)\chi_{\{i,j\}} (i_s)}\in \{1,(-1)^{s'}\}$.
If $s$ is odd, this proves the desired inclusion $J_P(\ii,\varepsilon)\in \hat I_P$.

It remains to consider the case when $s$ is even. Indeed, $\ii'={\bf j}_{i',k'}(\ii)$, where $i'=i'_{s'}$, $k'=i'_{s'+1}$ are given by the even case of \eqref{eq:map J_P}.
Note that
\begin{equation}
\label{eq:P_epsilon'}
P_{\varepsilon'}=\begin{cases}
 \{i,j\} & \text{if $i_s\in \{i,j\}$}\\
 \{k,\ell\} & \text{if $i_s\in \{k,\ell\}$}\\
\end{cases},P_{-\varepsilon'}=\begin{cases}
 \{i,j\} & \text{if $i_{s+1}\in \{i,j\}$}\\
 \{k,\ell\} & \text{if $i_{s+1}\in \{k,\ell\}$}\\
\end{cases}
\end{equation}
hence
$\{i_s\}=\{i'_{s'-1}\}= P_{\varepsilon'}\setminus\{i'_{s'},i'_{s'+1}\}$, $\{i_{s+1}\}=\{i'_{s'+2}\}=P_{-\varepsilon'}\setminus\{i'_{s'},i'_{s'+1}\}$.

Finally,
 $i'_{s'-2}\ne i'_{s'}$ and $i'_{s'+1}\ne i'_{s'+3}$ if and only if $\{i_{s-1},i_{s+2}\}\cap {\underline P}=\emptyset$ hence $\{i'\}=\{i,j\}\setminus \{i_s,i_{s+1}\}$.

This proves that $J_P(\ii,\varepsilon)\in \hat I_P$ for even $s$ as well.

The proposition is proved.
\end{proof}

Denote by $[I_P]\subset I_P$ the set of all $\ii \in \hat I_P$ such that $\ii=[\ii]$ is minimal in the partial order $\prec$ and abbreviate $[\hat I_P]:=\hat I_P\cap ([I_P]\times \{-1,1\})$.

Proposition \ref{pr:minimal in I_P} guarantees that the assignments $\ii\mapsto [\ii]$
define a projection $I_P\to [I_P]$ (resp. $\hat I_P\to [\hat I_P]$).

\begin{proposition}
\label{pr:minimal involution}
The assignments $(\ii,\varepsilon)\mapsto [J_P(\ii,\varepsilon)]$ define an involution  $[J_P]:\hat I_P\to \hat I_P$.
\end{proposition}

\begin{proof} Let $(\ii,\varepsilon)\in [\hat I_P]$, let $s:=ind_\ii(P)$, $(\ii',\varepsilon'):=[J_P(\ii,\varepsilon)]$, $s':=ind_{\ii'}(P)$. By definition,
\begin{equation}
\label{eq:ii' cases}
\ii'=[(\ldots,i_s,i',k',i_{s+1},\ldots)]=
\begin{cases}
(\ldots,i_s,i',k',i_{s+1},\ldots) & \text{if $i'\ne i_{s-1}$, $k'\ne i_{s+2}$}\\
(\ldots,i_{s-1},i_{s+2},\ldots) & \text{if $i'= i_{s-1}$, $k'= i_{s+2}$}\\
(\ldots,i_{s-1},i_s,i',i_{s+2},\ldots) & \text{if $i'\ne i_{s-1}$, $k'= i_{s+2}$}\\
(\ldots,i_{s-1},k',i_{s+1},i_{s+2},\ldots)& \text{if $i'= i_{s-1}$, $k'\ne i_{s+2}$}\\
\end{cases}
\end{equation}

in the notation \eqref{eq:map J_P}. In particular, $i'_{s'}=i'$, $i'_{s'+1}=k'$.

Note that, by Lemma \ref{le:ind+1} and Proposition \ref{pr:minimal in I_P}, $s'\equiv s+1\mod 2$.

First, show that $(\ii',\varepsilon')\in [\hat I_P]$ (i.e., $[J_P]$ is well-defined). If $s$ is odd, this is obvious. Suppose that $s$ is even. Then we have in each of the cases of \eqref{eq:ii' cases}:

$\bullet$  $i'\ne i_{s-1}$, $k'\ne i_{s+2}$.
Since $s'=s+1$ and $\{i'_{s'-1},i'_{s'},i'_{s'+1},i'_{s'+2}\}=\underline P$ and $i'\in \{i,j\}$, clearly, $(\ii',\varepsilon')\in [\hat I_P]$.

$\bullet$ $i'= i_{s-1}$, $k'= i_{s+2}$. Since $s'=s-1$ and $\{i'_{s'-1},i'_{s'+2}\}\cap {\underline P}=\emptyset$, clearly, $(\ii',\varepsilon')\in [\hat I_P]$.

$\bullet$ $i'\ne i_{s-1}$, $k'= i_{s+2}$. Since $s'=s+1$ and $\{i_s\}=\{i'_{s'-1}\}= {\underline P}_{\varepsilon'}\setminus \{i'_{s'},i'_{s'+1}\}$, $\{i'_{s'+2}\}=\{i_{s+3}\}\ne \{i_{s+1}\}= {\underline P}_{-\varepsilon'}\setminus \{i'_{s'},i'_{s'+1}\}$ by \eqref{eq:P_epsilon'} and $i'_{s'+1}=i_{s+2}\ne i_{s-1}=i'_{s'-2}$, clearly,  $(\ii',\varepsilon')\in [\hat I_P]$.

$\bullet$ $i'=   i_{s-1}$, $k'\ne i_{s+2}$.  Since $s'=s-1$ and
$\{i'_{s'+2}\}= {\underline P}_{-\varepsilon'}\setminus \{i'_{s'},i'_{s'+1}\}$ by \eqref{eq:P_epsilon'},
clearly, $(\ii',\varepsilon')\in [\hat I_P]$.


\medskip

Furthermore, let $(\ii'',\varepsilon'')=J_P(\ii',\varepsilon')$. That is,
$$\ii''={\bf j}_{i'',k''}^{s'}(\ii')\ ,$$
where $\varepsilon''=(-1)^{(s'-1)\chi_{\{i,j\}} (i'_{s'})}$, $\{i'',k''\}=\{i_s,i_{s+1}\}$ and one has (note that $\{i'_{s'},i'_{s'+1}\}=\{i',k'\}$):

$\bullet$ If $s$ is even,
then
$\{i'\}=\begin{cases} \{i_{s-1}\} & \text{if $i_{s-1}\in {\underline P} \setminus \{i_s,i_{s+1}\}$}\\
{\underline P}\setminus \{i_s,i_{s+1}, i_{s+2}\} & \text{if $i_{s+2}\in {\underline P}\setminus \{i_{s-1},i_s,i_{s+1}\}$}\\
\{i,j\}\setminus \{i_s,i_{s+1}\} &\text{ otherwise}
\end{cases}$,
$\varepsilon'=(-1)^{\chi_{ij}(i_s)}$, and:
\begin{equation}
\label{eq:universal i'k'}
\{i''\}=
{\underline P}_{\varepsilon'} \setminus \{i'_{s'},i'_{s'+1}\}={\underline P}_{\varepsilon'} \setminus \{i',k'\}=\begin{cases} \{i,j\} \setminus \{i'\} & \text{if $i_s\in \{i,j\}$}\\
\{k,\ell\} \setminus \{k'\} & \text{if $i_s\in \{k,\ell\}$}
\end{cases}=\{i_s\} \ .
\end{equation}

$\bullet$  If $s$ is odd, then: $\{i'\}={\underline P}_\varepsilon\setminus  \{i_s,i_{s+1}\}$,
\begin{equation}
\label{eq:universal i'k' odd}
\{i''\}=\begin{cases} \{i'_{s'-1}\} & \text{if $i'_{s'-1}\in {\underline P} \setminus \{i',k'\}$}\\
{\underline P}\setminus \{i',k', i'_{s'+2}\} & \text{if $i'_{s'+2}\in {\underline P}\setminus \{i'_{s'-1},i',k'\}$}\\
\{i,j\}\setminus \{i',k'\} &\text{ otherwise}
\end{cases}.
\end{equation}

First, show that $\varepsilon''=\varepsilon$. Indeed, by the above, $\varepsilon''=(-1)^{s\chi_{\{i,j\}} (i')}$.
Since $\varepsilon\in \{1,(-1)^s\}$, then the above implies that for even $s$ one has $\varepsilon''=\varepsilon=1$. If $s$ is odd,
then, by definition, $i'\in \{i,j\}$ iff $\varepsilon=-1$. This proves that $\varepsilon''=\varepsilon$ in this case as well.

Thus, it remains to prove that
\begin{equation}
\label{eq:ii prec ii'}
\ii\preceq \ii''\ .
\end{equation}

\medskip

To do so, show that $i''=i_s$ in each case of \eqref{eq:ii' cases}:

$\bullet$ $i'\ne i_{s-1}$, $k'\ne i_{s+2}$, $s'=s+1$, $\ii''=(\ldots,i_s,i',i'',k'',k',i_{s+1},\ldots)$,
where for even $s$ we have $i''=i_s$  by \eqref{eq:universal i'k'} and for odd $s$ we also have $i''=i_s$ by \eqref{eq:universal i'k' odd} because $i'_{s'-1}=i_s$ and $i'_{s'+2}=i_{s+1}$.

$\bullet$ $i'\ne  i_{s-1}$, $k'= i_{s+2}$, $s'=s+1$, $\ii''=(\ldots,i_s,i', i'',k'',i_{s+2},\ldots)$,
where for even $s$, $i''=i_s$
by \eqref{eq:universal i'k'} and for odd $s$ we have
$\{i''\}=
{\underline P}\setminus \{i',k'\} =\{i_s\}
$ by \eqref{eq:universal i'k' odd} because $i_s=i'_{s'-1}\in {\underline P}\setminus \{i',k'\}$.

$\bullet$ $i'= i_{s-1}$, $k'= i_{s+2}$, e.g.,  $\{i_{s-1},i_{s+2}\} ={\underline P}\setminus \{i_s,i_{s+1}\}$, $s'=s-1$,
$\ii''=(\ldots,i_{s-1},i'',k'',i_{s+2},\ldots)$, where for even $s$,  $i''=i_s$
by \eqref{eq:universal i'k'} and for odd $s$ we have: $i_{s-1}\in P_\varepsilon$, $i_{s+2}\in P_{-\varepsilon}$,
$i_{s-2}\ne i_s$, $i_{s+3}\ne i_{s+1}$ hence $i_s\in \{i,j\}$
and:
$\{i''\}=
\{i,j\}\setminus \{i_{s-1},i_{s+2}\}=i_s$
by \eqref{eq:universal i'k' odd}.


$\bullet$  $i'= i_{s-1}$, $k'\ne i_{s+2}$, $s'=s-1$, $\ii''=(\ldots,i_{s-1},i'',k'',k',i_{s+1},i_{s+2},\ldots)$,
where for even $s$, $i''=i_s$
by \eqref{eq:universal i'k'} and for odd $s$ we have $\{i''\}=
{\underline P}\setminus \{i',k', i'_{s'+2}\}=\{i_s,i_{s+1}\}\setminus \{i_{s+1}\} =\{i_s\}$
by \eqref{eq:universal i'k' odd} because:

$\bullet$  $i_{s-1}\in P_\varepsilon\setminus\{i_s,i_{s+1}\}$, $i_{s+2}\notin P_{-\varepsilon}\setminus\{i_s,i_{s+1}\}$
hence  $i_{s-2}\ne i_{s+1}$.

$\bullet$ $i_{s+1}=i'_{s'+2}\in {\underline P}\setminus \{i'_{s'-1},i',k'\}=\{i_s,i_{s+1}\}\setminus \{i_{s-2}\}$.

Thus, $i''=i_s$, $k''=i_{s+1}$ in all cases, which immediately implies \eqref{eq:ii prec ii'} in all these cases.

This proves that $[J_P]$ is an involution on $[\hat I_P]$.

The proposition is proved.
\end{proof}

Now suppose that $P=(i,j,k,\ell)$ where $\Delta\setminus \Delta'=\{(i,k),(k,i)\}$,  $\Delta'\setminus \Delta=\{(j,\ell),(\ell,j)\}$,
as in Lemma \ref{le:neighboring homomorphism}. In what follows, we assume that $(p,i)\cap (j,\ell)=\emptyset$  and $(p,q)\cap (i,j)\ne \emptyset$
(i.e., informally speaking,  $(i,j)$ is closer to $p$ than $(k,\ell)$).


By Definition \ref{def:admissible} of admissible sequences, if $\ii\in  Adm_\Delta(p,q)\subset [I_P]\sqcup Adm_{\Delta'}(p,q)$ then $[\ii]=\ii$ is minimal, its index $s:=ind_\ii(P)$ is positive and unique, and
$\{i_s,i_{s+1}\}=
\begin{cases}
\{i,k\} & \text{if $\ii\in  Adm_\Delta(p,q)$}\\
\{j,\ell\} & \text{if $\ii\in Adm_{\Delta'}(p,q)$}\\
\end{cases}$.

\begin{proposition}
\label{pr:combin bijection Delta Delta'}
Let $\Delta,\Delta'$ be triangulations of $[n]$ and $P=(i,j,k,\ell)$ as above.
Then the restriction of  $[J_P]$ to $(Adm_{\Delta'}(p,q)\times \{-1,1\})\cap [\hat I_P]$ is a bijection:
\begin{equation}
\label{eq:J_Delta,Delta'}
J_{\Delta,\Delta'}:(Adm_{\Delta'}(p,q)\times \{-1,1\})\cap [\hat I_P]\widetilde \to (Adm_{\Delta}(p,q)\times \{-1,1\})\cap [\hat I_P]
\end{equation}

\end{proposition}

\begin{proof}
We need the following  obvious fact.

\begin{lemma} Let $\ii\in Adm_\Delta(p,q)\sqcup Adm_{\Delta'}(p,q)$ such that $s:=ind_\ii(P)>0$. Then

(i) If $s$ is even, then $(\ii,1)$  belongs to $[\hat I_P]$.

(ii) If $s$ is odd, then both $(\ii,1)$ and $(\ii,-1)$ belong to $[\hat I_P]$.
\end{lemma}

Furthermore, for any triangulation $\Delta$ of $[n]$ and any $p,q\in [n]$ denote by $PreAdm_\Delta(p,q)$ the set of all $\ii\in [n]^\bullet$
such that $[\ii]\in Adm_\Delta(p,q)$. We need the following fact.

\begin{lemma}
\label{le:admissible bf j}
In the assumptions of Proposition \ref{pr:combin bijection Delta Delta'}, let  $\ii\in  Adm_{\Delta'}(p,q)$ and suppose that $s=ind_\ii(P)>0$. Then:

(a) if $s$ is odd, then  ${\bf j}_{i',k'}^s(\ii)\in PreAdm_\Delta(p,q)$
whenever $\{i_s,i_{s+1},i',k'\}=\{i,j,k,\ell\}$.

(b) if $s$ is even  then  $[J_P(\ii,1)]\in Adm_\Delta(p,q)\times \{-1,1\}$.

\end{lemma}

\begin{proof}
In what follows, we will write ${\bf p}\le {\bf p}'$ for any points ${\bf p},{\bf p}'$ in the chord $(p,q)$ such that either  ${\bf p}= {\bf p}'$ or ${\bf p}$ is closer to $p$ than ${\bf p'}$.


Prove (a). Indeed, it suffices to show that for $\ii=(i_1,\dots,i_{2m})\in Adm_{\Delta'}(p,q)$, one has
\begin{equation}
\label{eq:preadm}
\ii':=(\ldots,i_s,i',k',i_{s+1},\ldots)\in PreAdm_\Delta(p,q)\ ,
\end{equation}
where $s=ind_\ii(P)$ is odd (note that  $\{i_s,i_{s+1}\}=\{j,\ell\}$ and $\{i',k'\}=\{i,k\}$).

Let ${\bf p}_-$ and ${\bf p}_+$ be  the intersection points of $(p,q)$ respectively with  $(i_{s-1},i_s)$  and $(i_{s+1},i_{s+2})$
(with the convention that ${\bf p}_-=p$ if $s=1$ and ${\bf p'}=q$ if $s=2m-1$). Clearly, ${\bf p}_-<{\bf p}_+$.

We now consider a number of cases.

{\bf Case 1}. Suppose that $(p,q)\cap (i_s,i_{s+1})\ne \emptyset$,  $3\le s\le 2m-3$ (i.e., $\{p,q\}\cap \{i,j,k,\ell\}=\emptyset$).
Since $(i_r,i_{r+1})\in \Delta$
for $r=s-1,s,s+1$, the above and convexity of the $n$-gon $[n]$ imply that there exist $i'',k''\in [n]$ such that $\{i'',k''\}=\{i',k'\}$ and $(p,q)\cap (i_s,i'')\ne \emptyset$, $(p,q)\cap (i_s,k'')\ne \emptyset$, $(p,q)\cap (i',k')\ne \emptyset$ and
$${\bf p}_-\le (p,q)\cap (i_s,i'')<(p,q)\cap (i',k')<(p,q)\cap (i_s,k'') \le {\bf p}_+ \ .$$
In turn, this immediately implies \eqref{eq:preadm} in this case.


{\bf Case 2}. Suppose that $(p,q)\cap (i_s,i_{s+1})= \emptyset$, $3\le s\le 2m-3$. By definition, ${\bf p}_-<{\bf p}_0<{\bf p}_+$.
Then the convexity of the $n$-gon $[n]$ implies that there exist $i'',k''\in [n]$ such that $\{i'',k''\}=\{i',k'\}$ and $(p,q)\cap (i_s,k'')= \emptyset$, $(p,q)\cap (i_{s+1},k'')= \emptyset$. This and
the facts that $(i'',k'')\cap (i_s,i_{s+1})\ne \emptyset$ and that $i''$ does not belong to the convex hull of
${\bf p}_-$, ${\bf p}_+$, $i_s$, $i_{s+1}$ imply that $(p,q)\cap (i_s,i'')\ne \emptyset$, $(p,q)\cap (i_{s+1},i'')\ne \emptyset$, $(p,q)\cap (i',k')\ne \emptyset$ and
$${\bf p}_-\le (p,q)\cap (i_s,i'')<(p,q)\cap (i',k')<(p,q)\cap (i_{s+1},k'') \le {\bf p}_+\ .$$
In turn, this immediately implies \eqref{eq:preadm} in this case.

{\bf Case 3}. Suppose that $s=1$ or $s=2m-1$. If $s=1=2m-1$, we have nothing to prove because $\ii=(i_1,i_2)=(p,q)$, $\ii'=(p,i',k',q)\in Adm_\Delta(p,q)$. Therefore it remains to consider the sub-case when $s=1$, $m\ge 2$ (the sub-case $s=2m-1\ge 3$ is identical to it). Indeed,  the facts that $(i',k')\cap (i_1,i_2)\ne \emptyset$ implies that there exist  $i'',k''\in [n]$ such that $\{i'',k''\}=\{i',k'\}$ and
$(p,q)\cap (i_1,i'')= \emptyset$. This and
the facts that $(i'',k'')\cap (i_1,i_2)\ne \emptyset$ and that $k''$ does not belong to the convex hull of
${\bf p}_-=p=i_1$, $i_2$ ${\bf p}_+$ imply that $(p,q)\cap (i_2,k'')\ne \emptyset$, $(p,q)\cap (i_{s+1},i'')\ne \emptyset$, $(p,q)\cap (i',k')\ne \emptyset$ and
$$(p,q)\cap (i',k')<(p,q)\le {\bf p}_+\ .$$
In turn, this immediately implies \eqref{eq:preadm} in this case.

This finishes the proof of part (a).

Prove (b) now. That is, we have to show that
\begin{equation}
\label{eq:preadm even}
\ii':=[(\ldots,i_s,i',k',i_{s+1},\ldots)]\in Adm_\Delta(p,q)\ ,
\end{equation}
where $s=ind_\ii(P)$ is even  and $i',k'$ are as in \eqref{eq:map J_P} (note that $\{i_s,i_{s+1}\}=\{j,\ell\}$ and $\{i',k'\}=\{i,k\}$).

Denote ${\bf p}_0:=(p,q)\cap (i_s,i_{s+1})$ and consider a number of cases.

{\bf Case 1}. Suppose that $\{i_{s-1},i_{s+2}\}=\{i,k\}$.
Then $i'=i_{s-1}$, $k'=i_{s+2}$ by \eqref{eq:map J_P} and
$$\ii'=(\ldots,i_{s-1},i_{s+2},\ldots)\ ,$$
i.e., $\ii'$ is obtained from $\ii$ by simultaneously replacing $i_s$ with $i_{s-1}$ and $i_{s+1}$ with $i_{s+2}$. This  immediately implies  \eqref{eq:preadm even} in this case.

{\bf Case 2}. Suppose that $i_{s+2}\in \{i,k\}$, $i_{s-1}\notin \{i,k\}$ (the case $i_{s-1}\in \{i,k\}$, $i_{s+2}\notin \{i,k\}$ is identical to it).
Then $k'=i_{s+2}$ by \eqref{eq:map J_P} and
$$\ii'=(\ldots,i_s,i',i_{s+2},\ldots)\ ,$$
i.e., $\ii'$ is obtained from $\ii$ by replacing $i_{s+1}$ with $i'$. Thus, to prove \eqref{eq:preadm even}, it suffices to show that $(p,q)\cap (i_s,i')\ne \emptyset$.
Indeed, suppose that  $(p,q)\cap (i_s,i')\ne \emptyset$. If $s=2$,  $i_{s-1}=p\notin \{i,k\}$, then taking into account that $(i_s,i')\in \Delta$, we see that $i'$ belongs to the interior of the convex hull of $p,{\bf p}_0, i_s$.  If  $s\ge 4$, $(i_{s-2},i_{s-1})\in \Delta$, $(p,q)\cap (i_{s-2},i_{s-1})\ne \emptyset$, then $i'$  belongs to the interior of the convex hull of $p,{\bf p}_0, i_{s-1},i_s$. This contradicts to that $i'$ is a vertex of the convex $n$-gon $[n]$, which  immediately implies \eqref{eq:preadm even} in this case.

{\bf Case 3}. Suppose that $\{i_{s-1},i_{s+2}\}\cap  \{i,k\}=\emptyset$
Then $i'=i$, $k'=k$ by \eqref{eq:map J_P} and
$$\ii'=(\ldots,i_s,i,k,i_{s+2},\ldots)\ .$$

Thus, to prove \eqref{eq:preadm even}, it suffices to show that $(p,q)\cap (i_s,i)\ne \emptyset$, $(p,q)\cap (k,i_{s+1})\ne \emptyset$.
Since $(p,i)\cap (j,\ell)= \emptyset$, using the same argument as in {\bf Case 2},  we see that if $(p,q)\cap (i_s,i)= \emptyset$, then $i'$  belongs to the interior of the convex hull of $p,{\bf p}_0, i_{s-1},i_s$; and if $(p,q)\cap (k,i_{s+1})=\emptyset$, then $k$  belongs to the interior of the convex hull of $q,{\bf p}_0, i_s,i_{s+1}$. This  finishes the proof of  \eqref{eq:preadm even} in this case.

This finishes the proof of (b).

Lemma \ref{le:admissible bf j} is proved.
\end{proof}

Using Lemma \ref{le:admissible bf j}(b) with  $P=(i,j,k,\ell)$ such that $(p,i)\cap (j,\ell)=\emptyset$ and $(p,q)\cap (i,j)\ne \emptyset$  and Lemma \ref{le:admissible bf j}(a) with any $i',k'$ such that $\{i',k'\}=\{i,k\}$, we see that
$$[J_P]((Adm_{\Delta'}(p,q)\times \{-1,1\})\cap [\hat I_P])\subset  (Adm_{\Delta}(p,q)\times \{-1,1\})\cap [\hat I_P]$$
hence $J_{\Delta,\Delta'}$ given by \eqref{eq:J_Delta,Delta'}  is a well-defined map
$$(Adm_{\Delta'}(p,q)\times \{-1,1\})\cap [\hat I_P]\hookrightarrow   (Adm_\Delta(p,q)\times \{-1,1\})\cap [\hat I_P]\ .$$
Interchanging $\Delta$ and $\Delta'$, taking into account  that $(p,j)\cap (i,k)=\emptyset$,  and applying Lemma \ref{le:admissible bf j} again,  we see that
$$[J_P]((Adm_\Delta(p,q)\times \{-1,1\})\cap [\hat I_P])\subset  (Adm_{\Delta'}(p,q)\times \{-1,1\})\cap [\hat I_P] \ .$$
This gives a well-defined  map
$$J_{\Delta',\Delta}:(Adm_\Delta(p,q)\times \{-1,1\})\cap [\hat I_P]\hookrightarrow   (Adm_{\Delta'}(p,q)\times \{-1,1\})\cap [\hat I_P]\ .$$
Since $[J_P]$ is an involution by Proposition \ref{pr:minimal involution}, the maps $J_{\Delta,\Delta'}$ and $J_{\Delta',\Delta}$ are inverse of each other, hence each of them is a bijection.

Proposition \ref{pr:combin bijection Delta Delta'} is proved.
\end{proof}

Furthermore, we need the following obvious fact.

\begin{lemma}
\label{le:long ptolemy}
In the assumptions of Lemma \ref{le:neighboring homomorphism} let $s\in [2m-1]$ be odd and let $\ii=(i_1,\ldots,i_{2m})\in [n]^{2m}$,
$m\ge 1$ be such that $\{i_{s'},i_{s'+1}\}\ne \{j,\ell\}$
for $r\in [2m-1]\setminus \{s\}$.

(a) If $\{i_s,i_{s+1}\}=\{j,\ell\}$ then
$\varphi_{\Delta,\Delta'}(t_\ii)= t_{{\bf j}_{ik}^s(\ii)}+t_{{\bf j}_{ki}^s(\ii)}$.

(b) If  $\{i_s,i_{s+1}\}=\{i,k\}$, then $\varphi_{\Delta,\Delta'} (t_{{\bf j}_{j\ell}^s(\ii)}+t_{{\bf j}_{\ell j}^s(\ii)})=t_\ii$.
\end{lemma}

Now we are ready to prove \eqref{eq:phi delta delta' neighbors n-gon}. Indeed,
$t^{\Delta'}_{pq}=t_0+t_-+t_+$, where
$$t_0=\sum_{\ii'\in Adm_{\Delta'}(p,q): ind_{\ii'}(P)=0} t_{\ii'},~
t_-=\sum_{\ii'\in Adm_{\Delta'}(p,q): ind_{\ii'}(P)\in 2\ZZ+1} t_{\ii'},~t_+=\sum_{\ii'\in Adm_{\Delta'}(p,q): ind_{\ii'}(P)\in 2\ZZ_{\ge 1}} t_{\ii'}\ .$$
Clearly,
$$\varphi_{\Delta,\Delta'}(t_0)=t_0=\sum_{\ii\in Adm_{\Delta}(p,q): ind_\ii(P)=0} t_\ii \ .$$
Furthermore, combining Proposition \ref{pr:combin bijection Delta Delta'} and Lemma \ref{le:long ptolemy}, we obtain:
$$\varphi_{\Delta,\Delta'}(t_-)=\sum_{\ii'\in Adm_{\Delta'}(p,q): ind_\ii(P)\in 2\ZZ+1} t_{J_{\Delta,\Delta'}(\ii,1)}+t_{J_{\Delta,\Delta'}(\ii,-1)}=\sum_{\ii\in Adm_{\Delta}(p,q): ind_\ii(P)\in 2\ZZ_{\ge 1}} t_{\ii}\ ,$$
$$\varphi_{\Delta,\Delta'}(t_+)=\sum_{\ii\in Adm_\Delta(p,q): ind_\ii(P)\in 2\ZZ+1} \varphi_{\Delta,\Delta'}(t_{J_{\Delta',\Delta}(\ii,1)}
+t_{J_{\Delta',\Delta}(\ii,-1)})
=\sum_{\ii\in Adm_\Delta(pq): ind_\ii(P)\in 2\ZZ+1} t_{\ii}\ .
$$

This finishes the proof of \eqref{eq:phi delta delta' neighbors n-gon}.

Furthermore, we define a homomorphism $\psi_{\Delta,\Delta'}$ as follows. First, composing $\varphi_{\Delta,\Delta'}$ with  the universal localization by $S_\Delta$ and taking into the account that
$t^\Delta_{j\ell}\in S_\Delta$, we obtain a homomorphism of algebras:
$$\varphi'_{\Delta,\Delta'}:\QQ \TT_{\Delta'}\to \QQ \TT_\Delta[(t^\Delta_{j\ell})^{-1}] $$
such that $\varphi'_{\Delta,\Delta'}(t_{ij}^{\Delta'})=t_{ij}^\Delta$ for all $i,j$. Since $t_{ij}^\Delta\in S_\Delta$ is invertible in the image, $\varphi'_{\Delta,\Delta'}$ canonically extends to a homomorphisms of algebras
$$\psi_{\Delta,\Delta'}:\QQ \TT_{\Delta'}[S_{\Delta'}^{-1}]\to \QQ \TT_\Delta[S_\Delta^{-1}]\ .$$
Switching $\Delta$ and $\Delta'$ we obtain a homomorphism $\psi_{\Delta',\Delta}:\QQ \TT_\Delta[S_\Delta^{-1}]\to \QQ \TT_\Delta[S_{\Delta'}^{-1}]$,
which is, clearly, inverse of $\psi_{\Delta,\Delta'}$.

This proves Theorem  \ref{th:universal localization n-gon} for neighboring triangulations $\Delta,\Delta'$.

Now we prove Theorem  \ref{th:universal localization n-gon}  for any (non-neighboring) triangulations $\Delta,\Delta'$ of $[n]$ as follows. We say that the distance $dist(\Delta,\Delta')$ is the minimal number $d\ge 0$ such that there is a sequence of triangulations
$\Delta=\Delta^{(0)},\Delta^{(1)},\ldots,\Delta^{(d)}=\Delta'$ of $[n]$ such that $\Delta^{(s)},\Delta^{(s+1)}$, $s\in [r-1]$ are neighboring.

We construct appropriate $\varphi_{\Delta,\Delta'}$ by induction in $dist(\Delta,\Delta')$. If $dist(\Delta,\Delta')=1$, then $\Delta$ and $\Delta'$ are neighboring and we have nothing to prove. Suppose that $d=dist(\Delta,\Delta')>1$.
Then there is a triangulation $\Delta''$ of $[n]$ with $dist(\Delta,\Delta'')<d$ and $dist(\Delta'',\Delta')<d$.

By the inductive hypothesis, there are isomorphisms
$$\psi_{\Delta,\Delta''}:\QQ \TT_{\Delta''}[S_{\Delta''}^{-1}]\to \QQ \TT_{\Delta}[S_\Delta^{-1}],~\psi_{\Delta'',\Delta'}:\QQ \TT_{\Delta'}[S_{\Delta'}^{-1}]\to \QQ \TT_{\Delta''}[S_{\Delta''}^{-1}]$$
 such that $\psi_{\Delta,\Delta''}(t_{ij}^{\Delta''})=t_{ij}^\Delta$ and $\psi_{\Delta'',\Delta'}(t_{ij}^{\Delta'})=t_{ij}^{\Delta''}$  for all $i,j$.

Define $\psi_{\Delta,\Delta'}:=\psi_{\Delta,\Delta''}\circ \psi_{\Delta'',\Delta'}$. By definition, $\psi_{\Delta,\Delta'}$ is an isomorphism $\QQ \TT_{\Delta'}[S_{\Delta'}^{-1}]\to \QQ \TT_{\Delta}[S_\Delta^{-1}]$ such that $\psi_{\Delta,\Delta''}(t_{ij}^{\Delta'})=t_{ij}^\Delta$ for all $i,j$. In particular, $\psi_{\Delta,\Delta'}$ does not depend on the choice of $\Delta''$. This finishes the induction.

The transitivity \eqref{eq:universal localization n-gon transitivity} also follows.

Theorem  \ref{th:universal localization n-gon} is proved.
\end{proof}










Furthermore, we need the following result.

\begin{proposition}
\label{pr:canonical isomorphism localization polygon}
In the notation of Theorem \ref{th:universal localization n-gon}, for each triangulation $\Delta$ of $[n]$
the homomorphism $\ii_\Delta:\QQ \TT_\Delta\to \AA_\Delta\subset \AA_n$ given by \eqref{eq:i_Delta} extends to an isomorphism of algebras
$\QQ \TT_{\Delta}[S_\Delta^{-1}]\widetilde \to \AA_n$.

\end{proposition}

\begin{proof} We need the following result.

\begin{lemma}
\label{le:relations tDelta}
Let $\Delta$ be any triangulation of $[n]$. Then

\noindent (i) For any distinct $i,j,k\in [n]$, the elements $x'_{ab}:=t_{ab}^\Delta$, $\{a,b\}\subset \{i,j,k\}$ satisfy the triangle relations \eqref{eq:triangle relations}.

\noindent (ii) For any cyclic quadruple $(i,j,k,\ell)$ the elements  $x'_{ab}:=t_{ab}^\Delta$, $\{a,b\}\subset \{i,j,k,\ell\}$ satisfy the exchange relations \eqref{eq:exchange relations}.

\end{lemma}

\begin{proof} Indeed, to prove (i) note that for any distinct $i,j,k\in [n]$ there exists a triangulation $\Delta_0$ such that $(i,j,k)$ is a triangle in $\Delta_0$ therefore, the elements $t_{ab}\in \TT_{\Delta'}$, $\{a,b\}\subset \{i,j,k\}$ satisfy \eqref{eq:triangle relations}. Applying the isomorphism
$\psi_{\Delta,\Delta_0}$ given by \eqref{th:universal localization n-gon}, we finish the proof of (i).

To prove (ii) note that for any cyclic $(i,j,k,\ell)$ there exists a triangulation $\Delta_0$ such that both triangles $(i,j,k)$ and $(j,k,\ell)$
belong to $\Delta_0$ (hence $(j,\ell)\notin \Delta_0$). By \eqref{eq:tDelta jl} for $\Delta_0$, we see that $t_{ab}^{\Delta_0}$,   $\{a,b\}\subset \{i,j,k,\ell\}$ satisfy \eqref{eq:exchange relations}.

Thus applying the isomorphism $\psi_{\Delta,\Delta_0}$, we finish the proof of (ii).

The lemma is proved. \end{proof}

By Lemma \ref{le:relations tDelta},  the assignments $x_{pq}\mapsto t_{pq}^\Delta$ for all distinct $p,q\in [n]$ define an epimorphism of algebras
$$\AA_n\twoheadrightarrow \QQ \TT_\Delta[S_\Delta^{-1}]\ .$$
On the other hand,  by (already proved) Theorem \ref{th:noncomlaurent n-gon}, for each triangulation $\Delta$ of $[n]$ and any distinct $i,j\in [n]$ the element $x_{ij}\in \ii_\Delta(\QQ \TT_\Delta)$. Therefore, by the universality of localizations, $\ii_\Delta$ extends to an epimorphism of algebras $\hat \ii_\Delta:\QQ \TT[S_\Delta^{-1}]\twoheadrightarrow \AA_n$. Clearly, these two homomorphisms are mutually inverse.

This finishes the proof of Proposition \ref{pr:canonical isomorphism localization polygon}.  \end{proof}

Furthermore, denote by ${\bf S}$ the submonoid of $\AA_\Delta\setminus \{0\}$ generated by all $x_{ij}$.
Clearly, ${\bf S}=\ii_\Delta(S_\Delta)$ and $\AA_\Delta=\ii_\Delta(\QQ \TT_\Delta)$. Therefore, $\AA_n=\AA_\Delta[{\bf S}^{-1}]$.

This proves Theorem \ref{th:A_Delta}. \endproof

Finally, Theorem \ref{th:A_n divisible} follows from Theorem \ref{th:A_Delta} and that $\AA_n':=\AA_\Delta=\ii_\Delta(\QQ \TT_\Delta)$ is the group algebra of
$\TT_\Delta$, which is a free group in $3n-4$ generators by (already proved) Theorem \ref{th:T_Delta polygon}. \endproof

\subsection{Self-similarity implies injectivity}
\label{subsec:rigidity=>injectivity}

In this section we prove the following result.

\begin{proposition}
\label{pr:A_n divisible conj}
If Conjecture \ref{conj:rigid linear} holds for  $m=3n-4$, $n\ge 4$ and $k=2,\ldots,n-2$, then  for each triangulation $\Delta$ of $[n]$ the  homomorphism of algebras
$$\AA_n\to \FF_{3n-4}\ ,$$
which is the canonical (by Proposition \ref{pr:canonical isomorphism localization polygon} and Lemma \ref{le:universal property of localization}) extension to $\AA_n\cong \QQ\TT_\Delta[S_\Delta^{-1}]$  of the natural inclusion $\QQ \TT_\Delta\hookrightarrow Frac(\QQ \TT_\Delta)\cong \FF_{3n-4}$
is a also a monomorphism (hence $\AA_n$ has no zero divisors).

\end{proposition}

\begin{proof} it suffices to show that for at lest one triangulation  $\Delta$ of $[n]$ the submonoid
$\hat S_\Delta\subset \QQ \TT_\Delta\setminus \{0\}$
generated by all $t_{ij}^\Delta$ and by $(\QQ  \TT_\Delta)^\times=\QQ^\times\cdot \TT_\Delta$ is   factor-closed  in the sense of
Definition \ref{def:divisible submonoid}.  Since $\TT_\Delta$ is a free group by Theorem \ref{th:T_Delta polygon}, in view of Proposition \ref{pr:UFD},
it suffices to verify that each $t_{ij}^\Delta$, $(i,j)\notin \Delta$ is prime in  $\QQ \TT_\Delta$ and all primes similar to $t_{ij}$ belong to $\hat S_\Delta$.
Now let $\Delta=\Delta_1$ be the starlike triangulation as in \eqref{eq:starlike} with $i=1$.

We need the following obvious fact.

\begin{lemma}
\label{le:starlike relation n-gon}
For all $n\ge 2$ the  group $\TT_{\Delta_1}$ is freely generated by $\tau_j:=T_1^{j,j+1}$, $j=2,\ldots,n-1$, $t_{1,k}$, $t_{k,1}$, $k=2,\ldots,n$.
\end{lemma}

\begin{proof} Clearly, $\TT_{\Delta_1}$ has a presentation
$t_{j,j+1}=t_{j,1}\tau_jt_{1,j+1},~t_{j+1,j}=t_{j+1,1}\tau_jt_{1j}$
for $j=2,\ldots,n-1$.

This proves the lemma.
\end{proof}

Furthermore,  Corollary \ref{cor:angle is additive} implies that the monoid $\hat S_{\Delta_1}$ is generated by $\TT_{\Delta_1}$ and noncommutative angles
$$T_1^{ij}=\tau_i+\ldots +\tau_j$$
for $2\le i<j\le n$. Clearly, each  $T_1^{ij}$, $i<j-1$ is prime in $\QQ \TT_{\Delta_1}$.
Let $P_{ij}:= \QQ^\times \cdot \TT_{\Delta_1}\cdot T_1^{ij}\cdot \TT_{\Delta_1}$ for $2\le i<j\le n$. By Conjecture \ref{conj:rigid linear} with $m=3n-4$, $k=j$,
that the only primes similar to $T_1^{ij}$ are elements of $P_{ij}$. This together with Proposition \ref{pr:UFD} and Remark \ref{rem:free group algebra UFD}
proves that the submonoid $\QQ^\times \cdot \hat S_{\Delta_1}$
of $\QQ \TT_{\Delta_1}\setminus \{0\}$ is factor-closed because it is generated by $\QQ^\times\cdot \TT_{\Delta_1}$ and  $P=\bigcup\limits_{2\le i<j\le n} P_{ij}$. Therefore,
Corollary \ref{cor:free divisible} guarantees that $\QQ \TT_{\Delta_1}[S^{-1}]=\QQ \TT_{\Delta_1}[\QQ^\times \cdot \hat S_{\Delta_1}^{-1}]$ is a subalgebra of
$\FF_{3n-4}=Frac(\QQ \TT_{\Delta_1})$.

Using this and Proposition \ref{pr:canonical isomorphism localization polygon} with $\Delta=\Delta_1$,
we finish the proof of Proposition \ref{pr:A_n divisible conj}. \end{proof}

\section{Noncommutative surfaces}

\label{sect:Noncommutative triangulated surfaces}
In this section we extend all the constructions and results of Section 2 to  {\it marked surfaces}  i.e., (connected compact smooth) surfaces $\Sigma$
possibly with boundary equipped with a non-empty finite set $I=I(\Sigma)=I_b\sqcup I_p$ of marked points with a subset $I_b=I_b(\Sigma)\subset I$ of marked boundary points, the set $I_p=I_p(\Sigma)=I\setminus I_b$ of {\it ordinary punctures} and a set $I_s=I_s(\Sigma)$ of {\it special punctures} (which were called {\it orbifold point of order $2$} in \cite{FST}, however, we will not use this terminology). We also require that each boundary component contains at least one point from $I_b$. We denote by $\underline \Sigma$ the underlying topological space.


\subsection{Multi-groupoid of curves on $\Sigma$}
Given points $p_1,p_2\in I(\Sigma)$, consider
connected  smooth directed curves $C$ in $\underline \Sigma\setminus I_p(\Sigma)$ starting at $p_1$ and terminating at $p_2$.
For a curve $C$ denote by $\overline C$ the same  curve traversed from $p_2$ to $p_1$.
We say that
curves $C$ and $C'$ in $\Sigma$ from $p_1$ to $p_2$ are {\it equivalent} if $C$ and $C'$ are homotopy equivalent as (connected smooth directed) curves in $\underline \Sigma\setminus I_p(\Sigma)$.

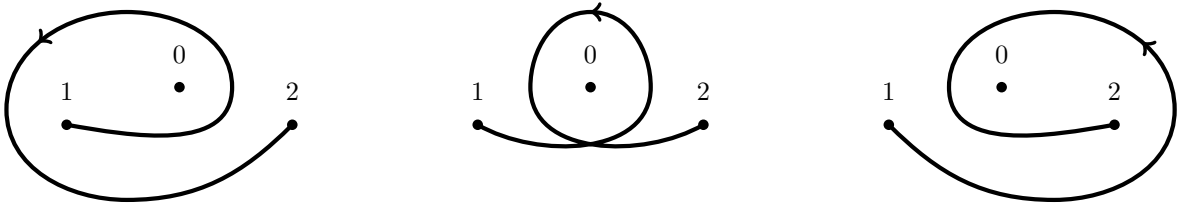
\begin{figure}[h!]
\begin{subfigure}{0.3\linewidth}
\centering
\begin{tikzpicture}
\clip (-2.5,-2) rectangle (2.5,2);
\fill (0,0) circle (0.07);
\fill (-1.5,-0.5) circle (0.07);
\fill (1.5,-0.5) circle (0.07);
\draw (0,0.2) node[above] {$0$};
\draw (-1.5,-0.3) node[above] {$1$};
\draw (1.5,-0.3) node[above] {$2$};
\draw [ultra thick,->-] (-1.5,-0.5) to[out=350,in=270] (0.7,0) to[out=90,in=0] (-0.7,1) to[out=180,in=90] (-2.3,-0.3) to[out=270,in=180] (-0.7,-1.5) to[out=0,in=225] (1.5,-0.5);
\end{tikzpicture}
\end{subfigure}
\quad
\begin{subfigure}{0.3\linewidth}
\centering
\begin{tikzpicture}
\clip (-2.5,-2) rectangle (2.5,2);
\fill (0,0) circle (0.07);
\fill (-1.5,-0.5) circle (0.07);
\fill (1.5,-0.5) circle (0.07);
\draw (0,0.2) node[above] {$0$};
\draw (-1.5,-0.3) node[above] {$1$};
\draw (1.5,-0.3) node[above] {$2$};
\draw [ultra thick,->-] (-1.5,-0.5) to[out=330,in=270] (0.8,0) to[out=90,in=0] (0,1) to[out=180,in=90] (-0.8,0) to[out=270,in=210] (1.5,-0.5);
\end{tikzpicture}
\end{subfigure}
\quad
\begin{subfigure}{0.3\linewidth}
\centering
\begin{tikzpicture}
\clip (-2.5,-2) rectangle (2.5,2);
\fill (0,0) circle (0.07);
\fill (-1.5,-0.5) circle (0.07);
\fill (1.5,-0.5) circle (0.07);
\draw (0,0.2) node[above] {$0$};
\draw (-1.5,-0.3) node[above] {$1$};
\draw (1.5,-0.3) node[above] {$2$};
\draw [ultra thick,->-] (-1.5,-0.5) to[out=315,in=180] (0.7,-1.5) to[out=0,in=270] (2.3,-0.3) to[out=90,in=0] (0.7,1) to[out=180,in=90] (-0.7,0) to[out=270,in=190] (1.5,-0.5);
\end{tikzpicture}
\end{subfigure}
\caption{Pairwise non-equivalent curves from puncture 1 to puncture 2}
\end{figure}


Denote by
$\Gamma_{ij}=\Gamma_{ij}(\Sigma)$ the set of equivalence classes of curves $C$ in $\Sigma$ which originate at $i$ and terminate at $j$
then let $\Gamma=\Gamma(\Sigma):=\bigsqcup\limits_{i,j\in I(\Sigma)}  \Gamma_{ij}$.
For $\gamma\in   \Gamma_{ij}$ we denote by $s(\gamma)\in I(\Sigma)$ (resp. by $t(\gamma)\in I(\Sigma)$) the source $i$ (resp. the target $j$).  

Thus we have a natural involution
$\overline{\cdot}: \Gamma\,\widetilde \rightarrow \,  \Gamma$ ($\gamma\mapsto \overline \gamma$). By definition,
$\overline \Gamma_{ij}= \Gamma_{j,i}$ for all $i,j\in I(\Sigma)$.

\begin{picture}(200,40)
\put(160,10){\includegraphics[scale=0.3]{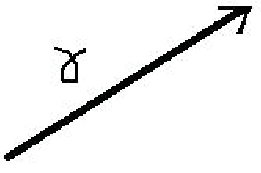}}
\put(270,10){\includegraphics[scale=0.3]{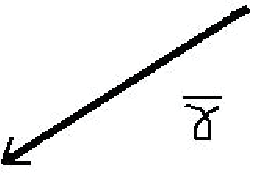}}
\put(190,0){Involution  $\gamma\mapsto \overline \gamma$}
\end{picture}

\medskip

For  $j\in I(\Sigma)$ denote by $id_j$ the {\it trivial loop} at $j$. Clearly, $\gamma=\overline \gamma$ iff
$\gamma$ is trivial.


It is easy to see that $\Gamma(\Sigma)$ is finite iff $\Sigma$ is homeomorphic to an $n$-gon, i.e., a disk with $n\ge 1$ marked  points and no punctures. In that case,  the assignments
$\gamma\mapsto (s(\gamma),t(\gamma))$  define a bijection $\Gamma\widetilde \to \{(i,j)\in [n],i \ne j\}$.

We say that  $\gamma\in \Gamma(\Sigma)$ is {\it simple} if it has a  non-self-intersecting  representative. Denote by $\Gamma^0(\Sigma)$ the set of all simple $\gamma\in \Gamma(\Sigma)$.


%

\begin{definition}
We say that a pair $(\gamma,\gamma')$ in $\Gamma(\Sigma)$  is {\it composable} if $t(\gamma)=s(\gamma')$ and define
the {\it composition} $\gamma''=\gamma\circ \gamma'$ to be the pullback, under the natural projection
$\Gamma(\Sigma)\twoheadrightarrow \Gamma(\Sigma\setminus (I_p(\Sigma) \setminus\{t(\gamma)\}))$ of the concatenation of
$\gamma$ and $\gamma'$.

\end{definition}

Clearly, the multi-composition  $\gamma\circ \gamma'$ is a 1-element set  iff $t(\gamma)=s(\gamma')\in I_b(\Sigma)$.
Otherwise $\gamma\circ \gamma'$ is countable.

\begin{picture}(110,95)

\put(140,10){\includegraphics[scale=0.7]{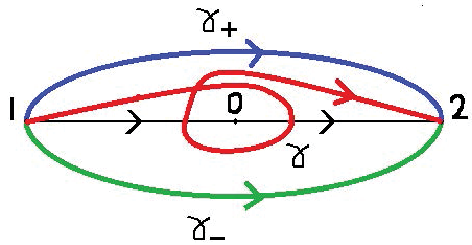}}
\put(50,0){Multi-composition: $\{\gamma_-,\gamma,\gamma_+\}\in (1,0)\circ (0,2)$.}

\end{picture}

%
%
%
%
%
%
%

\medskip

The following  is immediate.

\begin{lemma}
For each marked surface $\Sigma$ the set $\Gamma(\Sigma)$ is a multi-groupoid  with the object set $I(\Sigma)$ and the inverse given by $\gamma^{-1}:=\overline \gamma$.

%

\end{lemma}

\begin{remark} A multi-category (e.g., a multi-groupoid) is a natural generalization of a category (e.g., of  a groupoid) where we allow the composition of two morphisms to be a set of  arrow and require the associativity $(\gamma\circ \gamma')\circ \gamma''=\gamma\circ (\gamma'\circ \gamma'')$, which is an equality of sets, see e.g. \cite{comer} (where the term {\it polygroupoid} was introduced).
\end{remark}

\begin{remark} If $I_p(\Sigma)=\emptyset$, then   $\Gamma(\Sigma)$ is an ordinary groupoid
(cf. \cite[Section 2.2]{CCS}).

\end{remark}

%
%
%

\subsection{Category of surfaces and  reduced curves}

\begin{definition}
\label{def:proper map}
Given a  continuous map $f:\underline \Sigma\to \underline \Sigma'$ with discrete fibers,
we say that $f$ is a {\it morphism of marked surfaces}
$\Sigma\to \Sigma'$ if:


$\bullet$
$f^{-1}(I(\Sigma'))= I(\Sigma)$, $f(I_s(\Sigma))\subset  I_s(\Sigma')$
(we abbreviate
$I^f:=f^{-1}(I_s(\Sigma'))\setminus I_s(\Sigma))$.

$\bullet$ For each point $p\in \underline \Sigma\setminus I^f$ there is a neighborhood ${\mathcal O}_p$ of $p$ in $\underline \Sigma$ such that the restriction of $f$ to ${\mathcal O}_p$ is injective (if $p\in \partial \underline \Sigma$ is a boundary point, then ${\mathcal O}_p$ is a ``half-neighborhood").

$\bullet$
For each $p\in I^f$  there is a neighborhood ${\mathcal O}_p$ of $p$ in $\underline \Sigma$ such that the restriction of $f$ to ${\mathcal O}_p$ is a two-fold cover of $f({\mathcal O}_p)$ ramified at $f(p)$.

\end{definition}


\begin{theorem}
\label{th:category surf}
For any morphisms of marked surfaces $f:\Sigma\to \Sigma'$ and $f':\Sigma'\to \Sigma''$ the composition $f'\circ f:\underline \Sigma\to \underline \Sigma''$ is also a morphism of marked surfaces $\Sigma'\to \Sigma''$.


\end{theorem}

We prove Theorem \ref{th:category surf} in Section \ref{subsect:canonical covers}.

In what follows, denote by  ${\bf Surf}$ the category whose objects are marked surfaces and arrows are morphisms of marked surfaces.

%
%
%
%

Note that if $f:\Sigma\to \Sigma'$ is a morphism in {\bf Surf} with $I^f=\emptyset$,
then $f$ respects (homotopy) equivalence of curves and, in particular, defines a map $\Gamma(\Sigma)\to \Gamma(\Sigma')$. In general, this is no longer true. To fix it, we define below a stronger equivalence relation than the equivalence for curves in $\Sigma'$.

Indeed, given $i\in I_s(\Sigma)$, we say that a curve $C$ in $\Sigma$ is $i$-{\it reducible} if there is a self-intersection point $p\in C$ such that the loop $C_0\subset C$ defined by $p$ encloses exactly one special puncture $i$; otherwise, $C$ is $i$-{\it reduced}.
Respectively, $\gamma\in \Gamma(\Sigma)$ is $i$-reducible (resp. $i$-reduced) if $\gamma$ has an $i$-reducible (resp. $i$-reduced) representative.
Denote by $[\Gamma(\Sigma)]_i$ the set of all $i$-reduced $\gamma\in \Gamma(\Sigma)$, abbreviate
$[\Gamma(\Sigma)]:=\bigcap\limits_{i\in I_s(\Sigma)} [\Gamma(\Sigma)]_i$ and
refer to elements of $[\Gamma(\Sigma)]$ as {\it reduced}. Clearly, $[\Gamma(\Sigma)]=\Gamma(\Sigma)$ iff $I_s(\Sigma)=\emptyset$. It is also clear that and each $\gamma\in \Gamma^0(\Sigma)$ is reduced.

For each $i$-reducible $\gamma\in \Gamma(\Sigma)$ denote by $[\gamma]_i$  the class in $\Gamma(\Sigma)$ obtained by resolving the self-intersecting simple loop around $i$ in (a generic representative $C$ of) $\gamma$
so that the resulting curve is connected (the ``wrong" crossing resolution would
result in creating two connected components, one of which is a circle around $i$).

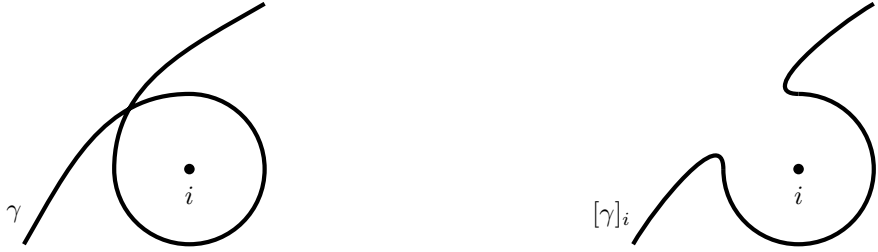
\begin{figure}[h!]
\begin{subfigure}{0.47\linewidth}
\centering
\begin{tikzpicture}
\draw[ultra thick] (-2.2,-1) to[out=60,in=180] (0,1) to[out=0,in=90] (1,0) to[out=270,in=0] (0,-1) to[out=180,in=270] (-1,0) to[out=90,in=210] (1,2.2);
\fill (0,0) circle (0.07);
\draw (0,-0.1) node [thick,below] {$i$};
\draw (-2.1,-0.6) node [thick,left] {$\gamma$};
\end{tikzpicture}
\bigskip
\end{subfigure}
\begin{subfigure}{0.47\linewidth}
\centering
\begin{tikzpicture}
\draw[ultra thick] (-2.2,-1) to[out=60,in=90] (-1,0);
\draw[ultra thick] (0,1) to[out=0,in=90] (1,0) to[out=270,in=0] (0,-1) to[out=180,in=270] (-1,0);
\draw[ultra thick] (0,1) to[out=180,in=210] (1,2.2);
\fill (0,0) circle (0.07);
\draw (0,-0.1) node [thick,below] {$i$};
\draw (-2.1,-0.6) node [thick,left] {$[\gamma]_i$};
\end{tikzpicture}
\bigskip
\end{subfigure}
\caption{Crossing resolution}
\end{figure}

The following is obvious.
\begin{lemma}

(a) The assignments
$\gamma\to \begin{cases} [\gamma]_i & \text{if $\gamma$ is $i$-reducible}\\
\gamma & \text{if $\gamma$ is $i$-reduced}\\
\end{cases}$ define a map $\pi_i:\Gamma(\Sigma)\to \Gamma(\Sigma)$.

(b) $\pi_i\circ \pi_j=\pi_j\circ \pi_i$ for all $i,j\in I_s(\Sigma)$.

(c) The assignments $\gamma\mapsto \pi_i^N(\gamma)$ for sufficiently big $N$ define a projection $\pi_i^\infty:\Gamma(\Sigma)\to [\Gamma(\Sigma)]_i$.

(d) The composition $\pi^\infty:=\prod\limits_{i\in I_s(\Sigma)} \pi_i^\infty$ is a projection $\Gamma(\Sigma)\to [\Gamma(\Sigma)]$.
\end{lemma}

This, in particular, defines an equivalence relation on $[\Gamma(\Sigma)]$, namely for $\gamma,\gamma'\in \Gamma(\Sigma)$ we say that any representatives $C\in \gamma$ and $C'\in \gamma'$ are {\it $I_s(\Sigma)$-equivalent} iff $\pi^\infty(\gamma)=\pi^\infty(\gamma')$. We naturally identify  $I_s(\Sigma)$-equivalence classes with elements of $[\Gamma(\Sigma)]$.

%
%

For each $i\in I_s(\Sigma)$ and $j\in I(\Sigma)$ let $\lambda_{ij}$  denote a (unique up to $\overline {\cdot }$) simple loop at $j$ around  $i$ in $[\Gamma(\Sigma)]$.
We refer to such loops as {\it special}.

For $n\ge 1$, $h\ge 0$ denote by $P_n(h)$ the  $n$-gon (i.e., a disk with $n$ marked boundary points) with $h$ special punctures
and abbreviate $P_n:=P_n(0)$.
Clearly, each special loop $\lambda$ determines a (homeomorphic) copy of $P_1(1)$ with the marked point set $\{\underline j\}$ and the special puncture set $\{i\}$.

\begin{lemma} For any marked surface $\Sigma$ the set $[\Gamma(\Sigma)]$ has a natural multi-groupoid structure:
$$[\gamma]\circ [\gamma']:=[\gamma\circ \gamma']$$
for any composable $(\gamma,\gamma')$ in the multi-groupoid $\Gamma(\Sigma)$ with the object set $I(\Sigma)$. Moreover,

(i) the assignments $\gamma\mapsto [\gamma]$ define a surjective homomorphism of multi-groupoids $\Gamma(\Sigma)\twoheadrightarrow [\Gamma(\Sigma)]$.

(ii) For each $i\in I_s(\Sigma)$ and $j\in I(\Sigma)$
each special loop satisfies $\overline \lambda_{ij}=\lambda_{ij}$.

\end{lemma}

\begin{picture}(200,100)
\put(50,10){\includegraphics[scale=0.3]{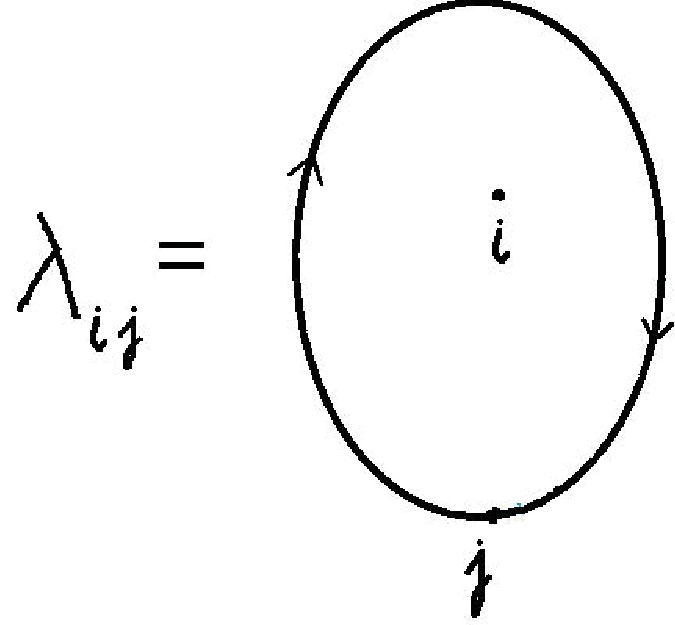}}
\put(170,10){\includegraphics[scale=0.3]{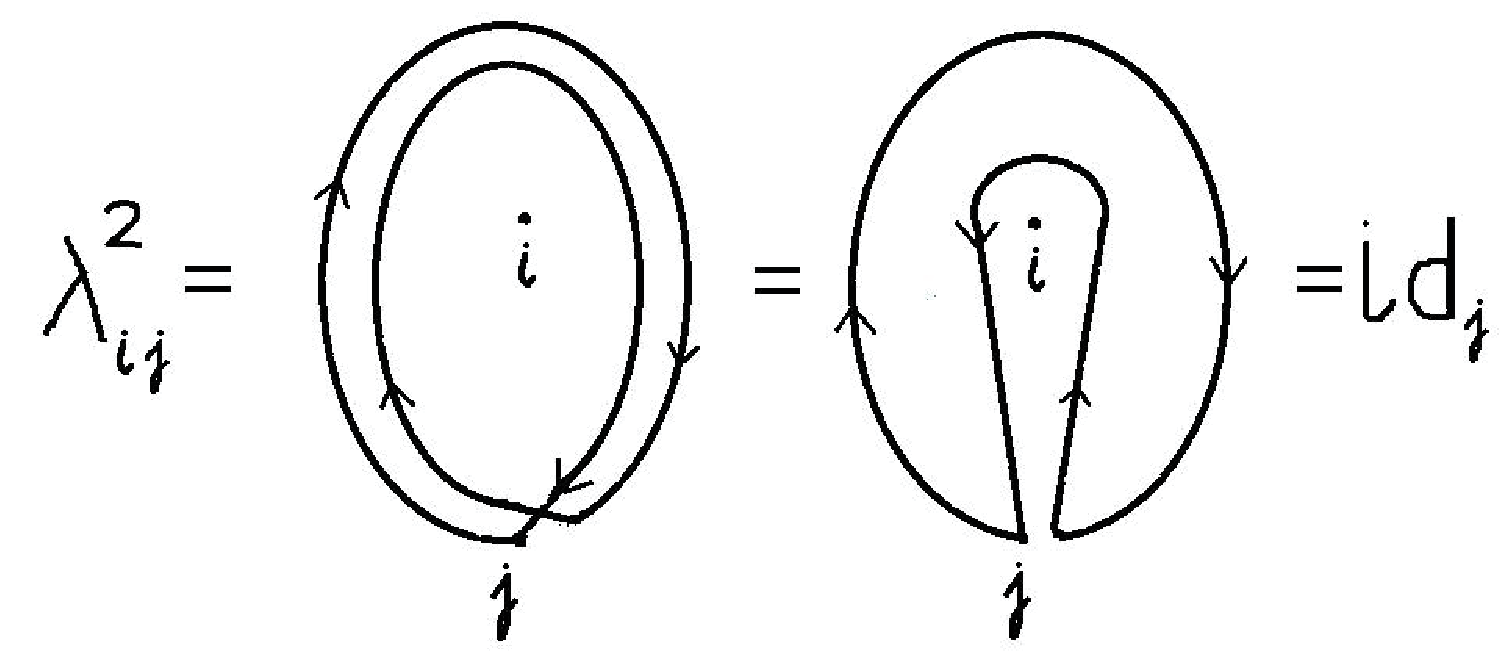}}
\put(140,0){Special loops are involutions in $[\Gamma(\Sigma)]$}

\end{picture}
\medskip

%

\medskip

The following result asserts functoriality  of the multi-groupoid under morphisms of surfaces.

\begin{theorem}
\label{th:I_s-homotopy equivalence}
Let $f$ be any morphism of marked surfaces $\Sigma\to \Sigma'$ and let $\gamma\in [\Gamma(\Sigma)]$.
Then

(a) for any generic representatives $C,C'\in \gamma$, their images $f(C)$ and $f(C')$ are $I_s(\Sigma')$-equivalent.

(b) For each $\gamma\in [\Gamma(\Sigma)]$ there exists a unique $I_s(\Sigma')$-equivalence class
$f(\gamma)\in [\Gamma(\Sigma')]$ such that $f(C)\in f(\gamma)$
for any generic curve $C\in \gamma$.

(c) $f:[\Gamma(\Sigma)]\to [\Gamma(\Sigma')]$ ($\gamma\mapsto f(\gamma)$) is a homomorphism of multi-groupoids.

(d) The assignments $\Sigma\mapsto [\Gamma(\Sigma)]$ define a functor from {\bf Surf} to the category of multi-groupoids.

\end{theorem}
\begin{picture}(50,120)
\put(80,0){\includegraphics[scale=0.3]{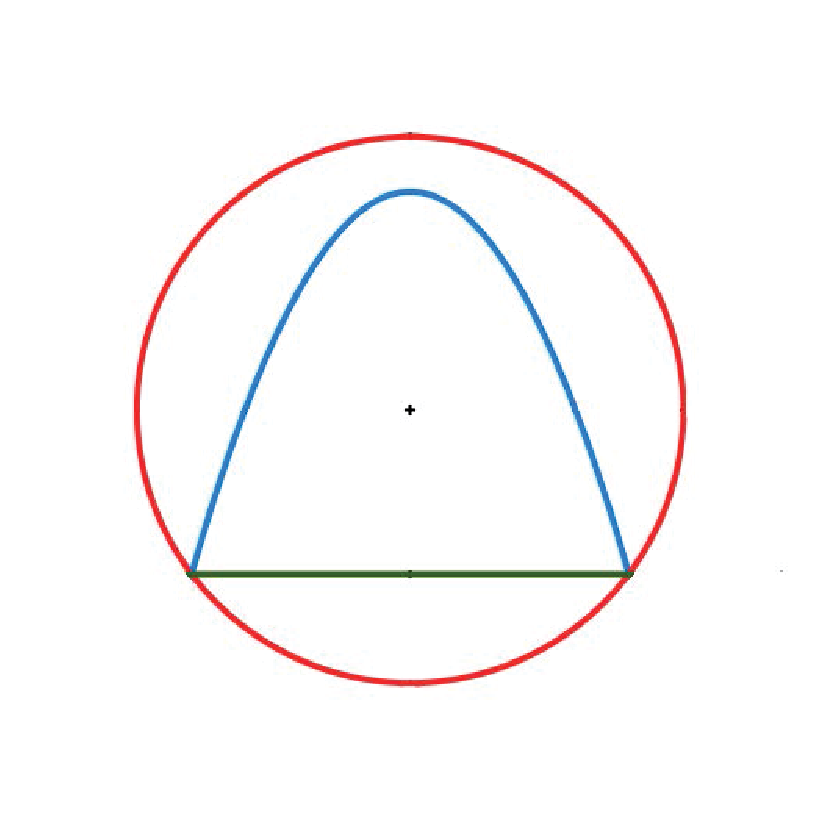}}
\put(220,0){\includegraphics[scale=0.3]{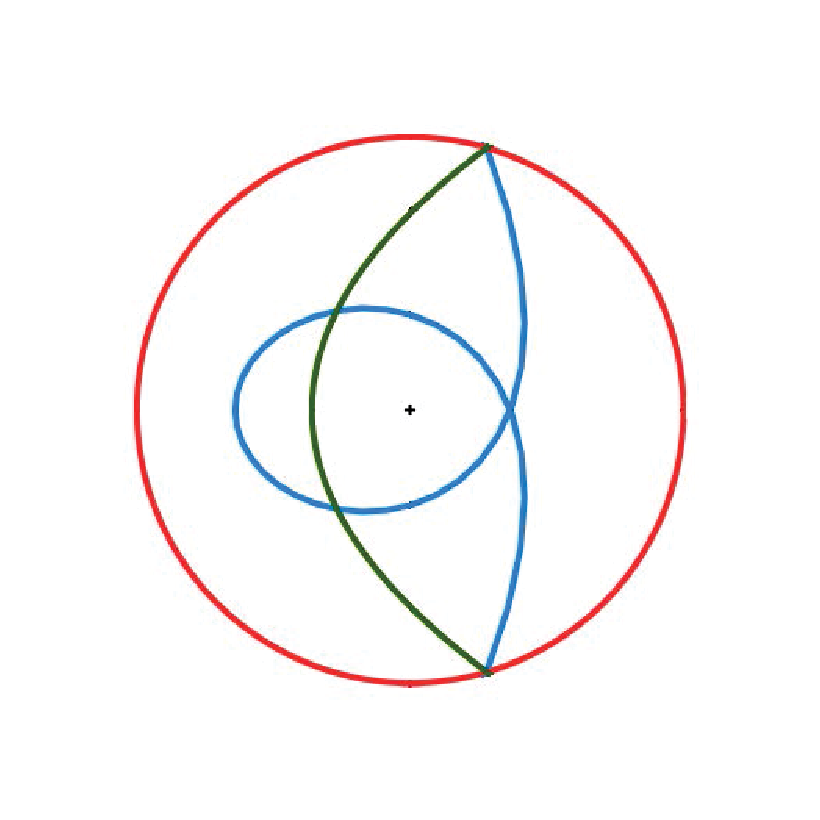}}
\put(20,0){$I_s(\Sigma)$-equivalence of images of curves under the ramified double cover $z\mapsto z^2$ of $\CC$}

\end{picture}

We prove Theorem \ref{th:I_s-homotopy equivalence} in Section \ref{subsect:canonical covers}.

%

It is well-known that marked surfaces can be glued out of polygons, i.e., for any  $\Sigma$ there exists a surjective gluing morphism $f:P_n(h)\twoheadrightarrow \Sigma$ in {\bf Surf} with $h=|I_s(\Sigma)|$, $n\ge 1$ such that all $f(i,i^+)\in \Gamma^0(\Sigma)$ and the restriction of $f$ to the interior of $P_n(h)$ is injective.
For readers' convenience we construct such a gluing morphism $f$ in Lemma \ref{le:n-gon cover of Sigma} for any triangulation of $\Sigma$.

The following fact is obvious.

\begin{lemma}
\label{le:finite Gamma}
Let $\Sigma$ be a marked surface. Then $[\Gamma(\Sigma)]$ is finite if an only if $\Sigma$ is  homeomorphic either a once punctured sphere or to $[n]=P_n$ or to $P_n(1)$ for some $n\ge 1$.  More precisely,  the assignments
$$\gamma\mapsto
\begin{cases}
(s(\gamma),t(\gamma),+) & if \text{the special puncture is to the right of $\gamma$}\\
(s(\gamma),t(\gamma),-) & if \text{the special puncture is to the left of $\gamma$}\\
\end{cases}$$ define a bijection $[\Gamma(P_n(1))]\widetilde \to \{(i,j)\in [n]\}\times \{-,+\}$.


\end{lemma}

\subsection{Polygons in surfaces, noncommutative surfaces and functoriality}

We say that a sequence $P=(\gamma_1,\ldots,\gamma_r)$ of not necessarily distinct $\gamma_i\in [\Gamma(\Sigma)]$, $i\in [r]$, is {\it cyclic} if each pair $(\gamma_i,\gamma_{i^+})$, $i\in [r]$ is composable.
%


%
\begin{definition}
\label{def:ngon}
We say that a sequence $P=(\gamma_1,\ldots,\gamma_n)$ is an $n$-gon in $\Sigma$ if there exists a morphism $f:P_n\to \Sigma$  such that
$f(i,i^+)=\gamma_i$ for $i\in [n]$. We also denote $\gamma_{ij}:=f(i,j)$ for all distinct $i,j\in [n]$ (clearly, $\gamma_{ij}$ is nontrivial for all distinct $i,j\in [n]$).
We will refer to such an $f$ as an {\it accompanying to $P$} morphism.

\end{definition}

Clearly, each $n$-gon $P=(\gamma_1,\ldots,\gamma_n)$ in $\Sigma$ is cyclic and for any $\gamma\in [\Gamma(\Sigma)]$ the pair $(\gamma,\overline \gamma)$ is a $2$-gon in $\Sigma$. It is convenient to define the {\it interior} $P^0$ of an $n$-gon $P=(\gamma_1,\ldots,\gamma_n)$ to be the image of the interior of $P_n$ under an accompanying morphism  (to do so we choose generic representatives $C_i\in \gamma_i$ so that $f(i,i^+)=C_i$
for $i\in [i]$).
It is also clear that
$P^0$ does not depend on the choice of $f$, and different choices of $C_i\in \gamma_i$  result in homotopic to each other morphisms $f:P_n\to \Sigma$.
We say that $P$ is {\it simple} if $P^0$ is homeomerphic to a disk.
%


%
%

%

We will sometimes refer to an $3$-gon in $\Sigma$ respectively as a triangle and to a $4$-gon -- as a quadrilateral.

%

\begin{picture}(110,110)
\put(80,10){\includegraphics[scale=0.33]{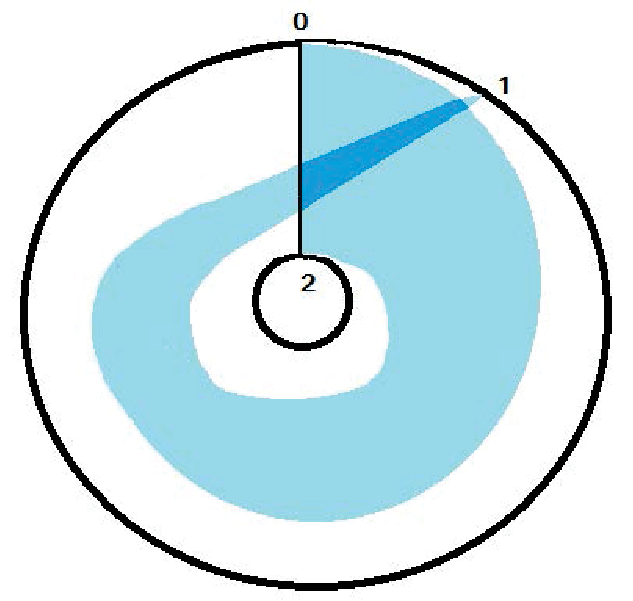}}
\put(230,10){\includegraphics[scale=0.3]{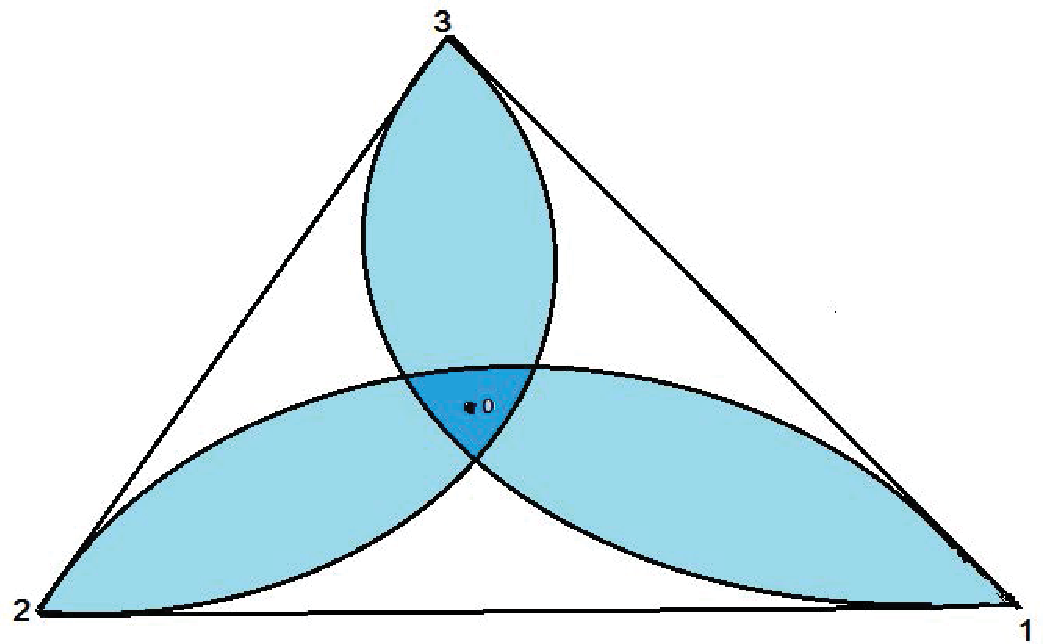}}
\put(130,0){Non-simple triangles in an annulus  and  in $P_3(1)$}

\end{picture}

\begin{definition}
\label{def:Asigma}
For a marked surface $\Sigma$ let  ${\mathcal A}_\Sigma$ be the $\QQ$-algebra generated by all $x_\gamma$, $\gamma\in [\Gamma(\Sigma)]$ subject to

(i) $x_{\gamma}=1$ if $\gamma$ is trivial.



(ii) (triangle relations)  For any triangle $(\gamma_1,\gamma_2,\gamma_3)$ in $\Sigma$ one has
\begin{equation}
\label{eq:triangle relations sigma}
x_{\gamma_1}x_{\overline \gamma_2}^{-1}x_{\gamma_3}=x_{\overline\gamma_3}x_{\gamma_2}^{-1}x_{\overline \gamma_1} \ .
\end{equation}

(iii) (exchange relations) For any quadrilateral $(\gamma_1,\gamma_2,\gamma_3,\gamma_4)$ in $\Sigma$:
\begin{equation}
\label{eq:exchange relations sigma}
x_{\gamma_{24}}=x_{\gamma_{21}}x_{\gamma_{31}}^{-1}x_{\gamma_{34}}+ x_{\gamma_{23}}x_{\gamma_{13}}^{-1}x_{\gamma_{14}} \ .
\end{equation}

Likewise (similarly to Section \ref{subsec:Big noncommutative polygons}), we define the {\it big triangle group}
$\TT_\Sigma$ of $\Sigma$ to be generated by all $t_\gamma$, $\gamma\in [\Gamma(\Sigma)]$ subject to:

$\bullet$  $t_{\gamma}=1$ if $\gamma$ is trivial.


$\bullet$ (triangle relations) $t_{\gamma_1}t_{\overline \gamma_2}^{-1}t_{\gamma_3}=t_{\overline \gamma_3}t_{\gamma_2}^{-1}t_{\overline \gamma_1}$
for all triangles $(\gamma_1,\gamma_2,\gamma_3)$ in $\Sigma$.
\end{definition}




The following fact is obvious.
\begin{lemma} For each marked surface $\Sigma$ the assignments $t_\gamma\mapsto x_\gamma$ define a homomorphism of groups:
\begin{equation}
\label{eq:big triangular homomorphism surface}
\TT_\Sigma\to \AA_\Sigma^\times \ .
\end{equation}

\end{lemma}

It is natural to conjecture that this homomorphism is an isomorphism.

The following result is also obvious.

\begin{lemma}
\label{le:bar involution}
(a) For each marked surface $\Sigma$ there is a unique involutive anti-automorphism $\bar{\cdot}$ of $\AA_\Sigma$ (resp. of $\TT_\Sigma$) such that
$\overline x_\gamma=x_{\overline \gamma}$ (resp $\overline t_\gamma=t_{\overline \gamma}$) for all $\gamma\in [\Gamma(\Sigma)]$.

(b)
If $\gamma$ is a simple special loop around $i\in I_p(\Sigma)\sqcup I_s(\Sigma)$, then $\overline x_\gamma=x_\gamma$ (resp. $\overline t_\gamma=t_\gamma$).
\end{lemma}

\begin{remark} This bar anti-involution is analogous to the one in quantum algebras. Also, Lemma \ref{le:bar involution}(b)
asserts that simple loops around an ordinary and special punctures are ``close relatives."
\end{remark}

The following result, in fact, asserts that the assignments $\Sigma\mapsto \TT_\Sigma$ and $\Sigma\mapsto \AA_\Sigma$ are respectively functors ${\bf Surf}\to {\bf Groups}$ and ${\bf Surf}\to \QQ-{\bf Alg}$.

\begin{theorem}
\label{th:functorial triangular}
For any morphism $f:\Sigma\to \Sigma'$ in {\bf Surf}
the assignments $t_\gamma\mapsto t_{f(\gamma)}$ (resp. $x_\gamma\mapsto x_{f(\gamma)}$) define
a homomorphism of groups $f_\star:\TT_\Sigma\to \TT_{\Sigma'}$ (resp. of algebras $f_*:\AA_\Sigma\to \AA_{\Sigma'}$)
and the following diagram is commutative
\begin{equation}
\label{eq:TA commutative}
\begin{CD}
\TT_\Sigma@>>>\AA_\Sigma\\
@Vf_\star VV@Vf_* VV\\
\TT_{\Sigma'}@>>>\AA_{\Sigma'}
\end{CD}.
\end{equation}
\end{theorem}

%

We prove Theorem \ref{th:functorial triangular} in Section \ref{subsect:canonical covers}.




\begin{definition}
\label{def:gluing holes}
For a marked surface  $\Sigma$  denote by $\hat \Sigma$ the marked surface obtained from $\Sigma$ by turning each special puncture into the ordinary one, i.e., $\underline {\hat \Sigma}=\underline \Sigma$, $I(\hat \Sigma)=I(\Sigma)\sqcup I_s(\Sigma)$, $I_s(\hat \Sigma)=\emptyset$.
\end{definition}

Clearly, $[\Gamma(\Sigma)]\subseteq [\Gamma(\hat \Sigma)]= \Gamma(\hat \Sigma)$ and the complement $[\Gamma(\hat \Sigma)]\setminus [\Gamma(\Sigma)]$ consists of classes of curves originating or terminating in formerly special punctures.

\begin{proposition}
\label{pr:gluing holes}
The assignments $t_\gamma\mapsto t_\gamma$ for $\gamma\in [\Gamma(\Sigma)]$ define a homomorphism of groups
\begin{equation}
\label{eq:hole to puncture}
\TT_{\Sigma}\to \TT_{\hat \Sigma} \ ,
\end{equation}
where $\hat \Sigma$ is as in Definition \ref{def:gluing holes}.
\end{proposition}

\begin{remark}
\label{rem:improper inclusion}
It is natural to conjecture that \eqref{eq:hole to puncture} is injective.
Note, however, that the natural identification $Id:\Sigma\hookrightarrow \hat \Sigma$ is not a morphism in {\bf Surf} since it takes
$I_s(\Sigma)$ to $I_p(\hat \Sigma)$, so we expect that there is {\bf no} homomorphisms $\AA_\Sigma\to \AA_{\hat \Sigma}$, which together with \eqref{eq:hole to puncture} would make the diagram \eqref{eq:TA commutative} commutative, and illustrate this with the following example.

\end{remark}

\begin{example}
\label{ex:tow-gon surface}  Let $\Sigma=P_2(1)$ with the vertex set $I=\{1,2\}$ and a single special puncture $0$. For $i\in I$ denote by $\gamma_i$ the clockwise loop at $i$ around $0$ inside $\Sigma$.
For $i,j\in I$, $i\ne j$ denote by $\gamma_{ij}^+$ (resp. $\gamma_{ij}^-$) the boundary curve from $i$ to $j$ so that $0$ is to the right (resp. to the left).

\begin{picture}(110,110)
\put(170,10){\includegraphics[scale=0.3]{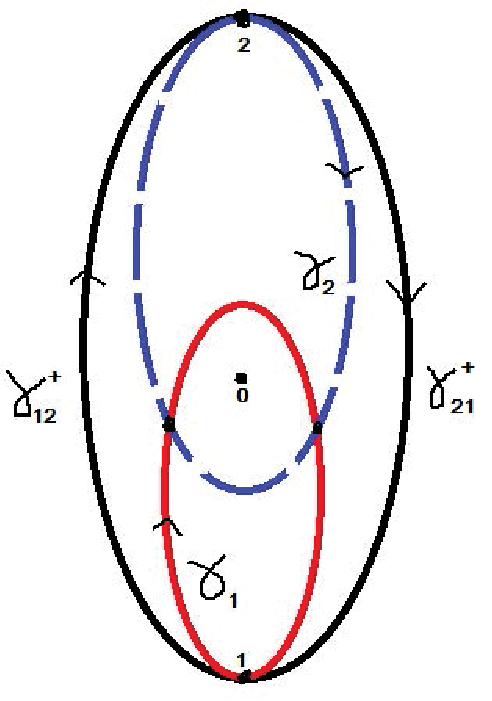}}
\put(135,0){A quadrilateral in $P_2(1)$}
\end{picture}

\medskip

We abbreviate $x_i:=x_{\gamma_i}$, $\overline x_i:=x_{\overline \gamma_i}$,  $x_{ij}^+:=x_{\gamma_{ij}^+}$, $x_{ij}^-:=x_{\gamma_{ij}^-}$ for the corresponding generators of $\AA_\Sigma$.

Then, according to Definition \ref{def:Asigma}, $\AA_\Sigma$ has a presentation:
$$
\overline x_1=x_1,~\overline x_2=x_2,~x_{21}^+x_1^{-1}x_{12}^+=x_{21}^-x_1^{-1}x_{12}^-,\
x_{12}^+x_2^{-1}x_{21}^+=x_{12}^-x_2^{-1}x_{21}^-\ ,
$$
$$
x_2=x_{21}^+x_1^{-1}x_{12}^- + x_{21}^-x_1^{-1}x_{12}^+, \
x_1=x_{12}^+x_2^{-1}x_{21}^- + x_{12}^-x_2^{-1}x_{21}^+\ .
$$

Let $\hat \Sigma$ be obtained from $\Sigma$ by converting all special punctures into ordinary ones (as in Definition \ref{def:gluing holes}).
Therefore, curves on $\hat \Sigma$ are those on $\Sigma$ plus four additional ones: directed intervals $\gamma_{0,i}$ from $0$ to each $i$ and
$\gamma_{i,0}:=\gamma_{0,i}^{-1}$.
We abbreviate the generators of  $\AA_{\hat \Sigma}$ same way as in $\AA_\Sigma$ and $x_{0,i}:=x_{\gamma_{0,i}}$, $x_{i,0}:=x_{\gamma_{i,0}}$.

Then, according to Definition \ref{def:Asigma}, $\AA_{\hat \Sigma}$ has a presentation:
$$
\overline x_1=x_1,~\overline x_2=x_2,~x_{21}^+x_1^{-1}x_{12}^+=x_{21}^-x_1^{-1}x_{12}^-,~x_{12}^+x_2^{-1}x_{21}^+=x_{12}^-x_2^{-1}x_{21}^-,~
x_{01}(x_{21}^\pm )^{-1}x_{20}=x_{02}(x_{12}^\mp)^{-1}x_{10}\ ,
$$
$$
x_1=x_{10}x_{20}^{-1}(x_{21}^+ + x_{21}^-),~x_2=x_{20}x_{10}^{-1}(x_{12}^+ + x_{12}^-) \ .
$$
In particular,
$$
x_2=x_{21}^-x_1^{-1}x_{12}^- + x_{21}^+x_1^{-1}x_{12}^+ +x_{21}^-x_1^{-1}x_{12}^+ + x_{21}^+x_1^{-1}x_{12}^-,~
x_1=x_{12}^-x_2^{-1}x_{21}^- + x_{12}^+x_2^{-1}x_{21}^+ +x_{12}^-x_2^{-1}x_{21}^+ + x_{12}^+x_2^{-1}x_{21}^-\ .
$$

Therefore,  there is no homomorphism $\AA_\Sigma\to \AA_{\hat \Sigma}$ or $\AA_{\hat \Sigma}\to \AA_\Sigma$ which would send $x_i\mapsto x_i$, $x_{ij}^\pm\mapsto x_{ij}^\pm$ (which justifies Remark \ref{rem:improper inclusion}).
\end{example}



\subsection{Triangulations of marked surfaces}
\label{subsec:triangulations}

Let $\Sigma$ be a marked surface, given distinct $\gamma,\gamma'\in [\Gamma(\Sigma)]$, define their {\it intersection number} $n_{\gamma, \gamma'}\in \ZZ_{\ge 0}$
to be the number of intersection points in the interiors of their generic representatives minus the endpoints of $\gamma$ and $\gamma'$.
Clearly, $n_{\gamma, \gamma'}$ is well-defined, i.e., does not depend on the choice of representatives. By definition, $n_{\gamma,\gamma'}=n_{\gamma',\gamma}=n_{\overline \gamma,\gamma'}$ for all $\gamma,\gamma'$. Note that $n_{\gamma,\gamma'}=0$ iff
$\gamma$ and $\gamma'$ do not intersect (and may have only endpoints in common).

Given a marked surface $\Sigma$, we say that a subset  $\Gamma'\subset \Gamma^0(\Sigma)$ is {\it non-crossing}  if $n_{\gamma,\gamma'}=0$ for all distinct $\gamma,\gamma'\in \Gamma'$, i.e.,  one can simultaneously choose generic representatives of classes in  $\Gamma'$ such that they pairwise do not intersect in $\Sigma$ and do not self-intersect (i.e., may have only endpoints in common).
Furthermore, we say that $\Delta$ is a {\it triangulation} of $\Sigma$ if $\Delta$ is a maximal non-crossing subset of $\Gamma^0(\Sigma)$ such that $\overline \Delta=\Delta$.

Clearly, if $I_s(\Sigma)\ne \emptyset$, then any triangulation $\Delta$
of $\Sigma$ has a special loop $\lambda_{ij}$ at some $j\in I_s(\Sigma)$ around each $i\in I_s(\Sigma)$, i.e., $\lambda_{ij}$ defines a
a $2$-gon $(\lambda_{ij},\lambda_{ij})$ in $\Delta$ homeomorphic to $P_1(1)$.
It is customary to fix a generic representative of each $\gamma^0\in \Delta$ so that $\Sigma$ is literally cut into triangles and $P_1(1)$'s.

It is well-known  that all triangulations of $\Sigma$ are finite of same cardinality. Moreover, any triangulation $\Delta'$ can be obtained from a given triangulation
$\Delta$ by a sequence of flips of diagonals in quadrilaterals in $\Delta$ (see e.g., \cite[Proposition 7.10]{FSTh} and \cite[Theorem 4.2]{FST}).

%
%

Given an $r$-gon $Q=(\gamma_1,\ldots,\gamma_r)$ in $\Sigma$ and a triangulation $\Delta$ of $\Sigma$, we say that $\gamma^0\in \Delta$ is {\it attracted} to $Q$ if either
$\gamma^0$ intersects $Q$ or there is a triangle $\tau=(\gamma^-,\gamma^0,\gamma^+)$ in $\Delta$ such that $\gamma^-$ intersects $Q$; denote by $\Delta_0=\Delta_0(Q,\Delta)$  the set of all $\gamma^0\in \Delta$ attracted to $Q$.

The following is immediate.

\begin{theorem}
\label{th:polygonal cover}
Let $\Delta$ be a triangulation of $\Sigma$.
Then for each  $r$-gon $Q=(\gamma'_1,\ldots,\gamma'_r)$ in $\Sigma$ there exists an $n$-gon $P=(\gamma_1,\ldots,\gamma_n)\in (\Delta_0(Q,\Delta))^n$ for some $n\ge r$, a triangulation $\Delta^0$ of $[n]$,
and an order-preserving embedding $\iota:[r]\hookrightarrow [n]$ such that:

(a) $\gamma_{ij}\in \Delta_0(Q,\Delta)$ iff $(i,j)\in \Delta^0$.

(b) $\gamma'_k=\gamma_{\iota(k),\iota(k^+)}$ for all $k\in [r]$ (i.e., $Q$ is a ``sub-polygon" of $P$).

\end{theorem}


In fact, if $Q=(\gamma,\overline \gamma)$, $\gamma\in [\Gamma(\Sigma)]$, we will construct a {\it canonical} polygon $P_\Delta(\gamma)$ as follows.

We need the following obvious fact.

\begin{lemma}
\label{le:first canonical triangle}
Let $\Delta$ be a triangulation of $\Sigma$ and let $\gamma\in [\Gamma(\Sigma)]\setminus \Delta$. Then there exists a unique (up to relabeling) triangle $\tau_1=(\gamma_1,\gamma_-,\gamma_+)\in \Delta^3$ such that
$n_{\gamma,\gamma_-}>0$ and the closest to $s(\gamma)$ intersection point of $\gamma$ with $\Delta$ is the intersection point of $\gamma$ and $\gamma_-$.

\end{lemma}

\begin{picture}(200,75)
\put(170,10){\includegraphics[scale=0.3]{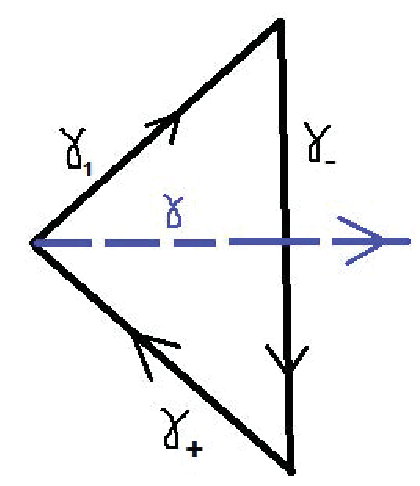}}
\put(150,0){The initial triangle for $\gamma$}
\end{picture}

\medskip

We refer to such a triangle as {\it initial} for $\gamma$.
Fix the initial triangle $\tau$ as in Lemma \ref{le:first canonical triangle} and denote by $\gamma^{(1)}$ the unique (class of) curve which starts as $\gamma_-$, follows this ``route" until the first intersection point of $\gamma_-$ and $\gamma$ and then ``becomes" $\gamma$.
Repeating this process, we obtain a new initial triangle $\tau_s=(\gamma_s,\gamma^{(s)}_-,\gamma^{(s)}_+)$ for $\gamma^{(s)}$, $s=1,\ldots,j-1$, where $j\ge 2$ is unique with $\gamma^{(j)}=\gamma_j\in \Delta$.
This process converges by induction in $n_{\gamma,\Delta}:=\sum\limits_{\gamma_0\in \Delta} n_{\gamma,\gamma_0}$ because $n_{\gamma,\Delta}>n_{\gamma^{(1)},\Delta}>\cdots >n_{\gamma^{(j)},\Delta}=0$. Denote $F_\Delta(\gamma):=(\gamma_1,\ldots,\gamma_j)\in \Delta^j$
and refer to this sequence as a {\it $\Delta$-factorization} of $\gamma$.
By definition,
$\gamma\in \gamma_1\circ \cdots \circ \gamma_j$ in the multi-groupoid $[\Gamma(\Sigma)]$,
which justifies the terminology.

%

Finally, we set $P_\Delta(\gamma):=(F_\Delta(\gamma),F_\Delta(\overline \gamma))$ and refer to it as the {\it canonical polygon} of
$\gamma$ in $\Delta$ due to the following obvious result.

\begin{lemma}
\label{le:canonical polygon}
Each $P_\Delta(\gamma)=(\gamma_1,\ldots,\gamma_n)$ is an $n$-gon in $\Delta$.

\end{lemma}

\vskip.3cm

\begin{picture}(10,90)
\put(0,10){\includegraphics[scale=0.25]{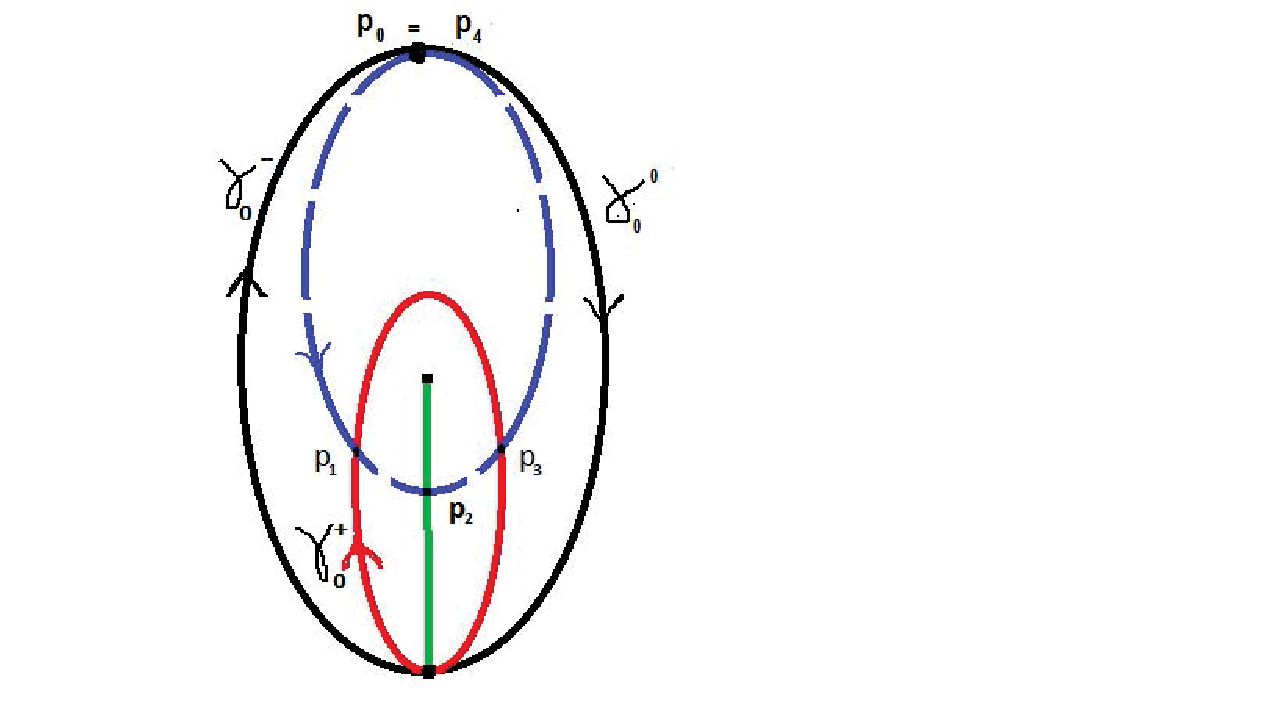}}
\put(100,10){\includegraphics[scale=0.4]{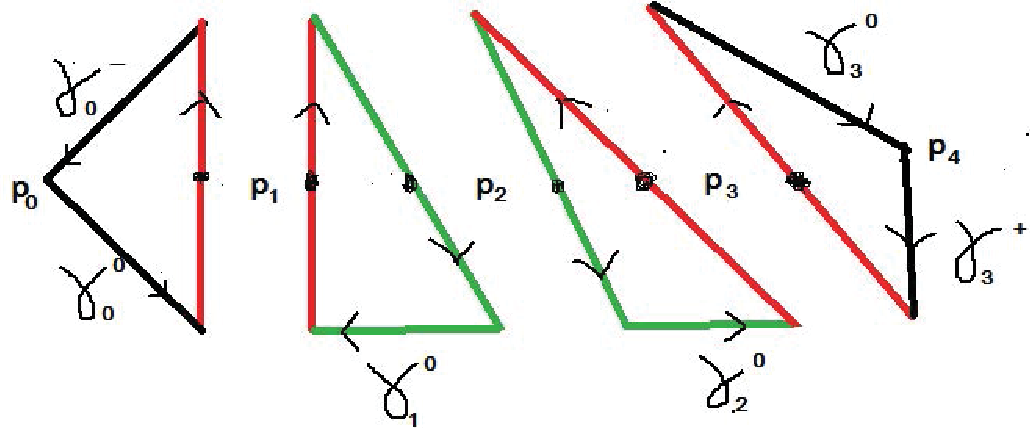}}
\put(310,10){\includegraphics[scale=0.45]{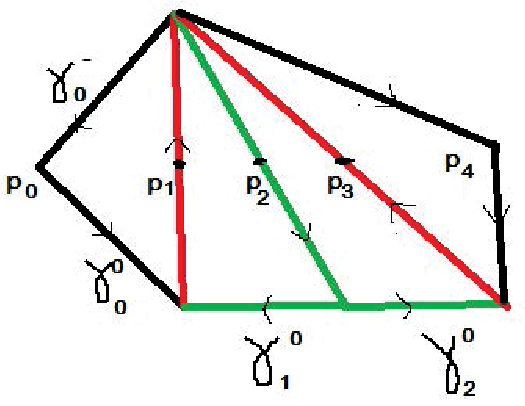}}
\put(100,0){Canonical polygon, no special punctures}
\end{picture}

\begin{picture}(110,130)
\put(12,30){\includegraphics[scale=0.35]{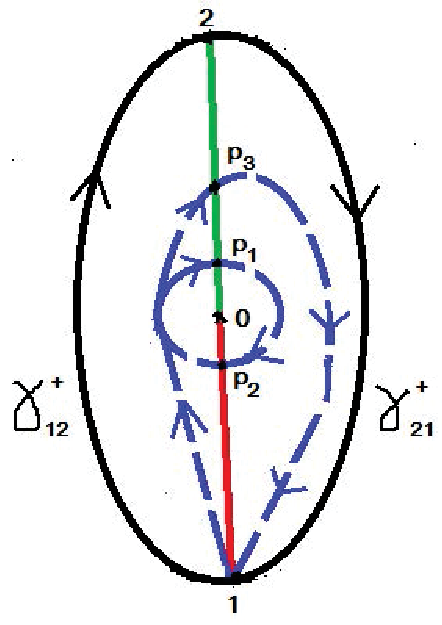}}
\put(100,40){\includegraphics[scale=0.35]{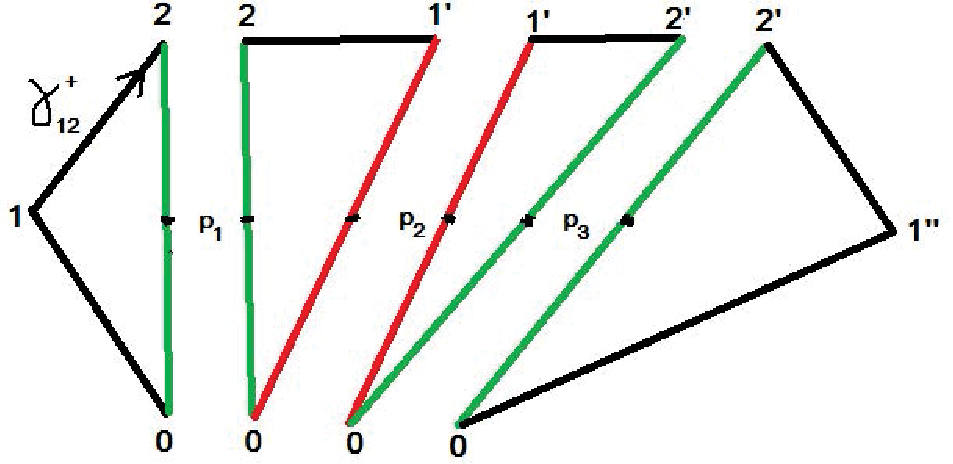}}
\put(270,40){\includegraphics[scale=0.35]{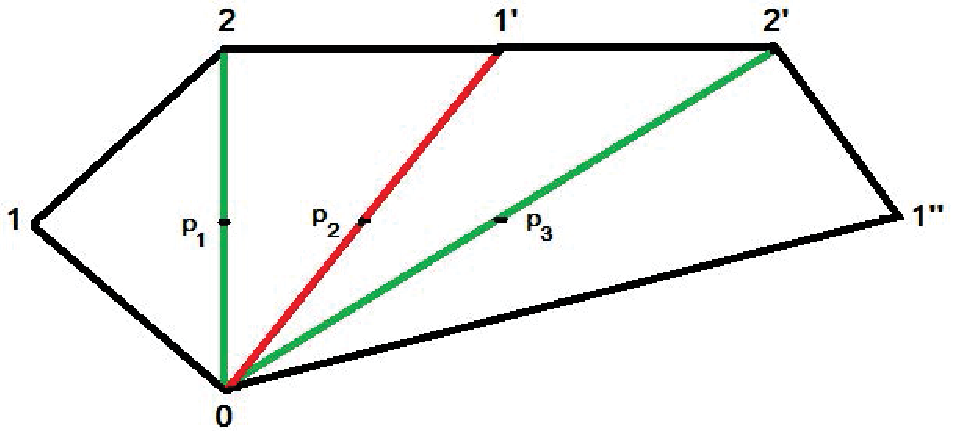}}
\put(80,20){Canonical polygon, no special punctures, $\gamma$ self-intersects}
\end{picture}

\begin{picture}(110,100)
\put(10,20){\includegraphics[scale=0.3]{080915-3.eps}}
\put(110,20){\includegraphics[scale=0.3]{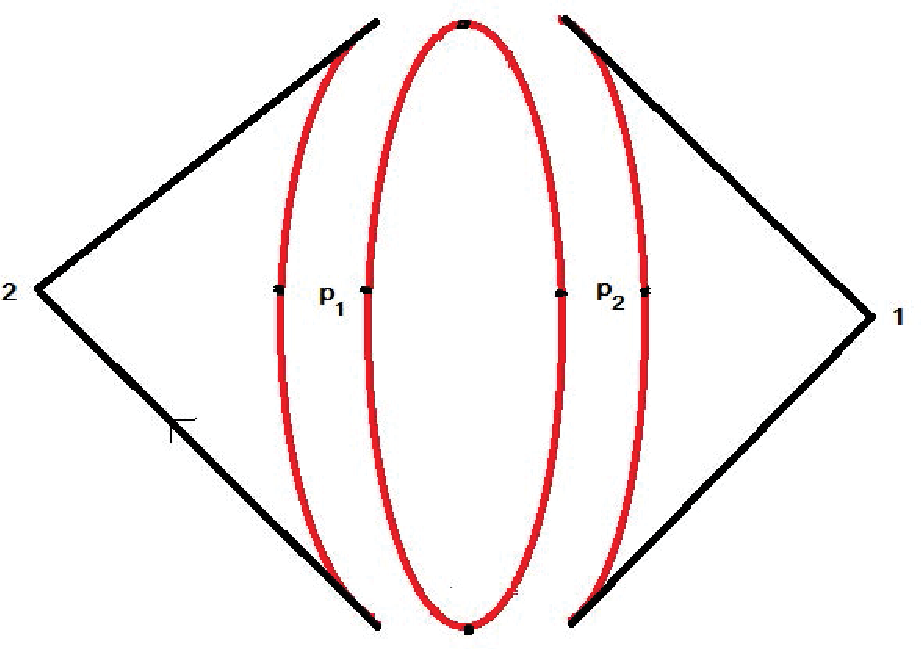}}
\put(260,20){\includegraphics[scale=0.32]{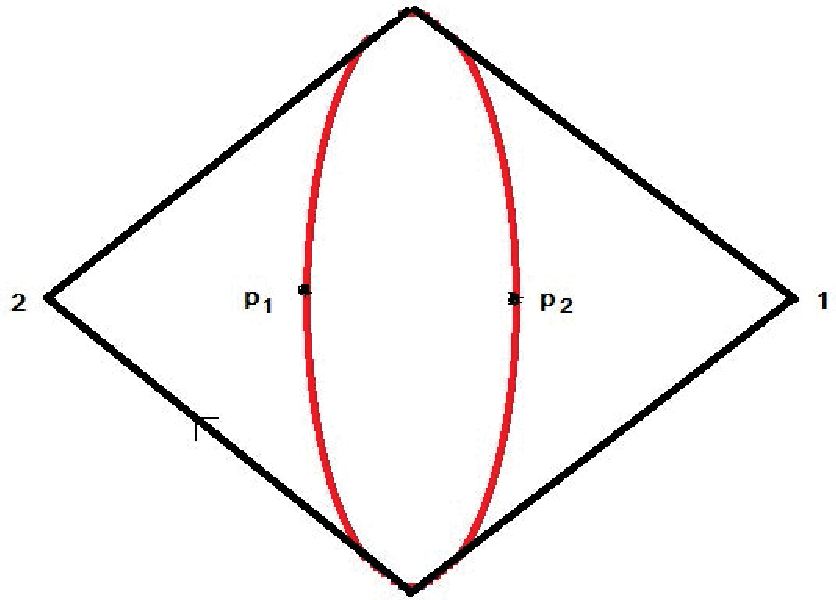}}
\put(100,0){Canonical polygon, one special puncture}
\end{picture}

\subsection{Triangle groups and their topological invariance}
\label{subsec:triangular groups}

For each triangulation  $\Delta$ of $\Sigma$ we define the {\it triangle group} $\TT_\Delta=\TT_\Delta(\Sigma)$
to be generated by all $t_\gamma^{\pm 1}$,  $\gamma\in \Delta$ subject to (same relations as in $\TT_\Sigma$):

$\bullet$  $t_{\gamma}=1$ if $\gamma$ is trivial.


$\bullet$ $t_{\gamma_1}t_{\overline \gamma_2}^{-1}t_{\gamma_3}=t_{\overline \gamma_3}t_{\gamma_2}^{-1}t_{\overline \gamma_1}$
for any triangle $T=(\gamma_1,\gamma_2,\gamma_3)$ in $\Delta$.

\medskip

Also, for each triangulation  $\Delta$ of $\Sigma$ denote by $\YY_\Delta$ the subgroup of $\TT_\Delta$ generated by:
$$y_{\gamma,\gamma'}:=t_{\overline \gamma}^{-1}t_{\gamma'}$$
for all $\gamma,\gamma'\in \Delta$  such that $(\gamma,\gamma',\gamma'')$ is a triangle in $\Delta$ for some
$\gamma''\in \Delta$.

\begin{theorem}
\label{th:TDeltaDelta' surface} For any two triangulations $\Delta$ and $\Delta'$ of a marked surface $\Sigma$ there exists a group isomorphism:
$$f_{\Delta,\Delta'}:\TT_\Delta\cong \TT_{\Delta'} $$
such that $f_{\Delta,\Delta'}(\YY_\Delta)=\YY_{\Delta'}$.
\end{theorem}

We prove Theorem \ref{th:TDeltaDelta' surface} in Section \ref{subsect:canonical covers}.


\begin{remark} Theorem \ref{th:TDeltaDelta' surface} implies that isomorphism classes of groups $\TT_\Delta$ and $\YY_\Sigma$ are topological  invariants of surfaces. However,  by contrast with Theorem \ref{th:functorial triangular}, we do not expect  
the assignments $\Sigma\mapsto \TT_\Delta$ to be functorial.

\end{remark}

%
%


Our next result is classification of  triangle groups of marked surfaces.

\begin{theorem}
\label{th:delta 1-relator} Let $\Sigma$ be a marked surface with the Euler characteristic $\chi(\Sigma)$, the set $I=I(\Sigma)\ne \emptyset$ of marked points,  the set $I_b\subseteq I$  of marked boundary points, and $h=|I_s|$ special punctures. Then for any triangulation $\Delta$ of $\Sigma$ one has:

\noindent (a) If $\Sigma$ has a boundary or special punctures, then $\TT_\Delta$ is a free group in:

$\bullet$ $|I|+1$ generators if $\Sigma$ is a disk with $|I|+|I_b|=2$, $h=0$.



$\bullet$ $2h+3|I|-4$ generators if $\Sigma$ is a disk with $|I|+|I_b|=2$, $h>0$.


$\bullet$ $2h+4(|I|-\chi(\Sigma))-|I_b|$ generators otherwise.

\noindent (b) If $\Sigma$ is a closed surface without special punctures, then $\TT_\Delta$ is:

$\bullet$    Trivial if $\Sigma$ is the sphere with $|I|=1$.

$\bullet$    A free group in $3|I|-4$ generators if $\Sigma$ is the sphere with $|I|\in \{2,3\}$.

$\bullet$    A free group in $2$ generators if $\Sigma$ is the real projective plane with $|I|=1$.

$\bullet$ A $1$-relator torsion free group (in the sense of Definition \ref{def:1-relator})  in $4(|I|-\chi(\Sigma))+1$ generators otherwise.

\end{theorem}

We prove Theorem \ref{th:delta 1-relator} in Section \ref{subsec:proof of theorem th:delta 1-relator} by choosing an appropriate triangulation of $\Sigma$.

\begin{remark} If $\Sigma$ has $r$ boundary components, then it is homotopy equivalent to a bunch of $g\ge r$ circles and $\chi(\Sigma)=1-g$.
If $\Sigma$ is a closed orientable (resp. non-orientable) surface, then it is  homeomorphic the connected sum of $g$ copies of the torus (resp. of the real projective plane) and $\chi(\Sigma)=2-2g$ (resp. $\chi(\Sigma)=2-g$).

\end{remark}

\begin{example} [See Example \ref{ex:closed surfaces detaled} below for details]
If $\Delta$ is a triangulation of the torus, the Klein bottle, the real projective plane respectively with one, one, two
(ordinary) punctures, then $\TT_\Delta$ is generated by variables $a,b,c,d,e$ subject to, respectively the following relations:

(i) for the torus with one puncture: $abcde=cbeda$;


(ii) for the Klein bottle with one puncture: $abcdc=ebeda$;



(iii) for the real projective plane with two punctures: $abcbc=ededa$.

\end{example}

\begin{example} If $\Delta$ is a triangulation of the sphere with $m+1$ punctures, we can view it as glued from a regular $2m$-gon with
$S=[m]\subset [2m]$, $\sigma(k)=m+1-k$, $\varepsilon(k)=+$ for $k\in [m]$ (with some notation from the proof of Theorem \ref{th:delta 1-relator} in Section \ref{subsec:proof of theorem th:delta 1-relator}).
Then $\TT_\Delta$ is generated by $c_1,\ldots,c_{2m}$, $t_3,\ldots,t_{2m-1}$
 subject to the relation
$t_3c_3t_4\cdots c_{2m-2}t_{2m-1}= t_{2m-1}c_3 t_{2m-2} \cdots  c_{2m-2}t_3$.
\end{example}

\subsection{Noncommutative Laurent Phenomenon for surfaces}
\label{subsect:Noncommutative Laurent Phenomenon for surfaces}

The following result
extends Noncommutative Laurent Phenomenon for $n$-gons (Theorem \ref{th:noncomlaurent n-gon}) to all marked surfaces.

\begin{theorem} (Noncommutative Laurent Phenomenon for surfaces)
\label{th:noncomlaurent surface}
Let $\Sigma$ be a marked surface and let $\Delta$ be a triangulation of  $\Sigma$.
Then for each  $\gamma\in [\Gamma(\Sigma)]$  the element $x_\gamma$ of ${\mathcal A}_\Sigma$ belongs to the subalgebra of $\AA_\Sigma$
generated by $x_{\gamma_0}^{\pm 1}$, $\gamma_0\in \Delta$.  More precisely, in the notation of Theorem \ref{th:noncomlaurent n-gon}, one has
\begin{equation}
\label{eq:noncomm schiffler surf}
x_\gamma=\sum_{\ii\in Adm_{\Delta^0}(1,j)} x_\ii \ ,
\end{equation}
where $\Delta^0$ is the  triangulation of $[n]$  assigned
(as in Theorem \ref{th:polygonal cover}(a))
to the canonical polygon $P_\Delta(\gamma)=(\gamma_1,\ldots,\gamma_n)$ in $\Delta$ with
$\gamma=\gamma_{1,j}$, and
we abbreviated
$$x_\ii:=x_{\gamma_{i_1,i_2}}x_{\gamma_{i_3,i_2}}^{-1}x_{\gamma_{i_3,i_4}}\cdots x_{\gamma_{i_{2m-1},i_{2m-2}}}^{-1}
x_{\gamma_{i_{2m-1},i_{2m}}}$$
for any sequence $\ii=(i_1,\ldots,i_{2m})\in [n]^{2m}$, $m\ge 1$.

%
%
\end{theorem}

We prove Theorem \ref{th:noncomlaurent surface} in Section  \ref{subsect:canonical covers}.

\begin{remark} Theorem \ref{th:noncomlaurent surface} is a noncommutative generalization of \cite[Theorem 6.1]{musiker-schiffler}.

\end{remark}

\begin{example}
\label{ex:Schiffler surface}  Let $\Sigma$ be a regular triangle with the clockwise vertex set $I=\{1,2,3\}$ and a special puncture $0$ in the center. For $i\in I$ denote by $\lambda_i$ the the special loop at $i$ around $0$.
As in Example \ref{ex:noncommutative (n,1) gon}, for $i,j\in I$, $i\ne j$ denote by $\gamma_{ij}^+$ (resp. $\gamma_{ij}^-$) the  curve  from $i$ to $j$ so that $0$ is to the right (resp. to the left) of the curve and abbreviate $x_i:=x_{\lambda_i}$,
$x_{ij}^\pm:=x_{\gamma_{ij}^\pm}$ for the corresponding generators of $\AA_\Sigma$.

Clearly, every triangulation of $\Sigma$ contains  $\gamma_{12}^+,\gamma_{21}^-,\gamma_{23}^+,\gamma_{32}^-,\gamma_{31}^+,\gamma_{13}^-$. Let $\Delta$ be the triangulation  of $\Sigma$ containing also $\gamma_1$  and $\gamma_{12}^-$.
Then \eqref{eq:noncomm schiffler surf}
reads:
%
%
$$
x_2=x^+_{21}x_1^{-1}x^-_{12} + x^-_{21}x_1^{-1}x^+_{12}, \ \
x^-_{23}=x^-_{21}(x^+_{21})^{-1}x^+_{23} + x^+_{21}x_1^{-1}x^-_{13}
+ x^-_{21}x_1^{-1}x^+_{12}(x^-_{12})^{-1}x^-_{13}\ ,
$$
$$
x_3=x^+_{31}x_1^{-1}x^-_{13} + x^+_{31}(x^+_{21})^{-1}x^-_{21}(x^+_{21})^{-1}x^+_{23} +
x^-_{32}(x^-_{12})^{-1}x_1(x^+_{21})^{-1}x^+_{23}$$
$$ + x^-_{32}(x^-_{12})^{-1}x^+_{12}(x^-_{12})^{-1}x^-_{13} +
x^+_{31}(x^+_{21})^{-1}x^-_{21}x_1^{-1}x^+_{12}(x^-_{12})^{-1}x^-_{13}\ .
$$

Let  $\hat \Sigma$ be as in Definition \ref{def:gluing holes}.
Therefore, simple curves on $\hat \Sigma$ are those on $\Sigma$ plus six additional ones: directed intervals $\gamma_{0,i}$ from $0$ to each $i$ and
$\gamma_{i,0}:=\overline \gamma_{0,i}$.
We abbreviate the generators of  $\AA_{\hat \Sigma}$ same way as in $\AA_\Sigma$ and $x_{0,i}:=x_{\gamma_{0,i}}$, $x_{i,0}:=x_{\gamma_{i,0}}$.

Let $\hat \Delta$  be the triangulation of $\hat \Sigma$ obtained from $\Delta$ by adding the intervals $\gamma_{0,1}$ and $\gamma_{1,0}$.
Then \eqref{eq:noncomm schiffler surf} reads:
%
%
$$
x_3=x^+_{31}x_1^{-1}x^-_{13} + x^+_{31}(x^+_{21})^{-1}x^-_{21}(x^+_{21})^{-1}x^+_{23} +
x^-_{32}(x^-_{12})^{-1}x_1(x^+_{21})^{-1}x^+_{23}$$
$$ + x^-_{32}(x^-_{12})^{-1}x^+_{12}(x^-_{12})^{-1}x^-_{13} +
x^+_{31}(x^+_{21})^{-1}x^-_{21}x_1^{-1}x^+_{12}(x^-_{12})^{-1}x^-_{13} +
$$
$$
+ x^+_{31}(x^+_{21})^{-1}x^+_{23} + x^-_{32}(x^-_{12})^{-1}x^-_{13} + x^+_{31}x_1^{-1}x^+_{12}(x^-_{12})^{-1}x^-_{13} +
x^+_{31}(x^+_{21})^{-1}x^-_{21}x_1^{-1}x^-_{13}\ .
$$

\end{example}

\subsection{Noncommutative $(n,1)$-gon}
\label{ex:noncommutative (n,1) gon}
In this section we consider   the $(n,1)$-gon $\Sigma=P_n(1)$  (with the clockwise ordering of the set $[n]=I_b(P(n,1))$). We abbreviate $\AA_{n,1}:={\mathcal A}_\Sigma$ and refer to it as the {\it noncommutative $(n,1)$-gon}. Clearly,
 ${\mathcal A}_{n,1}$ is  generated by
$x_{ij}^{\pm}:=x_{\gamma_{ij}^\pm}$ and $(x_{ij}^{\pm})^{-1}$, $i,j\in [n]$, where $\gamma_{ij}^\pm$ is the curve corresponding to $(i,j,\pm)$ under the bijection in Lemma  \ref{le:finite Gamma} where $x_{ii}^+=x_{ii}^-$ for $i\in [n]$ (we abbreviate $x_i:=x_{ii}^+=x_{ii}^-$). The following is immediate.

\begin{lemma}
\label{le:C_n} The algebra ${\mathcal A}_{n,1}$ is  generated by
$(x_{ij}^{\pm})^{\pm 1}$, $i,j\in [n]$ subject to:


(i) (triangle relations) For any distinct  $i,j,k\in [n]$:
$$x_{ij}^+(x_{kj}^+)^{-1}x_{ki}^+=x_{ik}^-(x_{jk}^-)^{-1}x_{ji}^-, x_{ij}^+(x_{kj}^-)^{-1}x_{ki}^+=x_{ik}^-(x_{jk}^+)^{-1}x_{ji}^-\ .$$


(ii) ($2$-gon exchange relations) For any distinct $i,j\in [n]$:
$$x_j=x_{ji}^+x_i^{-1}x_{ij}^-+x_{ji}^-x_i^{-1}x_{ij}^+\ .$$

(iii) ($4$-gon exchange relations) For any cyclic $(i,j,k,\ell)$ in $[n]$ and $\varepsilon \in \{-,+\}$:
$$x_{j\ell}^+=x_{jk}^+(x_{ik}^\varepsilon)^{-1} x_{i\ell}^\varepsilon+ x_{ji}^{-\varepsilon}(x_{ki}^{-\varepsilon})^{-1}x_{k\ell}^+\ ,
x_{i\ell}^+=x_{ik}^\varepsilon(x_{jk}^{-\varepsilon})^{-1} x_{j\ell}^{-\varepsilon}+ x_{ij}^{-\varepsilon}(x_{ki}^\varepsilon)^{-1}x_{k\ell}^\varepsilon$$
$$x_{j\ell}^-=x_{jk}^\varepsilon(x_{ik}^\varepsilon)^{-1} x_{i\ell}^-+ x_{ji}^-(x_{kj}^{-\varepsilon})^{-1}x_{k\ell}^{-\varepsilon}\ ,
x_{\ell i}^-=x_{\ell j}^{\varepsilon} (x_{kj}^\varepsilon)^{-1}x_{ki}^{-\varepsilon} + x_{\ell k}^{-\varepsilon}(x_{jk}^{-\varepsilon})^{-1}x_{ji}^\varepsilon$$


%
\end{lemma}

Clearly, the assignments $x_{ij}^\pm\mapsto x_{ji}^{\mp}$ define an involutive anti-automorphism of ${\mathcal A}_{n,1}$.
One can easily show

For each $n\ge 1$ define a map $\pi:[2n] \to [n]$ by
$\pi(i)=\begin{cases}
i &  \text{if $i\in [n]$}  \\
i-n &  \text{if $i\notin [n]$}
\end{cases}
$.

Also for distinct $i,j\in [2n]$  define the sign $\varepsilon_{ij}\in \{-,+\}$ by setting $\varepsilon_{ij}:=+$ if the clockwise arc from $i$ to $j$ is shorter than the clockwise arc from $i$ to $i+n$ and $\varepsilon_{ij}:=-$ otherwise.

Note that the restriction of the function  $\hat f:\CC\to \CC$ given by $z\mapsto z^2$ to the unit disk $D\subset \CC$ centered at $0$ is a map $f:D\to D$ hence for each $n\ge 1$ it
is a morphism $f_n:P_{2n}\to P_n(1)$ in {\bf Surf} for all $n\ge 1$ (where the marked boundary points are appropriate roots of unity and the special puncture in $P_n(1)$ is the center $0$ of $D$). The following is immediate corollary of Theorems \ref{th:I_s-homotopy equivalence} and \ref{th:functorial triangular}.

\begin{corollary} For each $n\ge 1$ one has:

$\bullet$  The morphism $f_n$ in {\bf Surf} defines a surjective map $\Gamma(P_{2n})=[2n]\times [2n]\twoheadrightarrow [n]\times [n]\times \{-,+\}=[\Gamma(P(n,1))]$ given by $(ij)\mapsto \gamma_{\pi(i),\pi(j)}^{\varepsilon_{ij}}$ for all distinct $i,j\in [2n]$.

$\bullet$ The assignments
$x_{ij}\mapsto 
x_{\pi(i),\pi(j)}^{\varepsilon_{ij}}
$ for all distinct $i,j\in [2n]$, define an epimorphism of algebras $(f_n)_*:\AA_{2n}\twoheadrightarrow {\mathcal A}_{n,1}$.

\end{corollary}

\begin{remark}
\label{rem:intersecting triangle}
For any $1\le i<j<k\le n$, the triple $(\gamma_{ij}^-,\gamma_{jk}^-,\gamma_{ki}^-)$ is a triangle in $\Sigma=P(n,1)$ because it is the image of the triangle $(i,j+n,k)$ in $[2n]$ under the above morphism $f_n:P_{2n}\to P_n(1)$. Note, however, that all intersections $\gamma_{ij}^-\cap \gamma_{jk}^-$, $\gamma_{ij}^-\cap \gamma_{ki}^-$, $\gamma_{jk}^-\cap \gamma_{ki}^-$ are non-empty.

\end{remark}

\subsection{Universal localizations of noncommutative surfaces}
\label{subsec:niversal localizations of noncommutative surfaces}
Generalizing \eqref{eq:i_Delta}, for any triangulation $\Delta$ of any marked surface $\Sigma$ let $\AA_\Delta$ be the subalgebra of $\AA_\Sigma$ generated by all $x_\gamma$, $\gamma\in [\Gamma(\Sigma)]$ and all $x_{\gamma_0}^{-1}$, $\gamma_0\in \Delta$.

Clearly, the assignments $t_\gamma\mapsto x_\gamma$,  $\gamma\in \Delta$ define a homomorphisms of algebras:
\begin{equation}
\label{eq:i_Delta surface}
\ii_\Delta:\QQ\TT_\Delta \to  \AA_\Delta \ .
\end{equation}

The following result is a generalization of Theorem \ref{th:A_Delta} to all marked surfaces.

\begin{theorem}
\label{th:A_Delta sigma}
For each triangulation $\Delta$ of $\Sigma$ one has:

(a) The homomorphism $\ii_\Delta$ given by \eqref{eq:i_Delta surface} is an isomorphism of algebras.

(b)   $\AA_\Sigma=\AA_\Delta[{\bf S}^{-1}]$, where ${\bf S}$ is the submonoid of $\AA_\Delta\setminus \{0\}$ generated by all $x_\gamma$, $\gamma\in [\Gamma(\Sigma)]$.

\end{theorem}

%
 %
%

We prove Theorem \ref{th:A_Delta sigma} in Section \ref{subsec:proof Laurent}.

Theorems \ref{th:delta 1-relator}, \ref{th:A_Delta sigma}, and \ref{th:invertible almost free} imply the following.

\begin{corollary}
\label{cor:injective TDelta surface}
For each triangulation $\Delta$ of $\Sigma$ the homomorphism \eqref{eq:i_Delta surface} is injective.

\end{corollary}

Theorem \ref{th:A_Delta sigma} implies that for each $\Sigma$ the natural homomorphism $\QQ\TT_\Delta\hookrightarrow Frac(\QQ\TT_\Delta)$ defines a homomorphism of algebras:
\begin{equation}
\label{eqLASigma to fractions}
\AA_\Sigma\to Frac(\QQ\TT_\Delta)\ .
\end{equation}

In view of Theorem \ref{th:invertible almost free}, we propose the following conjecture.

\begin{conjecture}
\label{conj:injective skew field surf}
For each $\Sigma$ the homomorphism \eqref{eqLASigma to fractions} is injective, e.g., the submonoid $S_\Delta$ of $\QQ \TT_\Delta\setminus \{0\}$ is divisible in the sense of Definition \ref{def:divisible submonoid}.

\end{conjecture}


\begin{remark} Conjecture \ref{conj:injective skew field surf} generalizes the expected  injectivity of \eqref{eq:phi from A'}. To prove Conjecture  \ref{conj:injective skew field surf} for non-closed surfaces (i.e., with free $\TT_\Delta$ according to Theorem \ref{th:delta 1-relator}) it would suffice to show that the monoid $S_\Delta$ is generated by $\QQ^\times\cdot \TT_\Delta$ and a subset of prime elements in $\QQ \TT_\Delta$.

\end{remark}

\subsection{Noncommutative angles and regular elements in noncommutative surfaces}
\label{subsec:regular surface}
Similarly to Section \ref{subsec:regular polygon}, for each triangle $(\gamma_1,\gamma_2,\gamma_3)$ denote
by $T_{\gamma_1,\gamma_2,\gamma_3}$ the element of $\AA_\Sigma$ given by:
\begin{equation}
\label{eq:general angle}
T_{\gamma_1,\gamma_2,\gamma_3}=x_{\overline \gamma_1}^{-1}x_{\gamma_2} x_{\overline \gamma_3}^{-1}
\end{equation}
and refer to it as a {\it noncommutative angle} of $(\gamma_1,\gamma_2,\gamma_3)$ at $s(\gamma_1)=t(\gamma_3)$.

Given a triangulation $\Delta$ of $\Sigma$, for any $i\in I$ define the {\it total angle} $T_i^\Delta$ at $i\in I$ by:
\begin{equation}
\label{eq:total angle}
T_i^\Delta:=\sum T_{\gamma_1,\gamma_2,\gamma_3} \ ,
\end{equation}
where the summation is over all clockwise triangles $(\gamma_1,\gamma_2,\gamma_3)$ in $\Delta$ such that $s(\gamma_1)=i$.

\begin{theorem}
\label{th:total anglle surface}
For any triangulations $\Delta$, $\Delta'$ of $\Sigma$ and $i\in I$ one has:
$$T_i^\Delta=T_i^{\Delta'}\ .$$

\end{theorem}

Therefore, in what follows, we simply denote $T_i:=T_i^\Delta$ for any triangulation $\Delta$ of $\Sigma$.

Furthermore, denote by $\cU_\Sigma$ the subalgebra of $\AA_\Sigma$ generated by all $x_\gamma$, $\gamma\in [\Gamma(\Sigma)]$,
$x_{\gamma_0}^{-1}$, $\gamma_0\in \partial \Gamma(\Sigma)$ and all total angles $T_i$.

In particular, the algebra $\cU_n$ from \ref{subsec:regular polygon} is naturally isomorphic to $\cU_{P_n}$.
The following is an analogue of Lemma \ref{le:upper bound type A}.

\begin{lemma}
\label{le:upper bound surfaces} The algebra $\cU_\Sigma$ satisfies the following relations:

(a) (reduced triangle relations) for all triangles $(\gamma_1,\gamma_2,\gamma_3)$ in $[\Gamma(\Sigma)]$ such that $\gamma_2$ is a boundary curve:
\begin{equation}
\label{eq:triangle relations upper surface}
x_{\gamma_1}x_{\overline \gamma_2}^{-1}x_{\gamma_3}=x_{\overline \gamma_3}x_{\gamma_2}^{-1}x_{\overline \gamma_1} \ .
\end{equation}

(b) (reduced exchange relations) for all quadrilaterals $(\gamma_1,\gamma_2,\gamma_3,\gamma_4)$ in $\Sigma$ such that $\gamma_2,\gamma_3$ are boundary curves:
\begin{equation}
\label{eq:exchange relations upper surface}
x_{\gamma_{13}}x_{23}^{-1}x_{\gamma_{24}}=
x_{\gamma_{14}}+x_{\gamma_{12}}x_{\gamma_{32}}^{-1}x_{\gamma_{34}}
\end{equation}



\end{lemma}


\begin{remark} It is natural to conjecture that the relations  \eqref{eq:triangle relations upper surface} and \eqref{eq:exchange relations upper surface} are defining for ${\mathcal U}_\Sigma$.

\end{remark}

Noncommutative Laurent phenomenon \eqref{eq:noncomm schiffler surf} guarantees that  $\cU_\Sigma$ belongs to each subalgebra $\AA_\Delta\subset \AA_\Sigma$.

The following is an analogue of Conjecture \ref{conj:upper bound type A}.

\begin{conjecture} For each $n\ge 2$ one has:
\label{conj:upper bound surface}
\begin{equation}
\label{eq:ij_delta cup sharp surface} \cU_\Sigma= \bigcap_{\Delta}{\mathcal A}_\Delta \ ,
\end{equation}
where the intersection is over all triangulations $\Delta$ of $\Sigma$.

\end{conjecture}

We say that an element of $\AA_\Sigma$ is {\it regular} if it belongs to each subalgebra $\AA_\Delta$ as $\Delta$ runs over all triangulations  of $\Sigma$. Thus,
similarly to Section \ref{subsec:regular polygon}, Conjecture \ref{conj:upper bound surface}  asserts that regular elements of ${\mathcal A}_\Sigma$ belong to $\cU_\Sigma$.


%
%
%
%
%
%
%
%

\subsection{Noncommutative cohomology of  surfaces}

Given a surface $\Sigma$, for each triangle $(\gamma_1,\gamma_2,\gamma_3)$ in $\Sigma$ we define the element $\tau_{\gamma_1,\gamma_2,\gamma_3}\in \AA_\Sigma$ (in notation \eqref{eq:general angle}) by:
$$\tau_{\gamma_1,\gamma_2,\gamma_3}=T_{\gamma_1,\gamma_2,\gamma_3}+T_{\gamma_2,\gamma_3,\gamma_1}+T_{\gamma_3,\gamma_1,\gamma_2} \ .$$
That is, $\tau_{\gamma_1,\gamma_2,\gamma_3}$ is the sum of all noncommutative angles of the triangle $(\gamma_1,\gamma_2,\gamma_3)$.

Then define the algebra ${\mathcal H}(\Sigma)$ to be the quotient of $\AA_\Sigma$ by the ideal generated by all $\tau_{(\gamma_1,\gamma_2,\gamma_3)}-\tau_{(\gamma'_1,\gamma'_2,\gamma'_3)}$ as $(\gamma_1,\gamma_2,\gamma_3)$ and $(\gamma'_1,\gamma'_2,\gamma'_3)$ run independently over all triangles of $\Sigma$.
We refer to ${\mathcal H}(\Sigma)$ as the {\it noncommutative cohomology} of $\Sigma$.

This notation is justified by the following construction.

Fix a triangulation $\Delta$ of $\Sigma$. For each loop $\theta$ in $\Sigma$  which does not pass through marked points, define the element
$[\theta]'_\Delta \in \AA_\Delta$ by:
$$[\theta]'_\Delta=\sum \varepsilon_{\gamma_1,\gamma_2,\gamma_3}(\theta) \cdot T_{\gamma_1,\gamma_2,\gamma_3}\ ,$$
the summation is over all clockwise triangles $(\gamma_1,\gamma_2,\gamma_3)$ in $\Delta$ such that $\theta$ intersects $\gamma_1$ and $\gamma_2$ (but not $\gamma_3$) and
$\varepsilon_{\gamma_1,\gamma_2,\gamma_3}(\theta):=\begin{cases}
1 & \text{if $\gamma_3$ is to the right of $\theta$}\\
-1 & \text{if $\gamma_3$ is to the left of $\theta$}\\
\end{cases}$.

Note that if $\theta=\theta_i$ is a (small) clockwise loop around a puncture $i\in I$, then $[\theta]'_\Delta=T_i^\Delta$, the total angle at $i$ (defined in \eqref{eq:total angle}).

Furthermore, define $[\theta]_\Delta\in {\mathcal H}(\Sigma)$ by
$$[\theta]_\Delta:= \pi(\ii_\Delta([\theta]'_\Delta))\ ,$$
where $\ii_\Delta$ is the homomorphism $\QQ \TT_\Delta\rightarrow \AA_\Sigma$ given by \eqref{eq:i_Delta surface} and $\pi:\AA_\Sigma\to {\mathcal H}(\Sigma)$ is the canonical epimorphism.

The following immediate result is an analogue of Theorem \ref{th:total anglle surface}.

\begin{theorem}
Given a loop on $\Sigma$ not passing through marked points, then for any triangulations $\Delta$ and $\Delta'$ of $\Sigma$ one has:
$$[\theta]_{\Delta'}=[\theta]_\Delta \ .$$

\end{theorem}

This allows us to define a noncommutative loop $[\theta]\in {\mathcal H}(\Sigma)$ by $[\theta]:=[\theta]_\Delta$ for any triangulation $\Delta$ of $\Sigma$.

\subsection{Proof of Theorems \ref{th:category surf}, \ref{th:I_s-homotopy equivalence}, \ref{th:functorial triangular}, \ref{th:TDeltaDelta' surface}, and \ref{th:noncomlaurent surface}}
\label{subsect:canonical covers}
${~}$

\noindent {\bf Proof of Theorem \ref{th:category surf}}.
Clearly, the composition $f'\circ f:\underline \Sigma\to \underline \Sigma''$ is a continuous map with finite fibers. Also,
$$(f'\circ f)^{-1}(I(\Sigma''))=f^{-1}({f'}^{-1}(I(\Sigma'')))=f^{-1}(I(\Sigma''))=I(\Sigma)\ ,$$
$$(f'\circ f)(I_s(\Sigma))=f'(f(I_s(\Sigma))\subset  f'(I_s(\Sigma'))\subset I_s(\Sigma'') \ .$$
This verifies the first requirement of Definition \ref{def:proper map} for $f'\circ f$.

Furthermore, prove that $I^{f'\circ f}=I^f\sqcup f^{-1}(I^{f'})$. Indeed,
$$I^{f'\circ f}=(f'\circ f)^{-1}(I_s(\Sigma''))\setminus I_s(\Sigma)=f^{-1}({f'}^{-1}(I_s(\Sigma''))\setminus I_s(\Sigma)$$
$$=
f^{-1}(I_s(\Sigma')\sqcup I^{f'})\setminus I_s(\Sigma)=(f^{-1}(I_s(\Sigma'))\sqcup f^{-1}(I^{f'}))\setminus I_s(\Sigma)
=I^f\sqcup f^{-1}(I^{f'})$$
since ${f'}^{-1}(I_s(\Sigma''))=I_s(\Sigma')\sqcup I^{f'}$, $f^{-1}(I_s(\Sigma'))= f^{-1}(I_s(\Sigma'))\setminus I^f$,
$f^{-1}(I^{f'})\cap I_s(\Sigma)=\emptyset$, and $f^{-1}(A\sqcup B)=f^{-1}(A)\sqcup f^{-1}(B)$ for any disjoint subsets $A$ and $B$ of
$\underline \Sigma'$.

Let now $p\in \underline \Sigma \setminus I^{f'\circ f}$. By above, this is equivalent to that
$p\in \underline \Sigma \setminus I^f$ and $f(p)\in \underline \Sigma' \setminus I^{f'}$.
Hence there is a neighborhood ${\mathcal O}_p$ of $p$ in $\underline \Sigma$ (${\mathcal O}_p$ is a half-neighborhood if $p\in \partial \underline \Sigma$) such that the restriction of
$f$ to ${\mathcal O}_p$ is injective and a (half-)neighborhood ${\mathcal O}_{f(p)}$ of $f(p)$ in $\underline \Sigma'$ such that the restriction of
$f'$ to ${\mathcal O}_{f(p)}$ is injective.  In particular, ${\mathcal O}_p':=f^{-1}({\mathcal O}_{f(p)})$ is a neighborhood of $p$ in $\underline \Sigma$ and the restriction of $f'\circ f$ to ${\mathcal O}_p'$ is injective. This verifies the second requirement of Definition \ref{def:proper map} for $f'\circ f$.

Let now $p\in I^{f'\circ f}$. By above, this is equivalent to that either $p\in I^f$ or $f(p)\in I^{f'}$.

In the first case, clearly, $f(p)\in \underline \Sigma'\setminus I^{f'}$, therefore
there is a neighborhood ${\mathcal O}_{f(p)}$ of $f(p)$ in $\underline \Sigma'$ such that the restriction
of $f'$ to ${\mathcal O}_{f(p)}$ is injective and
a neighborhood ${\mathcal U}_p$ of $p$ in $\underline \Sigma$ such that the restriction of $f$ to ${\mathcal U}_p$ is
a two-fold cover of the neighborhood ${\mathcal O}'_p=f({\mathcal U}_p)$ ramified at $f(p)$. Therefore, the restriction of $f$ to
the neighborhood ${\mathcal U}'_p=f^{-1}({\mathcal O}_p\cap {\mathcal O}'_p)$  is a two-fold cover of
${\mathcal O}_p\cap {\mathcal O}'_p$ ramified at $f(p)$ and the restriction of $f'$ to ${\mathcal O}_p\cap {\mathcal O}'_p$ is injective.
Thus, the restriction of $f'\circ f$ to  ${\mathcal U}'_p$  is a two-fold cover of  $f({\mathcal O}_p\cap {\mathcal O}'_p)$ ramified at $(f'\circ f)(p)$.

In the second case, clearly, $p\in \underline \Sigma \setminus I^f$, therefore
there is a neighborhood ${\mathcal O}_{p}$ of $p$ in $\underline \Sigma$ such that the restriction
of $f$ to ${\mathcal O}_p$ is injective and
a neighborhood ${\mathcal U}_{f(p)}$ of $f(p)$ in $\underline \Sigma'$ such that the restriction of $f'$ to ${\mathcal U}_{f(p)}$ is
a two-fold cover of the neighborhood ${\mathcal O}_{f'(f(p))}=f({\mathcal U}_p)$ ramified at $f(f'(p))$. Therefore, the restriction of $f'$ to
the neighborhood ${\mathcal U}'_{f(p)}=f({\mathcal O}_p)\cap {\mathcal U}_{f(p)}$  is a two-fold cover of
$f'({\mathcal U}'_{f(p)})$ ramified at $f'(f(p))$ and the restriction of $f$ to ${\mathcal O}'_p=f^{-1}({\mathcal U}'_{f(p)})$ is injective.
Thus, the restriction of $f'\circ f$ to  ${\mathcal O}'_p$  is a two-fold cover of  $f'({\mathcal U}'_{f(p)})$ ramified at $(f'\circ f)(p)$.

This verifies the last requirement of Definition \ref{def:proper map} for $f'\circ f$.

The theorem is proved. \endproof

\noindent {\bf Proof of Theorem \ref{th:I_s-homotopy equivalence}}. Without loss of generality, it suffices to prove the first assertion in the case when $C\subset C'$ and
$C'\setminus C$ is a single loop around $i\in I^f$ not enclosing any points  $I(\Sigma)\cup I_s(\Sigma)\cup I^f\setminus\{i\}$ (where we regard $C$ and $C'$ as subsets of $\underline \Sigma$).  Moreover, it suffices to  take
$C=\{p\}$ for some $p\in \underline \Sigma$, $p\ne i$, so that $C'$ is a simple loop at $p$ around $i$
(e.g., $C'$ is contractible to $p$ in $\underline \Sigma\setminus I(\Sigma)$).

By definition, there is a neighborhood ${\mathcal U}_p$ of $p$ such that the restriction of $f$ to ${\mathcal U}_p$ is a two-fold cover of
$f({\mathcal U}_p)$ ramified at $f(p)$. Once again, without loss of generality, we may assume that $C'$ intersects ${\mathcal U}_p$ and there exist exactly two distinct points $p',p''\in C$ such that $f(p')=f(p'')$. This implies that $f(C')\subset \underline \Sigma'$
is a (self-intersecting) loop at $f(p)$ with a single self-intersection point $f(p')=f(p'')$.
If we denote by $\gamma'$ the equivalence class of $f(C')$ in $\underline \Sigma'\setminus (I(\Sigma)\cup I_s(\Sigma))\cup\{f(p)\}$,
then, clearly, $[\gamma']_i$ is trivial.

This proves (a).

Parts (b), (c) and (d) follow.

The theorem is proved.
\endproof

\noindent{\bf Proof of Theorem \ref{th:functorial triangular}}.
We need the following fact.

\begin{lemma} In the notation of Theorem \ref{th:I_s-homotopy equivalence}, for any polygon  $P=(\gamma_1,\ldots,\gamma_n)$ in $\Sigma$ the tuple $f(P)=(f(\gamma_1),\ldots,f(\gamma_n))$
is a polygon in $\Sigma'$.

\end{lemma}

\begin{proof} Indeed, let $P=(\gamma_1,\ldots,\gamma_n)$ be a polygon in $\Sigma$ and let $g:P_n\to \Sigma$ be an accompanying morpism. Then $g'=f\circ g$ is a morphism $P_n\to
\Sigma'$ in {\bf Surf} such that $g'(i,i^+)=f(\gamma_i)$ for $i\in [n]$, i.e., $f(P)$ is an $n$-gon in $\Sigma')$.
%

The lemma is proved.
\end{proof}



Thus the  triangle relations in $\TT_\Sigma$ are carried by $f_\star$ to those in $\TT_{\Sigma'}$. This proves the assertion for groups.

Likewise, the triangle and exchange relations in $\AA_\Sigma$ are carried by $f_*$ to those in $\AA_{\Sigma'}$.
This proves the assertion for algebras. The commutativity of the diagram \eqref{eq:TA commutative} follows.

The theorem is proved.
\endproof

\noindent {\bf Proof of Theorem \ref{th:TDeltaDelta' surface}}. It suffices to prove the assertion only for neighboring triangulations $\Delta$ and $\Delta'$, i.e., for a quadrilateral
$(\gamma_1,\gamma_2,\gamma_3,\gamma_4)$ in $\Delta$ such that
$\Delta\setminus \Delta'=\{\gamma_{13},\gamma_{31}\}$ and $\Delta'\setminus \Delta=\{\gamma_{24},\gamma_{42}\}$.

The following result is obvious.

\begin{lemma}
\label{le:triangular noncanonical isomorphisms1}
In the notation as above, the assignment
$t_{\gamma}\mapsto
\begin{cases}
t_{\gamma_{12}}t_{\gamma_{42}}^{-1} t_{\gamma_{43}} & \text{if $\gamma=\gamma_{13}$}\\
t_{\gamma_{34}}t_{\gamma_{24}}^{-1} t_{\gamma_{21}} & \text{if $\gamma=\gamma_{31}$}\\
t_\gamma & \text{otherwise}
\end{cases}
$
for $i,j\in [4]$, $i\ne j$, defines an isomorphism $\varphi_{\Delta,\Delta'}:\TT_{\Delta}\widetilde \to \TT_{\Delta'}$.
\end{lemma}

%
%

The second assertion follows immediately because one has for $\gamma,\gamma'\in \Delta$, $\gamma'\notin \{\gamma_{13},\gamma_{31}\}$:
$$f_{\Delta,\Delta'}(y_{\gamma,\gamma'})=\begin{cases}
t_{\gamma_{21}}^{-1}t_{\gamma_{24}} t_{\gamma_{34}}^{-1}t_{\gamma'} & \text{if $\gamma=\gamma_{13}$}\\
t_{\gamma_{43}}^{-1}t_{\gamma_{42}} t_{\gamma_{12}}^{-1}t_{\gamma'} & \text{if $\gamma=\gamma_{31}$}\\
t_\gamma^{-1} t_{\gamma'} & \text{otherwise}
\end{cases}=\begin{cases}
y_{\gamma_{12},\gamma_{24}}y_{\gamma_{43},\gamma'} & \text{if $\gamma=\gamma_{13}$}\\
y_{\gamma_{34},\gamma_{42}}y_{\gamma_{21},\gamma'} & \text{if $\gamma=\gamma_{31}$}\\
y_{\gamma,\gamma'} & \text{otherwise}
\end{cases}\in \YY_{\Delta'}\ .$$

This proves the theorem.
\endproof

\noindent {\bf Proof of Theorem \ref{th:noncomlaurent surface}}. Indeed, let $f:P_n\to \Sigma$ be an accompanying map for the canonical polygon $P_\Delta(\gamma)=(\gamma_1,\ldots,\gamma_n)$. Then, by Theorem \ref{th:functorial triangular}, the assignments $x_{ij}\mapsto x_{\gamma_{ij}}$
define an algebra homomorphism $f_*:\AA_n\to \AA_\Sigma$, where $\AA_n=\AA_{P_n}$ is the noncommutative $n$-gon as in Section \ref{subsect:Noncommutative Laurent Phenomenon}. Applying $f_\ast$ to \eqref{eq:noncomlaurent n-gon} with $i=1$ yields \eqref{eq:noncomm schiffler surf}.

The theorem is proved. \endproof

\subsection{Noncommutative triangle groups and proof of Theorem \ref{th:delta 1-relator}}

\label{subsec:proof of theorem th:delta 1-relator}
We need the following immediate result.

\begin{lemma}
\label{le:n-gon cover of Sigma}
For any marked surface $\Sigma$ there is $n\ge 1$, a subset $S\subset  [n]$, an injective map $\sigma:S\to [n]\setminus S$, and a function $\varepsilon:S\to \{-,+\}$ such that   $\Sigma$ is obtained from  $P_n(h)$, $h=|I_s(\Sigma)|$ by gluing the chord $(i,i^+)$ to the chord
$\begin{cases}
 (\sigma(i)^+,\sigma(i))  & \text{if $\varepsilon(i)=+$}\\
 (\sigma(i),\sigma(i)^+)  & \text{if $\varepsilon(i)=-$}\\
\end{cases}
$ for all $i\in S$.

\end{lemma}

%

\begin{remark} Clearly, for any $n\ge 2$ and any pair $(\sigma,\varepsilon)$ as in Lemma \ref{le:n-gon cover of Sigma}, there is a marked surface $\Sigma_{\sigma,\varepsilon}$ obtained from $P_n(h)$ by such a gluing procedure.

\end{remark}

%

%

The following is an obvious version of Theorem \ref{th:functorial triangular}.

\begin{lemma}
\label{le:functorial triangular}
Let $f:\Sigma\to \Sigma'$ be as in Theorem \ref{th:functorial triangular} and let $\Delta$ and $\Delta'$ be triangulations of $\Sigma$ and $\Sigma'$ respectively such that $f(\Delta)\subset \Delta'$.
Then the assignments $t_{\gamma}\mapsto t_{f(\gamma)}$ for $\gamma\in \Delta$ define homomorphism of groups
$f_\star:\TT_{\Delta}\to \TT_{\Delta'}$.
\end{lemma}

Combining Lemmas \ref{le:n-gon cover of Sigma} and \ref{le:functorial triangular}  and taking into account that under the gluing  map $f:P_n(h)\to \Sigma$, the image $f(\Delta)$ of any triangulation  $\Delta$ of $P_n(h)$ is a triangulation of $\Sigma=\Sigma_{\sigma,\varepsilon}$,
we see that the quotient group of $\TT_{\Delta}$ by the relations
\begin{equation}
\label{eq:gluing relations}
t_{i,i^+}=\begin{cases}
 t_{\sigma(i)^+,\sigma(i)}  & \text{if $\varepsilon(i)=+$}\\
 t_{\sigma(i),\sigma(i)^+}  & \text{if $\varepsilon(i)=-$}\\
\end{cases}, t_{i^+,i}=\begin{cases}
 t_{\sigma(i),\sigma(i)^+}  & \text{if $\varepsilon(i)=+$}\\
 t_{\sigma(i)^+,\sigma(i)}  & \text{if $\varepsilon(i)=-$}\\
\end{cases},
\end{equation}
$i\in S$, is naturally isomorphic to $\TT_{f(\Delta)}$ (of course, $\TT_{f(\Delta)}\cong \TT_{\Delta''}$ for any triangulation $\Delta''$ of $\Sigma$ by Theorem \ref{th:TDeltaDelta' surface}).

We will use this observation with the appropriately modified starlike triangulation $\Delta=\tilde \Delta_1$ of $P_n(h)$, where $\Delta_1$ is the starlike triangulation of $[n]$ as in \eqref{eq:starlike} with $i=1$.

Namely, for all $n\ge 2$, $\tilde \Delta_1$ is obtained from $\Delta_1$  by adding $h$ curves $\gamma_{12}^{(s)}$, $s\in [h]$ from the vertex $1$ to the vertex $2$ {\it outside} of
$\Delta_1$ so that each $2$-gon $(({\gamma_{12}^{(s)}})^{-1},\gamma_{12}^{(s-1)})$, $s\in [h]$ contains exactly one special puncture (here, with a slight abuse of notation,
$\gamma_{12}^{(0)}$ is the chord $(1,2)$ in $[n]$) and a clockwise loop $\gamma_1^{(s)}$ around each special puncture inside
$((\gamma_{12}^{(s)})^{-1},\gamma_{12}^{(s-1)})$, $s\in [h]$.

\begin{lemma}
\label{le:starlike relation}
Suppose that $n\ge 2$. Then the  group $\TT_{\tilde \Delta_1}$ is generated by $t_j=T_j^{1,j^+}$, $j=3,\ldots,n-1$, $c_k=t_{k,k^+}$, $\overline c_k=t_{k^+,k}$, $k\in [n]$, $y_s=t_{\gamma_{12}^{(s)}}$, $z_s=t_{\gamma_1^{(s)}}$, $s\in [h]$, and $\overline y_h=t_{(\gamma_{12}^{(h)})^{-1}}$,
subject to (if $n\ge 4$):
\begin{equation}
\label{eq:1-relator T_n}
c_2t_3c_3\cdots t_{n-1}c_{n-1}\overline c_n^{\,-1}c_1= \overline c_1c_n^{-1}\overline c_{n-1} t_{n-1} \overline c_{n-2} \cdots  t_3\overline c_2
\end{equation}
and (if $h>0$):
\begin{equation}
\label{eq:1-relator y_h}
\overline y_h=\overline c_1 (z_1^{-1}y_1 c_1^{-1}z_1)(z_2^{-1}y_2 y_1^{-1}z_2)\cdots (z_h^{-1}y_h y_{h-1}^{-1}z_h)\ .
\end{equation}

\end{lemma}

\begin{proof} Clearly,
$t_{1,j}=c_1t_2c_2\cdots t_{j-1}c_{j-1},~t_{j,1}=\overline c_{j-1} t_{j-1}   \cdots  \overline c_2 t_2 \overline c_1$
in $\TT_{\Delta_1}$ for $j=1,\ldots,n$. Thus, $\TT_{\Delta_1}$  is generated by $t_2,\ldots,t_{n-1}$, $c_k$, $\overline c_k$, $k=1,\ldots,n$ subject to the relations:
$$\overline c_n=c_1t_2c_2\cdots c_{n-2}t_{n-1}c_{n-1}, c_n=\overline c_{n-1} t_{n-1}  \cdots  \overline c_2 t_2 \overline c_1\ .$$
By eliminating $t_2$, we see that $\TT_{\Delta_1}$ is subject to the relation \eqref{eq:1-relator T_n}. Furthermore,
the $1$-gon relations in the $1$-gons $(\gamma_1^{(s)})$ and triangle relations in the triangles
$((\gamma_{12}^{(s)})^{-1},\gamma_1^{(s)},\gamma_{12}^{(s-1)})$ for the remaining generators
$y_s=t_{\gamma_{12}^{(s)}}$, $\overline y_s=t_{(\gamma_{12}^{(s)})^{-1}}$ $z_s=t_{\gamma_1^{(s)}}$,
$\overline z_s=t_{(\gamma_1^{(s)})^{-1}}$, $s\in \{0\}\sqcup [h]$ of $\TT_{\tilde \Delta_1}$ read:
$$\overline z_s=z_s,~\overline y_s\overline z_s^{\,-1}y_{s-1}=\overline y_{s-1}z_s^{-1}y_s$$
for $s\in [h]$ (here  $y_0=c_1$, $\overline y_0=\overline c_1$). That is, one can eliminate all $\overline z_s$,
$s\in [h]$ and one can solve recursively for all $\overline y_s$, $s\in [h]$:
$$\overline y_s=\overline c_1 (z_1^{-1}y_1 y_0^{-1}z_1)(z_2^{-1}y_2 y_1^{-1}z_2)\cdots (z_s^{-1}y_s y_{s-1}^{-1}z_s)\ ,$$
so that the remaining generators $z_s$ and $y_s$, $s\in [h]$ are free.

The lemma is proved.
\end{proof}

Combining  Lemmas \ref{le:n-gon cover of Sigma} and \ref{le:starlike relation}, we see that
for $n\ge 3$ the group $\TT_{f(\Delta)}$ is generated by $t_j$, $j=3,\ldots,n-1$, $c_k$, $\overline c_k$, $k=1,\ldots,n$,  $y_s$, $z_s$, $s\in [h]$, and $\overline y_h$
subject to \eqref{eq:1-relator T_n} and the following relations for all $i\in S$:
$$c_{\sigma(i)}=\begin{cases}
\overline c_i' & \text{if $\varepsilon(i)=+$}\\
c_i' & \text{if $\varepsilon(i)=-$}\\
\end{cases},
\overline c_{\sigma(i)}=\begin{cases}
c'_i & \text{if $\varepsilon(i)=+$}\\
\overline c'_i & \text{if $\varepsilon(i)=-$}\\
\end{cases},\text{where}~ c'_i:=\begin{cases}
c_i & \text{if $i\ne 1$}\\
y_h & \text{if $i=1$}\\
\end{cases}, \overline c'_i:=\begin{cases}
\overline c_i & \text{if $i\ne 1$}\\
\overline y_h & \text{if $i=1$}\\
\end{cases}
 \ .$$

Thus, if $n\ge 3$, then the group  $\TT_{\tilde \Delta_1}$ has $(n-3)+2(n-|S|)+2h=3n-3-2|S|+2h$ generators $t_j$, $j=3,\ldots,n-1$, $c_k$, $\overline c_k$, $k\in [n]\setminus \sigma(S)$, $y_s$, $z_s$, $s\in [h]$
and exactly one relation \eqref{eq:1-relator T_n}. Now compute the Euler characteristic of $\Sigma$ using the triangulation $\Delta''$ of $\Sigma$ obtained by removing all $h$ loops around special punctures from $f(\Delta)$. By definition,
$$\chi(\Sigma)=|I|-E+T\ ,$$
where $E$ is the number of edges and $T$ is the number of triangles in $\Delta''$. Clearly, $T=n-2$ and $E=(n-3)+(n-|S|)$, therefore,
$$\chi(\Sigma)=|I|-((n-3)+(n-|S|))+n-2=|I|+1-n+|S|=|I|+1-\frac{n}{2}-\frac{|I_b|}{2}$$
because $n-2|S|=|I_b|$. Therefore, the number of generators of $\TT_{f(\Delta)}$ is:
$$3n-3-2|S|+2h=2n-3+|I_b|+2h=4(|I|+1-\chi(\Sigma)-\frac{|I_b|}{2})-3+|I_b|+2h=4(|I|-\chi(\Sigma))+1-|I_b|+2h\ .$$

We now consider several cases.

\noindent {\bf Case 1}.  $n\ge 3$ and either $\Sigma$ has boundary, i.e., $S\cup \sigma(S)\ne [n]$ or $h>0$.
The above implies that $\TT_{f(\Delta)}$ is free in $4(|I|-\chi(\Sigma))-|I_b|+2h$ generators.

\noindent {\bf Case 2}.  $n=2$ (hence $h>0$). Then, clearly, $\TT_{\tilde \Delta_1}$ is a free group in $2h+2$ generators. Therefore:

$\bullet$ If $n=2$, $h\ge 2$, then $\TT_{f(\Delta)}$ is free in $2h+2-2|S|$ generators, where $|S|\in \{0,1\}$.

$\bullet$ If $n=2$, $h=1$, then $\TT_{f(\Delta)}$ is free in $4-|S|$ generators, where $|S|\in \{0,1\}$.

\noindent {\bf Case 3}.  $n=1$, then $f$ is the identity map and $\Sigma$ is a disk with $|I|=|I_b|=1$ and $h$ special punctures.
If $h=0$, then, clearly, $\TT_\Delta$ is free in two generators $t_\gamma$ and $t_{\overline \gamma}$,
where $\gamma$ is the clockwise loop.
Suppose that $h>0$. Then one can choose a triangulation $\Delta$
of $\Sigma$ in such a way that, in addition to $\gamma$ it consists of a special loop $\lambda_s$,
$s\in [h]$ around each special puncture and a clockwise loop $\gamma_s$ enclosing first $s$ special punctures (from the left to the right),
$s=2,\ldots,h$ (so that $\lambda_1=\gamma_1$ and $\gamma_h=\gamma$). Then $\TT_\Delta$ is generated  by $z_s=t_{\lambda_s}$,
$y_s=t_{\gamma_s}$, $\overline y_s=t_{\overline \gamma_s}$, $s\in [h]$ subject to
the following triangle relations in the $h-1$ triangles
$(\gamma_{s-1},\lambda_s,\overline \gamma_s)$, $s=2,\ldots,h$:
$$
y_{s-1} z_s^{-1}\overline y_s=y_sz_s^{-1}\overline y_{s-1}$$
for $s=2,\ldots,h$ if $h\ge 2$. That is,
similarly to the equations \eqref{eq:1-relator y_h}, one can solve recursively for all $\overline y_s$, $s=2,\ldots,h$:
$$\overline y_s=(z_s y_{s-1} ^{-1}y_s) \cdots (z_3y_2^{-1}z_3 y_2)\cdot (z_2z_1^{-1}z_2 z_1)$$
(since $y_1=\overline y_1=z_1$)
so that $\TT_\Delta$ is freely generated by $z_s$, $s\in [h]$ and  $y_s$, $s=2,\ldots,h$.

This finishes the proof of Theorem \ref{th:delta 1-relator}(a).

\medskip

\noindent {\bf Case 4}. $n\ge 3$ and  $\Sigma$ has no boundary, i.e., $|I_b|=0$ hence $S\cup \sigma(S)=[n]$ and $h=0$. Then $n=2|S|$ is even and $\TT_{f(\Delta)}$ is a  $1$-relator torsion-free group in $4(|I|-\chi(\Sigma))+1$ generators $t_k$, $k=3,\ldots,n-1$, $c_k,\overline c_k$, $k\in S$.  Suppose that $\Sigma$ is a sphere with $|I|\le 3$ punctures. Then $\TT_{f(\Delta)}$ trivial for $|I|=1$ because all loops are contractible, is free in $2$ generators $t_{\gamma}$ and $t_{\gamma^{-1}}$ if $|I|=2$, where $\gamma$ is an arc between these two punctures, and if $|I|=3$, it is free in $5$ generators, because we can take $S=\{1,3\}\subset [4]$,  $\sigma(1)=2$, $\sigma(3)=4$, $\varepsilon(1)=\varepsilon(3)=+$  so that $\TT_{f(\Delta)}$ is freely generated by $c_1,\overline c_1,c_3,\overline c_3,t_1$.
Otherwise, it is, clearly, non-free.
This finishes the proof of Theorem \ref{th:delta 1-relator}(b).

The theorem is proved.
\endproof

\begin{example}
\label{ex:closed surfaces detaled} If $\Delta$ is a triangulation of the torus, the Klein bottle, the real projective plane respectively with one, one, two
(ordinary) punctures, then $\TT_\Delta$ is generated by $c_1,c_2,\overline c_1,\overline c_2,t_3$ subject to, respectively (with some notation from the proof of Theorem \ref{th:delta 1-relator} in Section \ref{subsec:proof of theorem th:delta 1-relator}):

(i) for the torus with one puncture: $c_2t_3\overline c_1c_2^{\,-1}c_1=\overline c_1\overline c_2^{\,-1} c_1t_3\overline c_2$, because $\Delta$ is glued from a square with diagonal $(1,3)$, where: $S=\{1,2\}\subset [4]$, $\sigma(1)=3$, $\sigma(2)=4$, $\varepsilon(1)=\varepsilon(2)=+$ (equivalently,
$a            b  c          d           e         =c                  b             e         d  a$
after substitution $a=t_3$, $b=\overline c_1$, $c=c_2^{-1}$, $d=c_1$, $e= \overline c_2^{\,-1}$).

(ii) for the Klein bottle with one puncture: $c_2t_3\overline c_1\overline c_2^{\,-1}c_1=\overline c_1c_2^{-1}c_1t_3\overline c_2$, because $\Delta$ is glued from a square with diagonal $(1,3)$, where: $S=\{1,2\}\subset [4]$, $\sigma(1)=3$, $\sigma(2)=4$,
$\varepsilon(1)=+$, $\varepsilon(2)=-$ (equivalently,
$a             b            c          d           c        =e                  b  e       d  a$
after substitution $a=t_3$, $b=\overline c_1$, $c=\overline c_2^{\,-1}$, $d=c_1$, $e= c_2^{-1}$).


(iii) for the real projective plane with two punctures: $c_2t_3c_1\overline c_2^{\,-1}c_1=\overline c_1c_2^{-1}\overline c_1t_3\overline c_2$, because $\Delta$ is glued from a square with diagonal $(1,3)$, where: $S=\{1,2\}\subset [4]$, $\sigma(1)=3$, $\sigma(2)=4$, $\varepsilon(1)=\varepsilon(2)=-$ (equivalently,
$a    b            c         b           c      =e                  d  e                 d  a$
after substitution $a=t_3$, $b=c_1$, $c=\overline c_2^{\,-1}$, $d=\overline  c_1$, $e= c_2^{-1}$).
\end{example}

\subsection{Noncommutative curves and proof of Theorem \ref{th:A_Delta sigma}}
\label{subsec:proof Laurent}


%
%

For each $\gamma\in [\Gamma(\Sigma)]$, a triangulation $\Delta$ of $\Sigma$
define the elements $t_{\gamma,\Delta}\in \QQ \TT_\Delta$ same way as in Theorem \ref{th:noncomlaurent n-gon}:
\begin{equation}
\label{eq:noncomm schiffler surf group TDelta}
t_\gamma^\Delta:=\sum_{\ii\in Adm_{\Delta^0}(1,j)} t_\ii \ ,
\end{equation}
where $\Delta^0$ is the  triangulation of $[n]$  assigned
(as in Theorem \ref{th:polygonal cover}(a))
to $\Delta$ and the canonical polygon $P_\Delta(\gamma)=(\gamma_1,\ldots,\gamma_n)$ with
$\gamma=\gamma_{1,j}$ and
we abbreviated
$$t_\ii:=t_{\gamma_{i_1,i_2}}t_{\gamma_{i_3,i_2}}^{-1}t_{\gamma_{i_3,i_4}}\cdots t_{\gamma_{i_{2m-1},i_{2m-2}}}^{-1}
t_{\gamma_{i_{2m-1},i_{2m}}}$$
for any sequence $\ii=(i_1,\ldots,i_{2m})\in [n]^{2m}$, $m\ge 1$.

%
%

%
We refer each $t_\gamma^\Delta$ as  it as a {\it noncommutative triangulated curve}.

Clearly, if  $\Sigma=P_n(0)$ is an $n$-gon (i.e., a disk with $I(\Sigma)=I_b(\Sigma)=[n]$) so that $\gamma=(p,q)\in [n]\times [n]$, then
$t_\gamma^\Delta=t_{pq}^\Delta$ is as in \eqref{eq:noncomm schiffler TDelta n-gon}.

To finish the proof of Theorem \ref{th:A_Delta sigma}, we need the following result.

\begin{proposition}
\label{pr:universal localization Sigma}
The
assignments $x_\gamma\mapsto t_\gamma^\Delta$ for $\gamma\in [\Gamma(\Sigma)]$ define an epimorphism of algebras
\begin{equation}
\label{eq:universal localization homomorphism}
\AA_\Sigma\to \QQ \TT_\Delta[{\bf S}_\Delta^{-1}] \ ,
\end{equation}
where ${\bf S}_\Delta$ is the sub-monoid of $\QQ \TT_\Delta$ generated by all $t_\gamma^\Delta$.
\end{proposition}

\begin{proof} It suffices to show that the elements $t_{\gamma}^\Delta$ satisfy the defining relations of $\AA_\Sigma$ from Definition \ref{def:Asigma}.

We need the following  result.

\begin{lemma} Let $Q=(\gamma'_1,\ldots,\gamma'_r)$ be an $n$-gon in $\Sigma$ and let $\Delta$ be any  triangulation of $\Sigma$. Then the assignments $x_{ij}\mapsto x_{\gamma_{ij}}$, $i,j\in [r]$ define a homomorphism of algebras
\begin{equation}
\label{eq:Ar to TDelta localized}
\AA_r\to \QQ \TT_\Delta[{\bf S}_\Delta^{-1}] \ .
\end{equation}

\end{lemma}

\begin{proof} Let $P=(\gamma_1,\ldots,\gamma_n)$,  $\Delta^0$, and $\iota:[r]\hookrightarrow [n]$ be as in Theorem \ref{th:polygonal cover}. Then, in view of Theorem, for any accompanying morphism $f:P_n\to \Sigma$ Therefore, the assignments $(i,j)\mapsto f(i,j)=\gamma_{ij}$
restricted to $\Delta^0$ define a homomorphism of algebras $f_\Delta:\QQ \TT_{\Delta^0}\to \QQ\TT_\Delta$ such that $f_\Delta(t_{ij}^{\Delta^0})=t_{\gamma_{ij}}^\Delta$ for $i,j\in [n]$.  Since $f_\Delta({\bf S}_{\Delta^0})\subset {\bf S_\Delta}$,
then passing to the universal localizations, this gives an algebra homomorphism
$$\QQ \TT_{{\Delta^0}}[{\bf S}_{\Delta^0}^{-1}]\to \QQ \TT_\Delta[{\bf S}_\Delta^{-1}] \ .$$
Composing it with the isomorphism  $\AA_n\cong \QQ \TT_{{\Delta^0}}[{\bf S}_{\Delta^0}^{-1}]$ given by Theorem \ref{th:A_Delta}(b) and the homomorphism $\AA_r\to \AA_n$ given by $x_{k,\ell}\mapsto x_{\iota(k),\iota(\ell)}$ give the
desired homomorphism \eqref{eq:Ar to TDelta localized}.
The lemma is proved.
\end{proof}

Using the Lemma with $r=3,4$, we finish the proof of the proposition.
\end{proof}

Since each $x_\gamma$, $\gamma\in [\Gamma(\Sigma)]$ is invertible in $\AA_\Sigma$, the universality of localization $\QQ \TT_\Delta[{\bf S}_\Delta^{-1}]$ implies that \eqref{eq:i_Delta surface} extends to a homomorphism of algebras
\begin{equation}
\label{eq:i_Delta surface localized}
\QQ \TT_\Delta[{\bf S}_\Delta^{-1}] \to  \AA_\Sigma \ .
\end{equation}

By the construction and Theorem \ref{th:noncomlaurent surface}, \eqref{eq:i_Delta surface} takes each $t_\gamma^\Delta$ to $x_\Delta$ and therefore is an epimorphism $\QQ[\TT_\Delta]\twoheadrightarrow \AA_\Delta$.
In turn, \eqref{eq:i_Delta surface localized} is an epimorphism as well.

Thus, we obtained two mutually inverse epimorphisms \eqref{eq:Ar to TDelta localized} and \eqref{eq:i_Delta surface localized}, which implies that they are isomorphisms of algebras.

Therefore, \eqref{eq:universal localization homomorphism} is an isomorphism, which proves Theorem \ref{th:A_Delta sigma}(b). Theorem \ref{th:A_Delta sigma}(a) also follows because \eqref{eq:i_Delta surface} is a restriction to $\QQ \TT_\Delta$ of the isomorphism \eqref{eq:i_Delta surface localized} and $\ii_\Delta(\QQ \TT_\Delta)= \AA_\Delta$.

Theorem \ref{th:A_Delta sigma} is proved. \endproof

%
%
%
%
%
%
%
%
%
%
%
%

\section{Noncommutative discrete integrable systems}
\label{sect:kontsevich}

\subsection{An integrable system on a cylinder}

Denote by $\Sigma_{1,r}$, an annulus (i.e., a cylinder) with no punctures, one marked point $p$ on the outer circle and $r$ marked points $p_1,\ldots,p_r$ on the inner circle (listed clockwise).

It is easy to see that equivalence classes of curves from $p$ to $\{p_1,\ldots,p_r\}$ in $\Sigma_{1,r}$ are in a natural bijection with $\ZZ$: the $n$-th curve $\gamma_n$ goes (without self-intersections) from $p$ to $p_s$ where $s\equiv n \mod r$ and $\gamma_n$ has the winding number $q$ such that $n=rq+s$ (so that the arc is winding clockwise if $q\ge 0$ and counterclockwise if $q<0$).

We also denote  $\gamma_i^-$ (resp. $\overline \gamma_i^{\,-}$)  the short counterclockwise boundary arc in the inner circle from $p_i$ to the previous point $p_{i^-}$
(resp. from  $p_{i^-}$ to  $p_i$), $i\in [r]$; and by $\gamma^+$ (resp. $\overline \gamma^{\,+}$) the clockwise (resp. counterclockwise) loop in the outer circle.

We abbreviate in the algebra ${\mathcal A}_{\Sigma_{1,r}}$:
$x_n:=x_{\gamma_n},~\overline x_n:=x_{\overline \gamma_n},~c_n:=x_{\gamma^-_n}, \overline c_n:=x_{\overline \gamma^{\,-}_n},~d:=x_{\gamma^+},~\overline d:=x_{\overline \gamma^{\,+}}$
for $n\in \ZZ$ (where we extend $\gamma^-_n$ periodically so that $\gamma^-_{n+r}=\gamma^-_r$ for all $n\in \ZZ$).

Since  $(\overline\gamma_{n-1},\gamma^+,\gamma_{n-r},\gamma_n^-)$ is a $4$-gon in  $\Gamma(\Sigma_{1,r})=[\Gamma(\Sigma_{1,r})]$ containing triangles $(\gamma_{n-1},\overline\gamma_{n-1}^{\,-},\overline\gamma_n)$ and  $(\overline \gamma_n,\gamma^+,\gamma_{n-r})$ for all $n\in \ZZ$, the following fact is immediate from Definition \ref{def:Asigma}.

\begin{lemma} For each $r\ge 1$ one has in ${\mathcal A}_{\Sigma_{1,r}}$:

(i) (triangle relations)
\begin{equation}
\label{eq:triangle relations sigma1r}
x_{n-1}\overline c_n^{\,-1}\overline x_n=x_n c_n^{-1} \overline x_{n-1},~\overline x_n\overline d^{\,-1} x_{n-r}=\overline x_{n-r} d^{-1} x_n \ .
\end{equation}

(ii) (exchange relations) For each $n\in \ZZ$:
\begin{equation}
\label{eq:exchange relations sigma1r}
\overline x_{n-r-1}d^{-1}x_n=c_n+ \overline x_{n-1}\overline d^{\,-1}x_{n-r},
~\overline x_n\overline d^{\,-1} x_{n-r-1}=\overline c_n+ \overline x_{n-r} d^{\,-1} x_{n-1} \ .
\end{equation}


\end{lemma}

%
%
%
Note that for each $m\in \ZZ$ the annulus $\Sigma_{1,r}$ has a triangulation
$$\Delta_m:=\{\gamma^+,\overline \gamma^{\,+};\gamma^-_1,\overline \gamma^-_1\ldots,\gamma^-_r,\overline \gamma^{\,-}_r;\gamma_m,\overline \gamma_m,\ldots,\gamma_{m+r},\overline \gamma_{m+r}\}\ .$$

\begin{picture}(10,100)
\put(180,15){\includegraphics[scale=0.4]{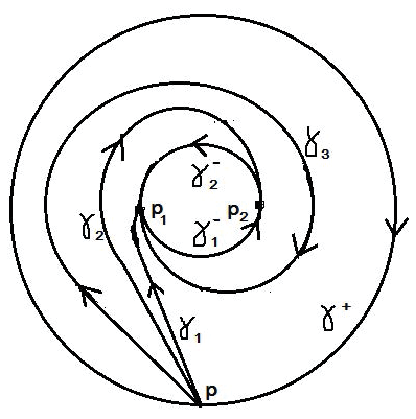}}
\put(130,5){Triangulation $\Delta_1$ of the cylinder $\Sigma_{1,2}$}
\end{picture}

Hence the group $\TT_r$ generated by $x_n,\overline x_n$,  $n=1,\ldots,r+1$,
$c_i,\overline c_i$, $i=1,\ldots,r$, $d,\overline d$  subject to the triangle relations
\begin{equation}
\label{eq:triangle relations sigma1r Delta0}
\overline x_{r+1}\overline d^{\,-1} x_1=\overline x_1 d^{-1} x_{r+1},~x_{s-1}\overline c_s^{\,-1}\overline x_s=x_s c_s^{-1} \overline x_{s-1},~
\end{equation}
$s=2,\ldots,r+1$ (with the convention $p_{r+n}=p_r$ hence $c_{r+n}=c_r$, $\overline c_{r+n}=\overline c_n$ for $n\in \ZZ$) is naturally isomorphic to the triangle group $\TT_{\Delta_1}$.
Moreover, in the notation of Section \ref{subsec:niversal localizations of noncommutative surfaces}, the subalgebra $\AA_{\Delta_1}$ of $\AA_{\Sigma}$
(generated by all $x_\gamma$, $\gamma\in \Gamma(\Sigma_{1,r})$ and all $x_{\gamma_0}^{-1}$, $\gamma_0\in \Delta_1$)
is the group algebra $\ZZ \TT_r$ by Theorem \ref{th:A_Delta sigma}(a).

\begin{proposition}
\label{pr:Sigma1r}
For each $r\ge 1$ we have:

(a)  Each $x_n,\overline x_n$, $n\in \ZZ$ is sum of elements of $\TT_r$ in $\ZZ \TT_r$.

(b) The total  angle $T_p\in \ZZ \TT_r$ at $p$
is given by
$T_p=\overline d^{-1}x_{n-r}x_n^{-1}+d^{-1}x_{n+r}x_n^{-1}=\overline x_n^{-1}\overline x_{n-r}d^{-1}+\overline x_n^{\,-1}\overline x_{n+r}\overline d^{\,-1}$

for each $n\in \ZZ$.
\end{proposition}

\begin{proof}Part (a) follows directly from Theorem \ref{th:noncomlaurent surface} and Corollary \ref{cor:injective TDelta surface}.

Prove (b). Consider a triangle  $(\gamma^+,\gamma_n,\overline \gamma_{n-r})$ in $\Delta_{n-r}$ and $(\gamma^+,\gamma_n,\overline \gamma_{n+r})$ in $\Delta_n$.

The following is an immediate corollary of Theorem \ref{th:total anglle surface}.

\begin{lemma} $T_p=T_{\gamma^+,\gamma_n,\overline \gamma_{n-r}}+T_{\gamma^+,\gamma_n,\overline \gamma_{n+r}}$.
\end{lemma}

Using this and taking into account that
$$T_{\gamma^+,\gamma_n,\overline \gamma_{n-r}}=\overline d^{-1}x_{n-r}x_n^{-1}=\overline x_n^{-1}\overline x_{n-r}d^{-1},~T_{\gamma^+,\gamma_n,\overline \gamma_{n+r}}=\overline d^{-1}x_{n+r}x_n^{-1}=\overline x_n^{-1}\overline x_{n+r}d^{-1}$$
in the notation \eqref{eq:general angle},
we obtain (b).

The proposition is proved.
\end{proof}
\begin{remark}
\label{re:Tp triangular}
Using the triangulation $\Delta_n$, it is easy see that
$$T_p=d^{-1}x_nx_{n-r}^{-1}+\overline d^{\,-1}x_{n-r} x_n^{-1}+\sum_{m=n+1-r}^n \overline x_{m-1}^{\,-1}c_mx_m^{-1}$$
for all $n\in \ZZ$.
\end{remark}

Clearly, by Theorem \ref{th:total anglle surface}, $T_p$ does not depend on $n$.

If $r$ is even, we can refine these observations and thus recover the recursion \eqref{eq:exchange relations sigma1k U}.

Indeed, set $U_n:=
\begin{cases} x_n & \text {if $n$ is even} \\
\overline x_n & \text {if $n$ is odd} \\
\end{cases},
C_n:=
\begin{cases} c_n & \text {if $n$ is even} \\
\overline c_n & \text {if $n$ is odd} \\
\end{cases},
$
and $D:=d^{-1}, \overline D:=\overline d^{\,-1}$.

By definition, $\TT_r$
is freely generated by $D$, $\overline D$ and  $C_i$, $i\in [r]$, $U_j$, $j\in [r+1]$ and, by  Proposition \ref{pr:Sigma1r}, $U_n\in \QQ \TT_r$ is a sum of elements of $\TT_r$. This and Proposition \ref{pr:Sigma1r} imply the following result.



\begin{theorem}
\label{th:discrete Kontsevich}
Let $r\ge 1$ be even. Then each element $U_n\in \ZZ\TT_r$, $n\in \ZZ$ satisfies  the recursion:
\begin{equation}
\label{eq:exchange relations sigma1r U}
\begin{cases} U_{n-r-1}DU_n=C_n+ U_{n-1}\overline D U_{n-r} & \text{if $n$ is even}\\
U_n\overline D U_{n-r-1}=C_n+ U_{n-r} D U_{n-1} &\text {if $n$ is odd} \\
\end{cases}.
\end{equation}
(with the convention $C_{n+r}=C_r$). Furthermore, the element $H_n\in Frac(\ZZ \TT_r)$, $n\in \ZZ$, given by
\begin{equation}
\label{eq:Hnprime}
H_n:= \begin{cases}
\overline D U_{n-r}U_n^{-1}+D U_{n+r}U_n^{-1}  & \text{if $n$ is even}\\
U_n^{-1} U_{n-r}D+ U_n^{-1}U_{n+r}\overline D & \text{if $n$ is odd}\\
\end{cases}
\end{equation}
does not depend on $n$ and belongs to $\ZZ \TT_r$.



\end{theorem}

The recursion \eqref{eq:exchange relations sigma1r U} clearly coincides with the recursion \eqref{eq:exchange relations sigma1k U} with $k=r+1$ and the element $H_n$ given by \eqref{eq:Hnprime} coincides with the element given by \eqref{eq:Hn}.

\begin{remark} In fact, Remark \ref{re:Tp triangular} implies that the ``conserved quantity" $H=H_n$ is equal (for any $n\in \ZZ$) to
$$\begin{cases}
DU_nU_{n-r}^{-1}+\overline D U_{n-r} U_n^{-1}+\sum\limits_{m=(n+2-r)/2}^{n/2} U_{2m-1}^{-1}C_{2m}U_{2m}^{-1}+ U_{2m-1}^{-1}C_{2m-1} U_{2m-2}^{-1}  & \text{if $n$ is even}\\
U_n^{-1}U_{n-r}D+U_{n-r}^{-1}U_n \overline D +\sum\limits_{m=(n+1-r)/2}^{(n-1)/2} U_{2m-1}^{\,-1}C_{2m} U_{2m}^{-1} + U_{2m+1}^{\,-1}C_{2m+1}U_{2m}^{-1} & \text{if $n$ is odd}\\
\end{cases}.
$$
\end{remark}

\subsection{An integrable system on an infinite strip}

In this section we establish Laurentness of another noncommutative recursion (which specializes to the discrete integrable system recently studied by P. Di Francesco in \cite{DF}). Indeed,
let $\Sigma_\infty$ be a horizontal strip with marked boundary points
$I=I_-\sqcup I_+$, where $I_+=\{i_+,i\in \ZZ\}$ (resp. $I_-=\{i_-,i\in \ZZ\}$) is the marked point set on the left (resp on the right) boundary line. Then, clearly,
$\Gamma(\Sigma_\infty)=[\Gamma(\Sigma_\infty)]=\{(i_\varepsilon,j_{\varepsilon'}):i,j\in \ZZ,~\varepsilon,\varepsilon'\in \{-,+\},i\ne j~\text{if $\varepsilon=\varepsilon'$}\}$.
Clearly,
$\Sigma_\infty=\bigcup\limits_{m^-,m^+\in \ZZ,n\in \ZZ_{>0}} \Sigma_{m^-,m^+}^n$, 
where $\Sigma_{m^-,m^+}^n\subset \Sigma$ is the convex hull of the real intervals $[m^-+1,m^-+n]_-\subset I_-$,
$[m^++1,m^++n]_+\subset I_+$. Clearly, it is an $2n$-gon embedded (as a parallelogram) into $\Sigma$, where we identify its vertex set $[2n]$ with $\{(m^-+1)_-\ldots,(m^-+n)_-\}\sqcup \{(m^++1)_+\ldots,(m^++n)_+\}$ via
$$k\mapsto
\begin{cases}(m^-+k)_- & \text{if $k\le n$}\\
(m^++2n+1-k)_+ & \text{if $k> n$}\\
\end{cases}.$$

We denote by $\AA_{\Sigma_{m^-,m^+}^n}$ a copy of $\AA_{2n}$ under the above identification of
the vertex set $[2n]$.

Then the natural inclusions $\Sigma_{m^-,m^+}^n\subset \Sigma_{{m'}^-,{m'}^+}^{n'}$ for ${m'}^-\le m^-$, ${m'}^+\le m^+$, ${m'}^-+n'\ge m^-+n$, ${m'}^++n'\ge m^++n$ are  morphisms in {\bf Surf} so they define (by Theorem \ref{th:functorial triangular}) the appropriate homomorphisms of algebras $\AA_{\Sigma_{m^-,m^+}^n}\to \AA_{\Sigma_{{m'}^-,{m'}^+}^{n'}}$, so we denote by $\AA_{\Sigma}$ the direct limit $\buildrel \longrightarrow \over {\lim} \AA_{\Sigma_{m^-,m^+}^n}$ under these homomorphisms.

Clearly, the following  noncommutative Ptolemy relations (in the form \eqref{eq:exchange relations upper}) hold in $\AA_{\Sigma_\infty}$:
\begin{equation}
\label{eq:difrancesco1prim}
x_{(i+1)_\pm,j_\mp}x_{(j+1)_\mp,j_\mp}^{-1} x_{(j+1)_\mp,i_\pm}=x_{(i+1)_\pm,i_\mp}+x_{(i+1)_\pm,(j+1)_\mp} x_{j_\mp,(j+1)_\mp}^{-1} x_{j_\mp,i_\pm}
\end{equation}
together with the triangle relations:
\begin{equation}
\label{eq:difrancesco2prim}
x_{i_\pm,j_\mp}x_{(j+1)_\mp,j_\mp}^{-1} x_{(j+1)_\mp,i_\pm}= x_{i_\pm,(j+1)_\mp}x_{j_\mp,(j+1)_\mp}^{-1} x_{j_\mp,i_\pm}
\end{equation}
for all $i,j\in \ZZ$.



\begin{remark} It is natural to conjecture that the relations \eqref{eq:difrancesco1prim} and  \eqref{eq:difrancesco2prim} are defining for $\AA_{\Sigma_\infty}$ and (in view of Remark \ref{re:injective A to A'} that) all natural homomorphisms $\AA_{\Sigma_{m^-,m^+}^n}\hookrightarrow \AA_{\Sigma_\infty}$ are injective.

\end{remark}

Note that $\Sigma_\infty$ has a triangulation
$$\Delta_\infty=\{(i_{\pm},(i+1)_\pm),((i+1)_\pm,i_\pm);(i_-,i_+),(i_+,i_-), (i_-,(i+1)_+),((i+1)_+,i_-):i\in \ZZ\}\ .$$

\begin{picture}(10,85)
\put(100,15){\includegraphics[scale=0.5]{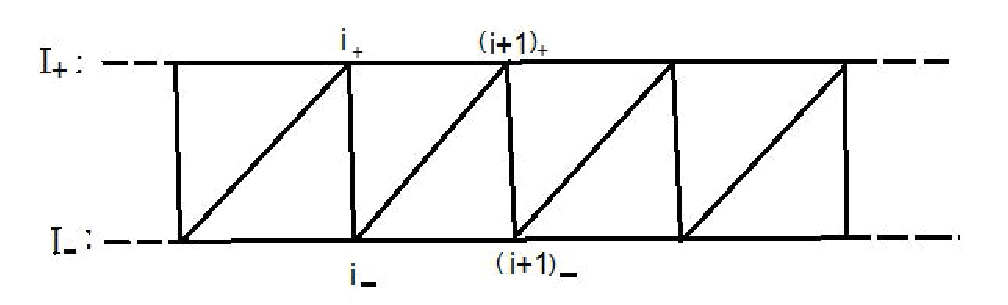}}
\put(165,5){Triangulation $\Delta_\infty$ of  $\Sigma_\infty$}
\end{picture}

Hence the group $\TT_\infty$ generated by $d_{i,\pm}:=x_{i_{\pm},(i+1)_\pm}$, $\overline d_{i,\pm}:=x_{i_{\pm},(i+1)_\pm}$,
$x_i:=x_{i_-,i_+}$, $\overline x_i:=x_{i_+,i_-}$, $y_i=x_{i_-,(i+1)_+}$, $\overline y_i=x_{i_-,(i+1)_+}$, $i\in \ZZ$   subject to the triangle relations
\begin{equation}
\label{eq:triangle relations sigmainfty Deltainfty}
x_i\overline d_{i,+}^{\,-1} \overline y_i=y_i d_{i,+}^{-1} \overline x_i,~\overline y_i\overline d_{i,-}^{\,-1} x_{i+1}=\overline x_{i+1}d_{i,-}^{\,-1}y_i
\end{equation}
for $i\in \ZZ$ is naturally isomorphic to the triangle group $\TT_{\Delta_\infty}$.
Corollary \ref{cor:injective TDelta surface} implies that
the subalgebra $\AA_{\Delta_\infty}$ of $\AA_{\Sigma_\infty}$
(generated by all $x_\gamma$, $\gamma\in \Gamma(\Sigma_\infty)$ and all $x_{\gamma_0}^{-1}$, $\gamma_0\in \Delta_\infty$)
is the group algebra $\ZZ \TT_\infty$.

\begin{proposition}
\label{pr:Sigmainfty}
In $\AA_{\Sigma_\infty}$ we have:

(a)  Each $x_{i_\pm, j_\mp}$, $i,j\in \ZZ$ is sum of elements of $\TT_\infty$ in $\ZZ \TT_\infty$.

(b) The total  angle $T_{i_\pm}\in \ZZ \TT_r$ at $i_\pm$
is given by
$$T_{i_\pm}=x_{j_\mp,i_\pm}^{-1}(x_{j_\mp,(i-1)_\pm}x_{i_\pm,(i-1)_\pm}^{-1}+x_{j_\mp,(i+1)_\pm}x_{i_\pm,(i+1)_\pm}^{-1})=(x_{(i-1)_\pm,i_\pm}^{-1}  x_{(i-1)_\pm,j_\mp}+x_{(i+1)_\pm,i_\pm}^{-1}  x_{(i-1)_\pm,j_\mp})x_{i_\pm,j_\mp}^{-1}$$


for each $j\in \ZZ$.
\end{proposition}

\begin{proof} Part (a) follows directly from Theorem \ref{th:noncomlaurent surface}.

Prove (b). Consider  triangles in the vertices $(i_\pm,j_\mp,(i-1)_\pm)$ and $(i_\pm,j_\mp,(i+1)_\pm)$ in $\Sigma_\infty$.

The following is an immediate corollary of Theorem \ref{th:total anglle surface}.

\begin{lemma} $T_{i_\pm}=T_{(i_\pm,j_\mp),(j_\mp,(i-1)_\pm),((i-1)_\pm,i_\pm)}+T_{(i_\pm,j_\mp),(j_\mp,(i+1)_\pm),((i+1)_\pm,i_\pm)}$.
\end{lemma}

Using this and taking into account that
$$T_{(i_\pm,j_\mp),(j_\mp,(i-1)_\pm),((i-1)_\pm,i_\pm)}=x_{j_\mp,i_\pm}^{-1}x_{j_\mp,(i-1)_\pm}x_{i_\pm,(i-1)_\pm}^{-1}=
x_{(i-1)_\pm,i_\pm}^{-1}  x_{(i-1)_\pm,j_\mp}x_{i_\pm,j_\mp}^{-1}\ ,$$
$$T_{(i_\pm,j_\mp),(j_\mp,(i+1)_\pm),((i+1)_\pm,i_\pm)}=x_{j_\mp,i_\pm}^{-1}x_{j_\mp,(i+1)_\pm}x_{i_\pm,(i+1)_\pm}^{-1}=
x_{(i+1)_\pm,i_\pm}^{-1}  x_{(i-1)_\pm,j_\mp}x_{i_\pm,j_\mp}^{-1}$$
in the notation \eqref{eq:general angle},
we obtain (b).

The proposition is proved.
\end{proof}

\begin{remark}
\label{re:Tinfty triangular}
Using the triangulation $\Delta_\infty$, it is easy see that
$$T_{i_-}=d_{i-1,-}^{-1}y_{i-1}x_i^{-1}+\overline x_i^{\,-1}d_{i,+}y_i^{-1}+ \overline d_{i+1,-}^{\,-1}x_{i+1}y_i^{-1},
~T_{i_+}=y_i^{-1}x_{i-1}\overline d_{i-1,+}^{-1}+y_i^{-1}d_{i-1}x_i^{-1}+x_i^{-1}y_id_{i,+}^{-1} \ .$$

\end{remark}


We can refine these observations and thus recover the recursions \eqref{eq:difrancesco1}, \eqref{eq:difrancesco2}.
Indeed, set
$$U_{ij}=x_{i_-,j_+}, ~V_{ij}:=x_{i_+,j_-},
~A_j:=x_{(j+1)_+,j_+}^{-1},~\overline A_j= x_{j_+,(j+1)_+}^{-1},~B_j:=x_{(j+1)_-,j_-}^{-1},~\overline B_j= x_{j_-,(j+1)_-}^{-1}\ .$$

By definition, $\TT_\infty$
is freely generated by $A_i,\overline A_i$, $B_i,\overline B_i$, $U_{i,i}$, $V_{i,i}$,  $U_{i,i+1}$, $i\in \ZZ$  and,
by  Proposition \ref{pr:Sigmainfty}, each $U_{ij}^\pm\in \QQ \TT_\infty$ is a sum of elements of $\TT_r$.
This and Proposition \ref{pr:Sigmainfty} imply the following result.

\begin{theorem}
\label{th:infinite strip}
The elements $U_{ij},V_{ij}\in \ZZ\TT_\infty$  $i,j\in \ZZ$ satisfy  \eqref{eq:difrancesco1}, \eqref{eq:difrancesco2}.
Furthermore, the elements $H_{ij}^\pm\in Frac(\ZZ \TT_\infty)$, $i\in \ZZ$, given by
\eqref{eq:Hnprimeinfty}
do not depend on $j$ and belong to $\ZZ \TT_\infty$.



\end{theorem}

\section{Appendix: Noncommutative localizations}
\label{sect:appendix}

Recall that for a multiplicative monoid $S$
its {\it linearization} $\ZZ S$  is the ring  $\ZZ S=\bigoplus_{s\in S} \ZZ \cdot [s]$ with the natural extension of multiplication on $S$.

Given a multiplicative submonoid $S$ of a unital ring $R$, define the {\it universal localization} $R[S^{-1}]$ of $R$ by $S$ to be quotient of the free product $R*(\ZZ S^{op})$ by the ideal generated by all elements of the form $s*[s]-1$, $[s]*s-1$ for any $s\in S$.

By definition, one has a canonical ring homomorphism
\begin{equation}
\label{eq:canonical embedding localization}
R\to R[S^{-1}]\ .
\end{equation}

In other words, $R[S^{-1}]$ is the unital ring $R'$ with the universal property that one has a ring homomorphism $R\to R'$ under which the image of each element of $S$ in invertible.

Note that \eqref{eq:canonical embedding localization} in not always injective. For each unital ring $R$ denote by $R^\times$ the set of all units (i.e., invertible elements) in $R$.

The following fact is obvious.

\begin{lemma}
\label{le:universal property of localization}
For any ring homomorphism $\varphi: R\to R'$ and any submonoid $S\subset R\setminus\{0\}$ such that $\varphi(S)\subset (R')^\times$ there is a unique ring homomorphism $\varphi_S:R[S^{-1}]\to R'$ such that the composition $R\to R[S^{-1}]\to R'$ is $\varphi$.

\end{lemma}

For each submonoid $S\subset R\setminus \{0\}$ define its {\it saturation} $\hat S$ to be the set of all $r\in R$ such that the image of $r$ in $R[S^{-1}]$ is invertible. Clearly, $\hat S$ is a submonoid of $R\setminus \{0\}$ containing $S$. We say that $S$ is saturated if $\hat S=S$. The following obvious fact justifies this definition.

\begin{lemma} For any submonoid $S\subset R\setminus \{0\}$ one has
$R[S^{-1}]=R[\hat S^{\,-1}]$.
Moreover, $\hat S$ is the largest submonoid of $R\setminus \{0\}$ with this property.

\end{lemma}


Following Malcev and Cohn, we say that a unital ring is of class ${\mathcal E}$  if it can be embedded into a skew-field.


\begin{lemma}
\label{le:free localization general}
Let $R$ be any ring of class ${\mathcal E}$. Then for any multiplicative submonoid $S$ of $R\setminus \{0\}$ the canonical homomorphism \eqref{eq:canonical embedding localization} is injective.

\end{lemma}

\begin{proof} Indeed, let $\FF$ be a skew field and $\varphi:R\to \FF$ be a monomorphism. By definition, for any submonoid $S$ of $R\setminus \{0\}$, $\varphi$ factors as $\varphi=g\circ f$, where $f:R\to R[S^{-1}]$ and $g:R[S^{-1}]\to \FF$ are canonical homomorphisms. Since $\varphi$ is a monomorphism, then $f$ is also a monomorphism.
\end{proof}

\begin{definition}
\label{def:divisible submonoid}
For a ring $R$ of class ${\mathcal E}$ we say that a submonoid $S$ of $R\setminus \{0\}$ is {\it divisible} if $R[S^{-1}]$ is also of class ${\mathcal E}$.
\end{definition}

Following Cohn, we say that a submonoid $S$ of $R\setminus \{0\}$ is {\it factor-closed} if for any $a,b\in R\setminus\{0\}$, $ab\in S$ implies that $a,b\in S$.

\begin{proposition} Let $R$ be of class ${\mathcal E}$ and $S$ be a divisible submonoid of $R\setminus \{0\}$. Then the saturation $\hat S$ of $S$ is a factor-closed submonoid of $R\setminus \{0\}$.
\end{proposition}

\begin{proof} Since $S$ is divisible, in particular, the canonical homomorphism $R\to R'=R[S^{-1}]=R[\hat S^{-1}]$ is injective.   It suffices to prove that
if  $x,y\in R$ such that $xy\in \hat S$, then $x\in \hat S$, $y\in \hat S$.
Indeed, let $z:=(xy)^{-1}$ and $t:=yzx-1$ in $R'$. By definition, $xyz=1=zxy$. This implies that $xt=xyzx-x=1\cdot x-x=0$. Since $R'$ has no zero divisors and $x\ne 0$, then $t=0$, i.e., $(yz)x=1$. Since $x(yz)=1$, we see that $x$ is invertible in $R'$ hence $x\in \hat S$. Similarly, $y\in \hat S$ as well.

The proposition is proved.
\end{proof}

Below we provide a sufficient criterion for a group algebra of a group to belong to class ${\mathcal E}$ and for divisibility of some of its submonoids.

\begin{definition}
\label{def:1-relator}
A group $G$ is called {\it $1$-relator torsion-free} if $G$ is isomorphic to $F/\langle x\rangle$ where $F$ is a finitely generated free group, $x\in F\setminus \{1\}$ is not a proper power in $F$, and $\langle x\rangle$ denotes the normal subgroup of $F$ generated by $x$.
\end{definition}

Results of  Malcev, Newman, J.~Lewin and T.~Lewin  (see e.g., \cite[Section 8.7]{Cohn}, \cite{lewin}) imply the following.

\begin{theorem}
\label{th:invertible almost free} Let  $G$ be any finitely generated free group or any $1$-relator torsion free group. Then the group algebra $R=\QQ G$ is of class ${\mathcal E}$. In particular, for any submonoid $S\subset \QQ G\setminus \{0\}$ the canonical homomorphism \eqref{eq:canonical embedding localization}
is injective.

\end{theorem}

We will need the following result, which is a particular case of \cite[Theorem 10.11]{Schofield} (here $\FF_\ell$ denotes a free skew field freely generated by $\ell$ elements).

\begin{proposition}  \label{pr:Schofield}
Let $\ell\ge 1$ and assume that $\ell$ elements $t_1,\ldots,t_\ell$ of $\FF_\ell$ generate $\FF_\ell$. Then $t_1,\ldots,t_\ell$ are free generators.
In particular, the assignments $c_i\mapsto t_i$ for $i=1,\ldots,\ell$ define an injective homomorphism of algebras
$\QQ F_\ell\hookrightarrow \FF_\ell$.
\end{proposition}


%
%
%

Following Cohn, we say that a ring $R$ is  a left (resp. right) {\it semifir} if each finitely generated left (resp. right) ideal $J$ is isomorphic to $R^n$ for a unique $n=n_J$.  $R$ is called a semifir if it is both left and right semifir.  We use below the standard definition of a universal $R$-field, see \cite[Section 7.2]{Cohn}.

\begin{theorem}
\label{th:semifir}
Let $R$ be a semifir. Then:

(a) There exists a universal skew field $Frac(R)$ containing  $R$ as a subalgebra and generated by $R$.

(b) For  any factor-closed submonoid $S$ of $R\setminus \{0\}$ the canonical homomorphism $R_S\to Frac(R)$ is injective.

\end{theorem}

\begin{proof} Recall from \cite{Cohn} that:

$\bullet$  an $n\times n$ matrix $A$ over a unital ring $R$ is {\it full} if for any factorization $A=BC$ for some  $n\times p$ matrix $B$ and  a $p\times n$ matrix $C$ one has  $p\ge n$;

$\bullet$ A homomorphism $f:R\to R'$ is {\it honest} if the image of each full matrix is full.

$\bullet$ A set $\Sigma$ of square matrices over a unital ring $R$ is  {\it multiplicative} if any upper block-triangular matrix with diagonal in $\Sigma$ also belongs to $\Sigma$ and $\Sigma$ is closed under simultaneous permutation of rows and columns;

 $\bullet$ A set $\Sigma$ of  matrices over a unital ring $R$ is called
{\it factor-closed} if $AB\in \Sigma$ for some $n\times n$ matrices $A$ and  $B$ over $R$ implies that $A,B\in \Sigma$.

$\bullet$ For any set $\Sigma$ of square matrices over a unital ring $R$, $R_\Sigma$ denotes the {\it universal localization} (\cite[Theorem 2.1]{Cohn}) so that the image of each element of $\Sigma$ under the canonical homomorphism $R\to R_\Sigma$ is an invertible matrix (e.g., $R_S=R[S^{-1}]$ in the notation as above);

Then Theorem \ref{th:semifir}(a) immediately follows from the following result.

\begin{theorem} \cite[Section 7.5, Corollary 5.11]{Cohn})
\label{th:cohn honest}
For each semifir $R$ the universal localization $Frac(R):=R_\Phi$, where $\Phi$ is the set of full matrices over $R$, is a skew field and the canonical homomorphism $R\to Frac(R)$ is honest (hence injective).

\end{theorem}

To prove (b) we need following results from \cite{Cohn}.

\begin{proposition} (\cite[Section 7.5, Proposition 5.7(ii)]{Cohn})
\label{pr:cohn honest2}
Given unital rings $R$ and $R'$ and a honest homomorphism $f:R\rightarrow R'$, then for any  factor-closed multiplicative
set  $\Sigma$ of square matrices over $R$, the canonical homomorphism
$f_\Sigma:R_{\Sigma}\to R'$ is injective.

\end{proposition}

For any $S\subset R$ denote by $\Sigma_S$ the set
of all matrices over $R$ of the form $PMQ$ where $P$ and $Q$ are invertible matrices
over $R$ and $M$ is an upper triangular matrix over $R$ with diagonal entries in $S$.
\begin{lemma}  (\cite[Section 7.5, Lemma 10.1]{Cohn})
\label{le:cohn closed}
Let $R$ be a semifir. Then for any  factor-closed submonoid $S$ of $R\setminus \{0\}$ the set $\Sigma_S$ is factor-closed and multiplicative.
\end{lemma}

Finally, letting $R$ be a semifir and $R'=Frac(R)$ in Proposition \ref{pr:cohn honest2},  $\Sigma=\Sigma_S$ as in Lemma \ref{le:cohn closed} and taking into account that $R[S^{-1}]=R_S=R_{\Sigma_S}$, we finish the proof of part (b).

Theorem \ref{th:semifir} is proved.
\end{proof}

It is well-known (see e.g., \cite{DM}) that for any finitely generated free group $F$ its group algebra  $R=\QQ F$ is a semifir. Therefore, Theorem \ref{th:semifir} implies the following corollary.

\begin{corollary}
\label{cor:free divisible}
Let $F$ be a finitely generated free group and $R=\QQ F$. Then any factor-closed submonoid $S$ of $R\setminus\{0\}$ is divisible, more precisely, $R[S^{-1}]\subset Frac(R)$.

\end{corollary}

\begin{remark}
\label{rem:free group algebra UFD}
Based on Theorem \ref{th:invertible almost free}, we expect that an analogue of Corollary \ref{cor:free divisible} also holds for $R=\QQ G$, where $G$ is a torsion-free $1$-relator group.

\end{remark}

Given a unital ring $R$, following Cohn, we say that:

$\bullet$  Elements $a,b\in R$ are {\it similar} if the right $R$-modules $R/aR$ and $R/bR$ are isomorphic (clearly, similarity is an equivalence relation on $R$).

$\bullet$ An element $p\in R\setminus R^\times$ is {\it prime} if for any factorization $p=p'p''$ one has: either $p'\in R^\times$ or $p''\in R^\times$.

$\bullet$  A unital ring $R$ is  a (noncommutative)  {\it unique factorization domain (UFD)} if each nonzero non-unit admits a prime factorization and
for any two prime factorizations of a non-unit $x\in R$:
$$x=p_1\cdots p_r=q_1\cdots q_s$$
one has $s=r$ and $q_i$ is similar to $p_{\sigma(i)}$ for $i=1,\ldots,r$  where $\sigma$ is a permutation of $\{1,\ldots,r\}$.


\begin{proposition}
\label{pr:UFD} Let $R$ be a UFD and $S$ be a submonoid of $R\setminus \{0\}$. Then $S$ is factor-closed iff
it is generated by $R^\times$ and a (empty or not) set $P$ which is the union of similarity classes of prime elements in $R$.



\end{proposition}

\begin{proof}
Denote by $P$ the set of all primes in $S$ and by $S_P$ the submonoid of $R\setminus \{0\}$ generated by $R^\times$ and $P$. Clearly, $S_P\subset S$.

Suppose that $S$ is factor-closed.  Let us show that $S=S_P$. We proceed by contradiction, i.e., suppose that there is at least one element $x\in S\setminus S_P$.
Then $x$ is not a unit hence $x$ has a prime factorization $x=p_1\cdots p_r$. If $r=1$, then
$x=p_1\in S$ hence $x\in S_P$ and we arrive at the contradiction. If $r\ge 2$, then since $S$ is factor-closed,  we have $p_i\in S$ for $i=1,\ldots,r$. Hence
$x\in S_P$ and we arrive at the contradiction once again.

Suppose that $P$ is a union of similarity classes and $S=S_P$. Let us prove that $S$ is factor-closed. Suppose that $ab\in S$ for some $a,b\in R$. Let us show that $a,b\in S$.
If either $a$ or $b$ is a unit, we have nothing to prove because $R^\times \subset S$. Thus, suppose that $a,b\in R\setminus R^\times$ and
let
$$a=p_1\cdots p_{r'},~b=p_{r'+1}\cdots p_r$$
be respective prime factorizations with  $1\le r'<r$, where $p_1,\ldots,p_r$ are some primes in $R$. On the other hand,
since $ab$ is a non-unit element of $S$, it admits a prime factorization 
$ab=q_1\cdots q_s$ in $S$, 
where $q_1,\ldots,q_s\in P$.
Comparing the factorizations $p_1\cdots p_r=q_1\cdots q_s$
and using the fact that $R$ is UFD, we obtain: $r=s$ and each $p_i$ is similar to one of $q_j$. Since all primes similar to each $q_j$ belongs to $P$, we obtain $p_1,\ldots,p_r\in P$ hence $a\in S$, $b\in S$.

The proposition is proved.
\end{proof}

\begin{remark}
\label{rem:free UFD}
The class of noncommutative UFD's is rather large: it contains  group rings $\QQ F$, where $F$ is any finitely generated free group (see e.g.,  \cite[Theorem 3.4, Proposition 3.5 and Corollary]{Cohn2}).
\end{remark}

Note however, that similarity classes of primes may contain some ``unexpected" elements. For instance, if $R$ is the free ring in $x,y$ then $xy+1$ and $yx+1$ are similar (see e.g. \cite{Cohn2}). This motivates the following definition.

\begin{definition} Given a ring $R$, we say that an element $a\in R\setminus \{0\}$ is {\it self-similar} if all elements similar to $a$ are of the form $uau'$, where $u,u'\in R^\times$.
\end{definition}

Taking into account that $(\QQ F)^\times=\QQ^\times\cdot F$ for a free (or, more generally, an ordered) group $F$  (see e.g., \cite[Theorem 6.29]{Lam}), we obtain the following conjectural characterization of certain self-similar primes in  $\QQ F$.

\begin{conjecture}
\label{conj:rigid linear}
Let $F$ be a free group freely generated by $t_1,\ldots,t_m$, $m\ge 2$. Then for $k=2,\ldots,m$ the element $\tau_k:=t_1+\ldots+t_k$ is a self-similar prime, e.g.,
all elements of $\QQ F$ similar to $\tau_k$ belong to $\QQ^\times\cdot F\cdot \tau_k\cdot F$.

\end{conjecture}

\begin{remark} This conjecture was shaped during our discussions with George Bergman, Dolors Herbera, and Alexander Lichtman. We are immensely grateful to these mathematicians.

\end{remark}

%





\end{document}